\RequirePackage{filecontents}
\documentclass[10pt]{article}
\usepackage{amsmath, amsthm, amssymb}
\usepackage{float,epsfig}
\usepackage{color}
\usepackage{multirow}
\usepackage{subfigure}

\begin{document}

\newtheorem{theorem}{\bf Theorem}[section]
\newtheorem{proposition}[theorem]{\bf Proposition}
\newtheorem{definition}[theorem]{\bf Definition}
\newtheorem{corollary}[theorem]{\bf Corollary}
\newtheorem{exam}[theorem]{\bf Example}
\newtheorem{remark}[theorem]{\bf Remark}
\newtheorem{lemma}[theorem]{\bf Lemma}
\newcommand{\nrm}[1]{|\!|\!| {#1} |\!|\!|}

\newcommand{\ba}{\begin{array}}
\newcommand{\ea}{\end{array}}
\newcommand{\von}{\vskip 1ex}
\newcommand{\vone}{\vskip 2ex}
\newcommand{\vtwo}{\vskip 4ex}
\newcommand{\dm}[1]{ {\displaystyle{#1} } }

\newcommand{\be}{\begin{equation}}
\newcommand{\ee}{\end{equation}}
\newcommand{\beano}{\begin{eqnarray*}}
\newcommand{\eeano}{\end{eqnarray*}}
\newcommand{\inp}[2]{\langle {#1} ,\,{#2} \rangle}
\def\bmatrix#1{\left[ \begin{matrix} #1 \end{matrix} \right]}



\def \R{{\mathbb R}}
\def \C{{\mathbb C}}
\def \M{{\mathcal M}}
\def \S{{\mathcal S}}
\def \pf{{\bf Proof: }}
\def \pds{{\mathrm {PDS}}}
\def \C{{\mathcal C}}
\def \T{{\mathcal T}}


\title{ On the eigenvalue region of permutative doubly stochastic matrices}
\author{ Amrita Mandal,\thanks{Department of Mathematics, Indian Institute of Technology Kharagpur, Kharagpur 721302, India. Email:AMRITAMANDAL@iitkgp.ac.in,  mandalamrita55@gmail.com}  \,\, Bibhas Adhikari,\thanks{Department of Mathematics, Indian Institute of Technology Kharagpur, Kharagpur 721302, India. Email: bibhas@maths.iitkgp.ac.in}  \,\, M. Rajesh Kannan\thanks{Department of Mathematics, Indian Institute of Technology Kharagpur, Kharagpur 721302, India. Email: rajeshkannan@maths.iitkgp.ac.in, rajeshkannan1.m@gmail.com} }
\date{}

\maketitle
\thispagestyle{empty}

{\small \noindent{\bf Abstract.}
This paper is devoted to the study of eigenvalue region of the doubly stochastic matrices which are also permutative, that is, each row of such a matrix is a permutation of any other row. We call these matrices as permutative doubly stochastic (PDS) matrices. A method is proposed to obtain symbolic representation of all PDS matrices of order $n$ by finding equivalence classes of permutationally similar symbolic PDS matrices. This is a hard problem in general as it boils down to finding all Latin squares of order $n.$ However, explicit symbolic representation of matrices in these classes are determined in this paper when $n=2, 3, 4.$  It is shown that eigenvalue regions are same for doubly stochastic matrices and PDS matrices when $n=2, 3.$ It is also established that this is no longer true for $n=4,$ and two line segments are determined which belong to the eigenvalue region of doubly stochastic matrices but not in the eigenvalue region of PDS matrices. Thus a conjecture is developed for the boundary of the eigenvalue region of PDS matrices of order $4.$ Finally, inclusion theorems for eigenvalue region of PDS matrices are proved when $n\geq 2.$   



	\vone \noindent{\bf Keywords.} Doubly stochastic matrix, Eigenvalue region,  Latin square,  Permutative matrix
	
	\vone\noindent{\bf AMS subject classification(2000):} 15A18, 15A29, 15B51.
		
	\section{Introduction}\label{sec:1}
	Characterization of the eigenvalue region of doubly stochastic (DS) matrices is one of the long-standing open problems in matrix theory. This problem relates to many other problems, including inverse eigenvalue problem and eigenvalue realizability problem. Despite several attempts, the eigenvalue region of DS matrices of order up to $4$  is known till date. However different conjectures and observations are reported for the eigenvalue region of higher-order DS matrices  \cite{Levick2015, Perfect1965,harlev}. In this paper, we consider permutative doubly stochastic (PDS) matrices  and investigate its eigenvalue region. Recall that a permutative matrix is a matrix whose rows are permutations of its first row. Hence a PDS matrix is a matrix which is both permutative and DS. Eigenvalue properties of permutative matrices is a topic of interest in itself due to its applications in solving eigenvalue realizability problems which include Suleimanova spectrum \cite{Loewy2016,Manzaneda2017,Paparella2016,Paparella2019, Soto2017,Andrade2018}.
	
	Let $\omega_n$ and $\omega_n^0$ denote the eigenvalue region of $\Omega_n$ and $\Omega_n^0$ respectively, where $\Omega_n^0$ denotes the set of all DS matrices of order $n$ with trace zero. In \cite{Perfect1965}, Perfect and Mirsky stated that $$\bigcup_{j=2}^n \Pi_j \subseteq \omega_n~~\mathrm{and}~~\Pi_n^0 \subseteq \omega_n^0,$$ where $\Pi_j$ is the convex hull of the $j^{th}$ roots of unity, and $\Pi_j^0$ is the union of $1$ and convex hull of the $j^{th}$ roots of unity except the point $1$. They also proved that $\omega_2=\Pi_2,$ $\omega_3=\Pi_2 \cup \Pi_3$ and $\omega_3^0=\Pi_3^0,$ and  conjectured that $$\omega_n =\bigcup_{j=2}^n \Pi_j,$$ for all $n.$ Mashreghi and Rivard established that this conjecture is not true by constructing an example of a DS matrix of order $n=5,$ some of whose eigenvalues lie outside the proposed region \cite{Mashreghi2007}. Recently Levick et al. proved that the conjecture is true for $n=4$ \cite{Levick2015}. 
Recently Harlev et al., discussed the experimental evidence, for the validity of the Perfect Mirsky conjecture for $n \leq 11$ except $ n = 5$  \cite{harlev}}. In \cite{Benvenuti2018}, Benvenuti proved that $\omega_4^0=\Pi_4^0 \cup [-1,1].$
	
	In this paper, we prove that the eigenvalue regions of PDS matrices of order $2$ and $3$ are the same as that of DS matrices of order $2$ and $3$, respectively. For PDS matrices of order $4$, we use a semi-analytical approach to determine its eigenvalue region. Let $\pds_n$ denote the set of all PDS matrices of order $n>1.$ We propose to exploit the structural pattern of entries of the PDS matrices of order $n$ to divide the set $\pds_n$ into different classes.
	  A class, which we call cogredient class (defined in Section 2), consists of PDS matrices which are permutationally similar to each other. Thus, all the matrices in the same class have the same eigenvalue region. Finally, the eigenvalue region of $\pds_n$ can be obtained by taking the union of the eigenvalue regions of these classes. However, for large values of $n$, the number of such classes is extremely large. In this paper, we restrict our attention to determine all cogredient classes of $\pds_n$ explicitly, for $n \leq 4.$ We show that there are $37$ cogredient classes in $\pds_4$. The analytical expression of eigenvalues of matrices in all the classes is hard to find. However we derive explicit expression of eigenvalues of all such classes whose eigenvalue region is a subset of the real line. Explicit expression of non-real eigenvalues of all but six classes are also provided. 
 Numerical computations are performed in order to plot the eigenvalue regions of all the classes.
	
	We show that the eigenvalue region of $\Omega_4$ is not equal to the eigenvalue region of $\pds_4$. Indeed, we show that any complex number on the line segments $-\frac{1}{2} \pm yi,$ where $\frac{1}{2}<y<\frac{\sqrt3}{2}$ and $r \pm \frac{1-r}{\sqrt 3} i,$ where $-\frac{1}{2} < r < \frac{1-\sqrt3}{1+\sqrt3}$ cannot be an eigenvalue of any PDS matrix of order $4$, whereas we construct  DS matrices of order $4$ having eigenvalues from this collection. We conjecture that non-real  eigenvalues of two parametric $\pds$ matrices should yield the boundary of the eigenvalue region of $\pds_4,$ that lie inside $\Pi_3$ but not in $\Pi_4.$  The explicit expression of those matrices are also given. Finally, we identify a subset in the unit circle of the complex plane which is contained in the eigenvalue region of PDS matrices of order $n \geq 2$.

We believe that the framework developed in this paper can be extended to determine the cogredient classes and the eigenvalue region of $\pds_n$ for $n>4$ but it could boil down to finding all Latin squares of order $n.$ However it can be accomplished for $n\leq 11$ for which the number of all Latin squares is known. This we leave as a future work.\\
	

	{\bf Notation:} Unless stated otherwise, a vector is always considered as a row vector. For any vector $c\in \R^n$, we denote $\sum c$ as the sum of all the entries of $c.$ The symmetric group on $n$ elements is denoted as $S_n,$ and the group of permutation matrices of order $n$ is denoted as $\S_n.$ For brevity, we denote a permutation and its corresponding permutation matrix using the same notation interchangeably. If $P \in S_n$ is a product of $r$ disjoint cycles of lengths $k_1, k_2,\ldots, k_r$, then we say $P$ is a permutation of cycle type $(k_1) + (k_2) + \ldots +(k_r).$ By $I_n$ we denote the identity matrix of order $n.$ We denote $J$ as the all-one matrix, the dimension of which will be clear from the context.
Given a set of matrices $S,$ we denote $\Lambda(S)$ as the eigenvalue region of the matrices in $S.$ Besides, for a given matrix $X$ we denote $\Lambda(X)$ as the set of eigenvalues of $X.$


\section{Classification of PDS matrices}

In this section we propose to partition $\pds_n$ into equivalence classes, which we call cogredient classes, of symbolic $\pds$ matrices such that any two matrices in a class are similar thorough a permutation matrix. Thus it turns out that eigenvalue region of $\pds_n$ is union of eigenvalue region of the cogredient classes. We determine explicit symbolic representation of matrices in all the cogredient classes for $\pds_n$ and use it to determine its eigenvalue region for $n=2,3,4.$ Due to the symbolic representation of the matrices in a class, the entries of such a matrix follow a pattern. We elaborate more on this after defining permutative doubly stochastic matrix as follows. For any permutation $\pi \in S_n$, we define the corresponding permutation matrix $P=(P_{ij})\in \S_n$ such that $P_{ij}=1$ if $j=\pi(i),$ and $0$ otherwise.


\begin{definition}\label{def:permu}(Permutative matrix)
 A matrix $A$ of order $n\times n$ is said to be a permutative matrix if it is of the form \begin{equation}\label{permutivem} A=\bmatrix{c \\ c P_1 \\ c P_2 \\ \vdots \\ c P_{n-1} },\end{equation} where $P_j$ is a permutation matrix of order $n,$ $j=1, \hdots, n-1,$ and $c$ is a row vector of dimension $n.$\end{definition}

In \cite{Paparella2016}, permutative matrix is introduced in an alternative way by considering $c$ as a column vector. However we adapt the Definition \ref{def:permu} in this paper.

\begin{definition}\label{def:pds}(Permutative doubly stochastic matrix)
A permutative matrix $A$ of the form (\ref{permutivem}) is said to be a permutative doubly stochastic matrix if it is also a doubly stochastic matrix, that is, $c$ is nonnegative with $\sum c=1$ and $\sum Ae_j=1$ where $\{e_j : j=1, \hdots, n\}$ is the standard basis (column vectors) of $\R^n.$
\end{definition}

For an $(n-1)$-tuple $(P_1,P_2,\hdots,P_{n-1})$ of permutation matrices of order $n,$  a class of $\mathrm{PDS}$ matrices can be obtained as $$\M(c;P_1,\hdots,P_{n-1}) = \left\{\bmatrix{c \\ c P_1 \\ c P_2 \\ \vdots \\ c P_{n-1} }\in\Omega_n : c=\bmatrix{c_1 & \hdots & c_n}\right\}.$$  Treating $c$ to be a vector of symbols $c_i$ the matrices in $\M(c;P_1,\hdots,P_{n-1})$ follow a pattern with respect to the position of its entries, which is determined by the permutation matrices $P_j, 1\leq j\leq n-1.$
 As an example, $\pds$ matrices of order $3$ in $\M(c; (132), (123))$   can be represented as
\begin{equation}\label{sec2:exp1}\bmatrix{c_1 & c_2 & c_3 \\ c_2 & c_3 & c_1 \\ c_3 & c_1 & c_2}.\end{equation}

From now onward we call the row vector $c$ as the generic vector which defines the class of matrices $\M(c;P_1,\hdots,P_{n-1})$ for a given tuple $(P_1, \hdots, P_{n-1}),$ and we call the corresponding PDS matrix as the generic matrix which represents $\M(c;P_1,\hdots,P_{n-1}).$  Then the set of all PDS matrices of order $n\times n$ can be written as $$\pds_n = \bigcup_{P_1,P_2,\hdots,P_{n-1}\in\S_n}\M(c; P_1, \hdots, P_{n-1}),$$ where $c$ is a generic row vector of order $n.$

Note that the matrix $$\bmatrix{c_1 & c_3 & c_2 \\ c_3 & c_2 & c_1 \\ c_2 & c_1 & c_3}$$ has the similar pattern of entries in it as that of (\ref{sec2:exp1}). Indeed for a generic vector $c=\bmatrix{c_1 & \hdots & c_n}$ and a tuple $(P_1,P_2,\hdots,P_{n-1})$ of permutation matrices of order $n$, the PDS matrices in $\M(c; P_1, \hdots, P_{n-1})$ and $\M(cP; P_1, \hdots, P_{n-1})$ have similar pattern of entries for any permutation matrix $P$ of order $n.$ Thus, we introduce this class of PDS matrices which follow a similar pattern of its entries as \begin{equation}\label{def:class1}\left[\M(c; P_1, \hdots, P_{n-1})\right]=\{A\in\M(cP; P_1, \hdots, P_{n-1}) : P\in\S_n\}\end{equation} for the generic vector $c$ and the $(n-1)$-tuple $(P_1,P_2,\hdots,P_{n-1}).$ However, it may so happen that for the generic matrix $A$ which represents $\M(c; Q_1, \hdots, Q_{n-1}),$ there may exist a permutation matrix $X$ of order $n$ such that $X^TAX\in \left[\M(c; P_1, \hdots, P_{n-1})\right]$ even if $P_j\neq Q_j$ for some $j\in\{1,\hdots, n-1\}.$ For example, the generic matrix $A$ which represents $\M(c;(1243),(1342),(14)(23))$ given by $$ A= \bmatrix{c_1 & c_2 & c_3 & c_4 \\ c_3 & c_1 & c_4 & c_2 \\ c_2 & c_4 & c_1 & c_3\\ c_4 & c_3 & c_2 & c_1},$$ and the generic matrix $$B=\bmatrix{c_1 & c_4 & c_3 & c_2 \\ c_4 & c_1 & c_2 & c_3\\ c_2 & c_3 & c_1 & c_4 \\ c_3 & c_2 & c_4 & c_1} \in \left[\M(c;(12)(34),(1324),(1423))\right]$$ are permutationally similar via  the permutation matrix $X=(1)(24)(3).$ Thus, the generic matrix $A$ which represents $\left[\M(c;(1243),(1342),(14)(23))\right]$ is permutationally similar to the generic matrix $B$ which represents $\left[\M(c;(12)(34),(1324),(1423))\right],$ and hence eigenvalue region of these two classes of matrices are the same. It is then important to characterize all such classes of matrices that are similar to each other through a permutation matrix. We introduce the following definition.

\begin{definition}(Cogredient classes of PDS matrices)\label{def:cm}
Let $c=\bmatrix{c_1 & \hdots & c_n}$ be a generic vector. Let $(P_1,P_2,\hdots,P_{n-1})$ and $(Q_1,Q_2,\hdots,Q_{n-1})$ be two $(n-1)$-tuples of permutation matrices of order $n.$ Then $\left[\M(c; P_1, \hdots, P_{n-1})\right]$ and $\left[\M(c; Q_1, \hdots, Q_{n-1})\right]$ are said to be cogredient if for the generic matrices $A$ and $B$ representing $\left[\M(c; P_1, \hdots, P_{n-1})\right]$ and $\left[\M(c; Q_1, \hdots, Q_{n-1})\right]$ respectively, there exist a permutation matrix $X$ such that $X^TAX=B.$
\end{definition}

Let us denote the cogredient relation by $\rho$ which is defined on the set $\{\left[\M(c; P_1, \hdots, P_{n-1})\right] : (P_1, \hdots, P_{n-1})\in \S_n\times \hdots \times\S_n\}$ for the generic vector $c.$ Then it is immediate that $\rho$ is an equivalence relation. We call an equivalence class defined by $\rho$ as a cogredient class and we denote it by the bold-faced notation ${\bf{\left[\M(c; P_1, \hdots, P_{n-1})\right]}}$ for some $(P_1, \hdots, P_{n-1})\in \S_n\times \hdots \times\S_n.$ It is important to notice that the eigenvalue region of $\pds_n$ must be the union of eigenvalue regions of its cogredient classes of matrices. It is then pertinent to determine all the cogredient classes of $\pds_n$ for any given $n.$ The following observation provides a constructive method to determine all the cogredient classes of $\pds_n.$

Let $A$ be a generic matrix which represents the set of matrices $\M(c; P_1, \hdots, P_{n-1})$ for some permutation matrices $P_j, j=1,\hdots, n-1$. Then for any pair of permutation matrices $P$ and $ Q$ the matrix $B=PAQ\in \left[\M(c; Q_1, \hdots, Q_{n-1})\right]$ for some $Q_j, j=1, \hdots, n-1,$ the set of such matrices is called the isotopic class defined by $A$ when $A$ is a Latin square \cite{Latin}. For example, let $\hat{A}\in\pds_4$ be a generic matrix, such that $1$st and $2$nd column of $\hat{A}$ have all entries equal to $\frac{1}{4}.$ Then $P\hat{A}Q$ for $(P,Q)\in \S_4 \times \S_4,$ are the matrices having two columns with all entries equal to $\frac{1}{4}$ and the rest columns preserve the same pattern of the entries as that of $3$rd and $4$th columns of $\hat{A}.$ Note that such a pair $A,B$ are called permutation equivalent \cite{Minc1988}. However $B=Q^T(QPA)Q=Q^T(XA)Q$ where $X=QP$ is a permutation matrix. Thus $B$ is permutationally similar to the matrix $XA.$ For any $X\in \S_n,$ $XA$ is a generic matrix which either belongs to $\bf{\left[\M(c; P_1, \hdots, P_{n-1})\right]}$ or does not belong to $\bf{\left[\M(c; P_1, \hdots, P_{n-1})\right]}.$ If  $XA\in \bf{\left[\M(c; P_1, \hdots, P_{n-1})\right]},$ then the class $\left[\M(c; Q_1, \hdots, Q_{n-1})\right] \in \bf{\left[\M(c; P_1, \hdots, P_{n-1})\right]},$ otherwise we generate a new cogredient class $\bf{\left[\M(c; Q_1, \hdots, Q_{n-1})\right]}.$

However consideration of possible $(n-1)$-tuples of permutation matrices to determine all generic matrices which do not arise by following the above procedure could be a daunting task. This can be overcome by constructing matrices which are not permutation equivalent to $A$, and then following the same procedure on such a matrix as mentioned above.
This results in identifying a collection of generic matrices in $\pds_n$ having a specific structure and that are not permutation equivalent to each other. Then for all such matrices $A,$ the matrices $XA, X\in\S_n$ determine all the cogredient classes of $\pds_n$. Indeed, in the next subsection we will see that the selection of $A$ in the beginning can be done from especially structured generic matrices based on repetitions of entries of the generic vector $c$ in the columns of such a matrix.

Here we emphasize that the above classification of $\pds_n$ is the same for the set of permutative square matrices of order $n$ since there is no explicit use of doubly stochastic property of a PDS matrix. Although it can be observed from the next section that the number of cogredient classes will be very large if the DS property is dropped. We articulate the above formulation as to characterize $\pds_n$  so that the number of terminologies to be used in the entire paper is reduced. Below we derive all cogredient classes of $\pds_n$ when $n\leq 4.$ When $n\geq 5,$ the above procedure should be implementated by devising an algorithm which is beyond the scope of this paper.

\subsection{PDS matrices of order $2$ and order $3$}\label{sec:2}

For $n=2,$ it is needless to say that any matrix $A\in\pds_2$ if and only if $A$ is doubly stochastic. Thus $$A=\bmatrix{c_1 & c_2 \\ c_2 & c_1}$$ where $c_1 +c_2=1.$


Let $n=3.$ Then, it is easy to check that the sets $\left[\M(c; (132), (123))\right]$ and $\left[\M(c;(123),(132))\right]$ comprise of all matrices such that no entry of $c$ is repeated in any column symbolically. Indeed, the matrices in $\M(c; (132), (123))$ and $\M(c;(123),(132))$  can be represented as
$$\bmatrix{c_1 & c_2 & c_3 \\ c_2 & c_3 & c_1 \\ c_3 & c_1 & c_2} \,\, \mbox{and} \,\, \bmatrix{c_1 & c_2 & c_3 \\ c_3 & c_1 & c_2 \\ c_2 & c_3 & c_1},$$ respectively. For any other choices of the pair $(P_1, P_2),$ at least one entry (symbol) is repeated in a column, and hence two entries of $c$ must be equal or all entries must be $\frac{1}{3}$ (for example, if an entry is same in the entire column) due to the constraint that the sum of entries in each column is $1.$ The following theorem characterizes $\pds_3$.

\begin{theorem} \label{thm1}
$\pds_3=\bf{\left[\M(c; (132), (123))\right]} \cup \bf{\left[\M(c;(123),(132))\right]}.$
\end{theorem}
\noindent\pf Let $A\in \M(c;P_1,P_2)$ where $P_1= (132),$ $P_2= (123)$ and $c=\bmatrix{c_1 & c_2 & c_3}$ is the generic vector. Then it is clear that $PAQ \in [\M((c;(123),(132))],$ where $P=I_3$ and $Q=(23).$ Hence to find all the cogredient classes in $\pds_3$ it is enough to consider the set of matrices $\{XA : X\in \S_3\}.$
Then $XA$ belongs to $ [\M((c;(123),(132))]$ for $X$ of cycle type $(1)+(2)$ and $XA$ is in $ [\M((c;(132),(123))]$ for $X$ of cycle type $(3)$ or $(1)+(1)+(1)$ i.e., $X=I_3.$ It is to be noted here that $\bf{\left[\M(c; (132), (123))\right]}$=$\left[\M(c; (132), (123))\right]$ and $\bf{\left[\M(c;(123),(132))\right]}$=$\left[\M(c;(123),(132))\right].$ Thus the desired result follows.
 $\hfill{\square}$




\subsection{PDS matrices of order $4$}\label{sec:3}

Recall that $\pds_4=\cup_{P,Q,R\in\S_4}[\M(c;P,Q,R)]$ where $c=\bmatrix{c_1 & c_2 & c_3 & c_4},$ $\sum c=1$ (see Definition \ref{def:pds}). Obviously, classification and description of cogredient classes of $\pds_4$ would be cumbersome than the earlier case of $\pds_3$ since the number of possible choices of permutations $P,Q,R\in\S_4$ to define $\M(c;P,Q,R)$ is huge (i.e $24 \times 24 \times 24$). We divide $\pds_4$ broadly into two subsets of generic matrices. First we consider generic matrices in which every column has no repetition of its entries symbolically.

Then for a generic vector $c=\bmatrix{c_1 & c_2 & c_3 & c_4}$ there are $24$ choices of the triplet of permutations $(P,Q,R)$ for which any generic matrix in $[\M(c;P,Q,R)]$ has all the entries of $c$ in every column of the matrix. This follows from the fact that there are $24$ semi reduced Latin squares of order $4$ \cite{Latin}. However such classes could be cogredient (see Definition \ref{def:cm}). The following theorem provides such cogredient classes of $\pds_4.$


\begin{table}[h!]
  \begin{center}
            \begin{tabular}{|c|c|c|c|}
    \hline
     \multirow{2}{*}{\textbf{Class}} & \multirow{2}{*}{$(P,Q,R)$} & \multirow{2}{*}{$(P,Q,R)$} &  \multirow{2}{*}{$X$}\\
       &  &  &  \\
         \hline
      \hline
     \multirow{2}{*}{$\C_1$} & \multirow{2}{*}{$((12)(34),(1324),(1423))$} & $((1243),(1342),(14)(23))$ & $(1)(24)(3)$ \\
      &  & $((1234),(13)(24),(1432))$ & $(1)(23)(4)$\\
      \hline \hline
       \multirow{2}{*}{$\C_2$} & \multirow{2}{*}{$((12)(34),(14)(23),(13)(24))$} & $((13)(24),(12)(34),(14)(23))$ & $(1)(24)(3)$ \\
      &  & $((14)(23),(13)(24),(12)(34))$ & $(1)(23)(4)$\\
      \hline \hline
\multirow{11}{*}{$\C_3$} & \multirow{11}{*}{$((1432),(1234),(13)(24))$} & $((1342),(14)(23),(1243))$ & $(1)(2)(34)$ \\
      &  & $((1243),(14)(23),(1342))$ & $(1324)$\\
 &  & $((13)(24),(1432),(1234))$ & $(1)(24)(3)$\\
 &  & $((13)(24),(1234),(1432))$ & $(1234)$\\
 &  & $((1423),(12)(34),(1324))$ & $(1)(243)$\\
 &  & $((1423),(1324),(12)(34))$ & $(1243)$\\
 &  & $((1234),(1432),(13)(24))$ & $(13)(24)$\\
 &  & $((1324),(12)(34),(1423))$ & $(123)(4)$\\
 &  & $((14)(23),(1243),(1342))$ & $(1)(234)$\\
 &  & $((1324),(1423),(12)(34))$ & $(1)(23)(4)$\\
 &  & $((14)(23),(1342),(1243))$ & $(124)(3)$\\
      \hline \hline
$\C_4$ & $((13)(24),(14)(23),(12)(34))$  & $((14)(23),(12)(34),(13)(24))$ & $(1)(23)(4)$\\
      \hline \hline
$\C_5$ & $((12)(34),(13)(24),(14)(23))$  & - & - \\
      \hline \hline
 \multirow{2}{*}{$\C_6$} & \multirow{2}{*}{$((12)(34),(1423),(1324))$} & $((1432),(13)(24),(1234))$ & $(1)(23)(4)$ \\
      &  & $((1342),(1243),(14)(23))$ & $(1)(24)(3)$\\
\hline
          \end{tabular}
\caption{Classes $\C_i, i=1,\hdots,6$ of matrices in $\pds_4$. Each class $\C_i$ is the union of classes of matrices $[\M(c;P,Q,R)]$ where the triplets $(P,Q,R)$ are described in the second and third columns of the $i$th row of the table. Further, the  generic matrices defined by $(P,Q,R)$ in second column and third column are permutationally similar through the permutation matrix in the last column(see Definition \ref{def:cm}).} \label{table:cc}.
  \end{center}
\end{table}

\begin{theorem}\label{thm:norepeat}
Let $c$ be a generic row vector. Then any matrix $A\in [\M(c;P,Q,R)],$ $P,Q,R\in \S_4$ such that no entry in any column of $A$ is repeated belongs to $\bigcup_i \C_i$ where $\C_i, i=1,\hdots,6$ are given in Table \ref{table:cc}.
\end{theorem}
\noindent\pf
We consider the matrices $\M(c;(12)(34),(1423),(1324)),$
$\M(c;(12)(34),(13)(24),(14)(23))$,\\$\M(c;(1432),(13)(24),(1234))$ and $\M(c;(1342),(1243),(14)(23))$ correspond to $4$ reduced Latin squares of order $4$ given by
$$A_1=\bmatrix
{c_1&c_2&c_3&c_4\\
c_2&c_1&c_4&c_3\\
c_3&c_4&c_2&c_1\\
c_4&c_3&c_1&c_2},
A_2=\bmatrix
{c_1&c_2&c_3&c_4\\
c_2&c_1&c_4&c_3\\
c_3&c_4&c_1&c_2\\
c_4&c_3&c_2&c_1},
A_3=\bmatrix
{c_1&c_2&c_3&c_4\\
c_2&c_3&c_4&c_1\\
c_3&c_4&c_1&c_2\\
c_4&c_1&c_2&c_3},
A_4=\bmatrix
{c_1&c_2&c_3&c_4\\
c_2&c_4&c_1&c_3\\
c_3&c_1&c_4&c_2\\
c_4&c_3&c_2&c_1}$$ respectively. It can be verified that any generic matrix, without repetition of entries of $c$ in any of its column, can be obtained as $XA$ where $A \in \{A_1,A_2,A_3,A_4\}$ and $X \in \S_4$ such that $X_{11}=1.$ Hence one obvious conclusion is that the triplets $(P,Q,R),$ where $P,Q,R$ are chosen from one of the sets from the following set, all the $\pds_4$ matrices mentioned in Table \ref{table:cc} can be obtained. $$\{(12)(34),(1324),(1423)\}, \{(14)(23),(1243),(1342)\}, \{(13)(24),(1234),(1432)\},\{(12)(34),(14)(23),(13)(24)\}.$$

Below we show that $A_1$ and $A_2$ do not belong to the isotopic classes of each other, however, $A_3, A_4$ belong to one of those isotopic classes. For any $X\in \S_4,$ $XA_1$ belongs to one of the classes $\C_1,\C_3,\C_6,$ and $XA_2$ belongs to either of the classes $\C_2,\C_4,\C_5.$

 We prove the cogredient relation of classes $[\M(c;P,Q,R)]$ defined by triplets $(P,Q,R)$ in second and third columns for the class $\C_1.$

The generic matrices corresponding to the set of matrices $\M(c;(12)(34),(1324),(1423)),$ $\M(c;(1243),(1342),(14)(23))$ and $\M(c;(1234),(13)(24),(1432))$ are given by \begin{equation}\label{Thm4.1:eq1} A= \bmatrix{c_1 & c_2 & c_3 & c_4 \\ c_2 & c_1 & c_4 & c_3\\ c_4 & c_3 & c_1 & c_2 \\ c_3 & c_4 & c_2 & c_1}, B=\bmatrix{c_1 & c_2 & c_3 & c_4 \\ c_3 & c_1 & c_4 & c_2 \\ c_2 & c_4 & c_1 & c_3\\ c_4 & c_3 & c_2 & c_1} \, \mbox{and} \,\, C= \bmatrix{c_1 & c_2 & c_3 & c_4 \\ c_4 & c_1 & c_2 & c_3 \\ c_3 & c_4 & c_1 & c_2 \\ c_2 & c_3 & c_4 & c_1}\end{equation} respectively. Then observe that
$$ X^TBX=\bmatrix{c_1 & c_4 & c_3 & c_2 \\ c_4 & c_1 & c_2 & c_3\\ c_2 & c_3 & c_1 & c_4 \\ c_3 & c_2 & c_4 & c_1}\in \left[\M(c;(12)(34),(1324),(1423))\right]$$
for $c=\bmatrix{c_1 & c_4 & c_3 & c_2}$ and hence by Definition \ref{def:cm}, $\left[\M(c;(12)(34),(1324),(1423))\right]$ is cogredient to $\left[\M(c;(1243),(1342),(14)(23))\right].$ Further $$X^TCX=\bmatrix{c_1 & c_3 & c_2 & c_4 \\ c_3 & c_1 & c_4 & c_2\\ c_4 & c_2 & c_1 & c_3 \\ c_2 & c_4 & c_3 & c_1}\in \left[\M(c;(12)(34),(1324),(1423))\right]$$ for $c=\bmatrix{c_1 & c_3 & c_2 & c_4}.$
Similarly the other cogredient classes can be obtained. $\hfill{\square}$



Next we consider generic matrices that contain columns in which at least one entry of $c$ is repeated. Then for the generic vector $c=\bmatrix{c_1 & c_2 & c_3 & c_4}$ three types of matrices can exist. Let us denote $\T_i$ to be the type of generic matrices in which every column is allowed repetitions of at least one entry of $c$ precisely $i$ times but not $i+1$ times for $i=2,3,4.$ For a fixed $i$ there can be several distinct cogredient classes of PDS matrices. For example, if $i=4$ then there can be at least two cogredient classes of matrices: an entry is repeated in only one column of the matrix or two entries are repeated in two columns of the matrix. Obviously, four entries of $c$ can be repeated in $4$ columns of the matrix, and the matrix is $\frac{1}{4}J$ but it is not a non-trivial class since this matrix belongs to $\M(c;P,Q,R)$ for every $(P,Q,R)$ by setting all the entries of $c$ equal.


The following theorem provides all cogredient classes of $\pds_4$ such that any matrix which belongs to a class has at least one entry of $c$ repeated in at least one of the columns of the matrix.

\begin{table}[h!]
  \begin{center}
            \begin{tabular}{|c|c|c|}
    \hline
     \textbf{Type} & \textbf{Class} & $(P,Q,R)$  \\
        \hline
      \hline
  \multirow{8}{*}{$\T_4$} & $\C_7$ & $((1)(2)(3)(4),(1)(2)(34),(1)(2)(34))$ \\
 & $\C_8$ & $((1)(2)(34),(1)(2)(34),(1)(2)(3)(4))$ \\
\cline{2-3}
          & $\C_9$ & $((1)(2)(34),(1)(243),(1)(24)(3))$ \\
 & $\C_{10}$ & $((1)(2)(34),(1)(24)(3),(1)(243))$\\
 &  $\C_{11}$ & $((1)(24)(3),(1)(2)(34),(1)(243))$  \\
 & $\C_{12}$ & $((1)(243),(1)(2)(34),(1)(24)(3))$ \\
 & $\C_{13}$ & $((1)(24)(3),(1)(243),(1)(2)(34))$  \\
 & $\C_{14}$ & $((1)(243),(1)(24)(3),(1)(2)(34))$\\
        \hline  \hline
  \multirow{12}{*}{$\T_3$} & $\C_{15}$ & $((1)(2)(34),(1)(24)(3),(142)(3))$ \\
 & $\C_{16}$ & $((1)(2)(34),(1)(243),(1432))$ \\
 & $\C_{17}$ & $((1)(234),(1)(24)(3),(14)(2)(3))$ \\
 & $\C_{18}$ & $((1234),(14)(2)(3),(124)(3))$ \\
 & $\C_{19}$ & $((1)(24)(3),(1)(2)(34),(142)(3))$ \\
 & $\C_{20}$ & $((1)(2)(34),(142)(3),(1)(24)(3))$ \\
 & $\C_{21}$ & $((1)(2)(34),(1432),(1)(243))$ \\
 & $\C_{22}$ & $((14)(2)(3),(1234),(124)(3))$ \\
 & $\C_{23}$ & $((1)(243),(1)(2)(34),(1432))$ \\
 & $\C_{24}$ & $((1)(24)(3),(1)(234),(14)(2)(3))$ \\
 & $\C_{25}$ & $((1)(234),(14)(2)(3),(1)(24)(3))$ \\
 & $\C_{26}$ & $((1234),(124)(3),(14)(2)(3))$ \\
\hline\hline
  \multirow{12}{*}{$\T_2$} & $\C_{27}$ & $((1)(2)(3)(4),(12)(34),(12)(34))$ \\
 & $\C_{28}$ & $((12)(34),(12)(34),(1)(2)(3)(4))$ \\
 & $\C_{29}$ & $((1)(2)(34),(12)(3)(4),(12)(34))$ \\
 & $\C_{30}$ & $((12)(34),(12)(3)(4),(1)(2)(34))$ \\
\cline{2-3}
 & $\C_{31}$ & $((1)(234),(12)(34),(132)(4))$ \\
 & $\C_{32}$ & $((1)(243),(123)(4),(13)(24))$ \\
 & $\C_{33}$ & $((132)(4),(12)(34),(1)(234))$ \\
 & $\C_{34}$ & $((1)(23)(4),(12)(34),(1342))$ \\
 & $\C_{35}$ & $((1)(23)(4),(1243),(13)(24))$ \\
 & $\C_{36}$ & $((1342),(12)(34),(1)(23)(4))$ \\
 & $\C_{37}$ & $((1)(23)(4),(1342),(12)(34))$ \\
 \hline
 \end{tabular}
\caption {Classes of matrices $\C_i=\bf{[\M(c;P,Q,R)]}$ where the triplet $(P,Q,R)$ is in the last column of the $i$th row, $ i=7,8,\hdots,37.$ Each $\C_i$ is closed under cogredient relation and contain matrices having at least one entry repeated in one of the columns. Explicit symbolic form of the generic matrix which represents $\C_i$ is listed in the Appendix.}\label{table:cc2}
  \end{center}
\end{table}

\begin{theorem}\label{thm:repeat}
Let $A$ be a generic matrix in $\pds_4$ for the generic vector $c=\bmatrix{c_1 & c_2&c_3&c_4}$ such that at least one of the columns of $A$ contains a repeated entry. Then $A$ represents $\C_i,$ $i\in\{7,8,\hdots,37\}$ where $\C_i$s are defined in Table \ref{table:cc2}. In particular, $A$ represents $\C_i$ for some $i\in\{1,\hdots,6\}$ when some of the entries of $c$ are equal.
\end{theorem}
\noindent\pf Let $A\in\T_4.$ Then consider the following two cases. First, two columns of $A$ contain entries which are repeated $4$ times in each of the columns. Without loss of generality assume that $c_1$ and $c_2$ are repeated $4$ times in first two columns. Then $c_3,c_4$ are repeated twice in the other two columns. Let $A$ be of the form \begin{eqnarray}\label{matrix:repeat1}\bmatrix{c_1&c_2&c_3&c_4\\c_1&c_2&c_3&c_4\\c_1&c_2&c_4&c_3\\c_1&c_2&c_4&c_3}\end{eqnarray} which represents  $\M(c;I_4,(34),(34)).$ Now construct the set of matrices $\{XA : X\in \S_4\}$, and identify matrices which are not cogredient to (\ref{matrix:repeat1}). Then an easy calculation shows that there can be two different cogredient classes of matrices represented by $\C_7$ and $\C_8.$ For other choices of entries in third and fourth columns of $A,$ setting the column sum to be $1$ the matrix $A$ becomes $\frac{1}{4}J_4.$


Otherwise, let $A\in\T_4$ have only one column such that an entry of $c$ is repeated $4$ times in that column. Suppose $c_1$ is such an entry. Suppose among other columns of $A,$ one column contains any $2$ entries from $c_2,c_3,c_4$ that are repeated $2$ times and other $2$ columns contain one of the $c_2, c_3, c_4$ that are repeated two times. Then one of the possible forms of $A$ is $$\bmatrix{c_1&c_2&c_3&c_4\\c_1&c_2&c_4&c_3\\c_1&c_3&c_4&c_2\\c_1&c_4&c_3&c_2}$$ which represents $\M(c;(34),(243),(24)).$ It can be verified that the other generic matrices in $\T_4$ that are of the similar pattern as that of $A$ can be obtained by $PAQ$ for $P,Q\in \S_4.$ Then construct the set $\{XA : X\in\S_4\},$ and identify the matrices in it corresponds to distinct cogredient classes. Then it follows that the classes are given by $\C_9,\hdots,\C_{14}.$

On any other occasions it can be verified that $A\in\T_4$ shall imply $A=\frac{1}{4}J_4.$ For example, let $$A=\bmatrix{c_1&c_2&c_3&c_4\\ c_1&c_2&c_3&c_4\\ c_1&c_2&c_3&c_4\\ c_1&c_3&c_4&c_2}.$$ Observe that $A$ does not belong to any of the cases discussed. However column sum equals $1$ yields $c_1=c_2=c_3=c_4,$ and hence $A=\frac{1}{4}J.$

Let $A\in\T_3,$ that is, in one column of $A,$ an entry of $c$ is repeated $3$ times and no entry of $c$ is repeated more than $3$ times in any column. Then there can be following instances other than $A$ to be $\frac{1}{4}J$. One another entry of $c$ is repeated $3$ times in an another column, two distinct entries are repeated in one column $2$ times, and the  other column contains all four entries of $c.$ Then one of the possible forms of $A$ is $$\bmatrix{c_1&c_2&c_3&c_4\\c_1&c_2&c_4&c_3\\c_1&c_4&c_3&c_2\\c_2&c_4&c_3&c_1}$$ which represents $ \M(c;(34),(24),(142))$ and it is the unique matrix from $\T_3$ of this form up to the pre and post multiplication by permutation matrices $\in \S_4.$ Now construct the set of matrices $\{XA : X\in \S_4\}$, and identify the cogredient matrices in this set. Then it can be verified that there are $12$ distinct cogredient classes of matrices given by $\C_{15},\hdots,\C_{26}.$

There can be other arrangements of entries of $A.$ For example, let $$A=\bmatrix{c_1 & c_2 & c_3 & c_4 \\ c_1 & c_3&c_4&c_2\\ c_1 & c_4&c_2&c_3\\ c_2&c_3&c_1&c_4}$$ which does not fit with any other instances discussed. However employing the condition of column sum to be $1,$ it follows that $c_1=c_3=c_4.$ In that case observe that $A\in[\M(c;(1243),(14)(23),(1342))] \in\C_2$ where $c=\bmatrix{c_1 & c_2 & c_1 & c_1}.$ Similarly for any other specific arrangements of entries an easy calculation show that $A\in\C_i$ for some $i\in\{1,\hdots,14\}$ for some particular choice of the generic vector $c.$

Let $A\in\T_2.$ Let there be a column of $A$ such that one entry of $c$ is repeated $2$ times. Without loss of generality let that column be the first column of $A.$ Suppose the other two positions of the first column are occupied by an another entry of $c$ repeated $2$ times. Then consider the following cases.

In one case, one other column comprises of two entries which are present in first column and each repeats $2$ times. Rest  two columns have two entries other than that of column $1,$ each of which is repeated two times. Then $$A_1=\bmatrix{c_1&c_2&c_3&c_4\\c_1&c_2&c_3&c_4\\c_2&c_1&c_4&c_3\\c_2&c_1&c_4&c_3}~\mathrm{and} ~A_2=\bmatrix{c_1&c_2&c_3&c_4\\c_1&c_2&c_4&c_3\\c_2&c_1&c_3&c_4\\c_2&c_1&c_4&c_3}$$ which represent $\M(c;I_4,(12)(34),(12)(34))$ and $\M(c;(34),(12),(12)(34)),$ respectively, are the only two generic forms of $A$ such that they are not permutation equivalent.

In another case,  let one of the three columns contains two entries of  $c$ that are repeated $2$ times and these entries do not appear in the first column.  The remaining two columns consist of all four entries of $c$ in different places. Then the possible forms of $A$ are $$A_3=\bmatrix{c_1&c_2&c_3&c_4\\c_1&c_4&c_2&c_3\\c_2&c_1&c_4&c_3\\c_2&c_3&c_1&c_4}~ \mathrm{and} ~A_4=\bmatrix{c_1&c_2&c_3&c_4\\c_1&c_3&c_2&c_4\\c_2&c_1&c_4&c_3\\c_2&c_4&c_1&c_3}$$ which represent $\M(c;(234),(12)(34),(132))$ and $\M(c;(23),(12)(34),(1342))$ respectively, such that any other matrix from $\T_2$ having this pattern can be obtained as $PA_3Q$ or $PA_4Q$ for some $P,Q \in \S_4.$ Then deriving the sets of generic matrices $\{XA_i : X\in \S_4\},i=1,\hdots 4$ the cogredient classes can be identified that are of the type $\T_2,$ and these are the classes $\C_{27},\hdots,\C_{37}$ in Table \ref{table:cc2}.

 For any other combination of positions of the entries, the corresponding generic matrix $A\in\C_i$ for some $i=1,\hdots,14.$ For example, let  $$A=\bmatrix{c_1 & c_2 & c_3&c_4\\ c_1 &c_2&c_3&c_4\\ c_2&c_3&c_4&c_1\\ c_2&c_3&c_4&c_1}$$ for which the column sum of $A$ equals $1$ implies that $c_1=c_3,c_2=c_4.$ Hence $A$ represent $\M(c;(13)(24),(12)(34),(14)(23))\in \C_2$ for  $c=\bmatrix{c_1 & c_2 & c_1& c_2}.$ If the non repeated entries in the first column of $A$ are distinct, then all such combinations of positions of the entries of $A$ imply $A\in\C_i$ for some $i=1,\hdots,6.$  For example, let $$A=\bmatrix{c_1 & c_2 & c_3&c_4\\ c_1 &c_4&c_2&c_3\\ c_2&c_3&c_1&c_4\\ c_3&c_1&c_4&c_2}.$$ Then the column sum of $A$ equals $1$ implies that $c_1=c_4.$ The pattern of the entries of $A$ does not fall in any instances above. However, $A$ represent $\M(c;(1234),(1432),(13)(24))\in\C_3$ for $c=\bmatrix{c_1 & c_2 & c_3 & c_1}$ as described in Table {\ref{table:cc2}.  This completes the proof. $\hfill{\square}$


Note that if the generic vector $c$ be such that some of the entries of $c$ are related to each other by the equality signs. Then from Table \ref{table:cc} and Table \ref{table:cc2} it is clear that the corresponding generic matrices $\M(c;P,Q,R)$ may fall in different cogredient classes concurrently.


\section{Eigenvalue region of PDS matrices}
 In this section we determine the eigenvalue region of PDS matrices of order up to $4,$ based on the classification of all such matrices in the previous section. Also, we show some results concerning the eigenvalue region of $\pds_n$ for $n \geq 2.$ 
\subsection{$\pds$ matrices of order $2$ and $3$}
It is trivial to check that $\pds_2$=$\left\{\bmatrix{
a & 1-a \\
1-a & a}:0\leq a \leq 1\right\}$. Obviously $\Lambda(\pds_2) = \Pi_2=\omega_2.$

Now we call the following result from Theorem \ref{thm1}.
\begin{theorem} \label{eig_order3}
$\Lambda(\pds_3)= \Pi_2 \bigcup \Pi_3.$
\end{theorem}
\noindent\pf
It is easy to see that, the matrices  in $[\M(c;(123),(132))]$ are non negative circulant. Thus, the set of eigenvalues of the matrices in  $[\M(c;(123),(132))]$ equals to $$\left\{c_1 + c_2 \xi_k + c_3 \xi_k^2 :  \xi_k = \exp(i\frac{2\pi k}{3}),~ k=0,1,2,~\mbox{and}~  c_i \geq 0, \displaystyle \sum_{i=1}^{3}{c_i} =1\right \}=\Pi_3.$$

  Now, define
  $$
    A_x = \bmatrix{
   0 & x & 1-x \\
  x & 1-x & 0\\
  1-x & 0 & x}
    ,0\leq x \leq 1.
    $$

    Then, $A_x \in [\M(c;(132),(123))]$ and  the eigenvalues of $A_x$  are $ 1, \sqrt{1-3x + 3x^2}$ and $ -\sqrt{1 - 3x + 3x^2}.$ Now, $f(x)=\sqrt{1-3x + 3x^2}$ is a continuous function on the interval  $[\frac{1}{2},1]$.

     It is easy to see that, $f(\frac{1}{2})=\frac{1}{2}$ and $f(1)=1$. Thus, by the intermediate value theorem,  $f$ assumes every values in the interval $[ \frac{1}{2}, 1].$
  Hence $\left[-1, -\frac{1}{2}\right] \cup \left[\frac{1}{2}, 1\right] \subseteq \Lambda(\pds_3)$.
  Thus $ \Lambda(\pds_3) = \Pi_2 \cup \Pi_3=\omega_3. \hfill{\square}$

\subsection{$\pds$ matrices of order $4$}
In this section we determine $\Lambda(\pds_4)=\bigcup_{i=1}^{37} \Lambda(\C_i).$ As mentioned in the introduction that $\Lambda(\C_i)\subseteq \R$ for some $i,$ and for the rest of the classes, exactly two (symbolic) eigenvalues (conjugate pair) are non-real. We derive analytical expression of eigenvalues of $C_i$ when $\Lambda(\C_i)\subseteq \R.$ However for non-real eigenvalues the analytic expression is hard to find for some classes. 

We recall the Perron Frobenius theorem for nonnegative matrices that will be used in sequel \cite{Minc1988}.  \begin{theorem}[Perron-Frobenius Theorem] \label{PF}
  If $A$ is a nonnegative matrix then it has a real nonnegative eigenvalue, say $r$, which is greater than or equal to the modulus of rest of the eigenvalues. Again eigenvector of $A$ corresponding to $r$ is nonnegative. If $A$ is irreducible then $r$ is positive and simple having eigenvector entries positive.
  \end{theorem}
For doubly stochastic matrix $r$ is $1$ and $(1, 1,...,1)^{T}$ is the corresponding positive eigenvector.

\subsubsection{ Real spectrum}

Note that the classes $\C_5=\bf{[\M(c;(12)(34),(13)(24),(14)(23))]}$ and $\C_6=\bf{[\M(c;(12)(34),(1423),(1324))]}$ contain symmetric matrices. Thus we have the following.
 \begin{theorem}
 $\Lambda(\C_5 \cup \C_6)=[-1,1].$
\end{theorem}
\noindent \pf
Since, the matrices in $\C_5$ and $\C_6$ are symmetric, so the eigenvalues of $\C_5$ and $\C_6$  are in  $[-1, 1].$ Now, let us consider the following sub collection of matrices of  $\C_5 \cup \C_6,$
 $$K=
\left\{
\bmatrix
{c_1 & c_2 & 0 & 0 \\
c_2 & c_1 & 0 & 0\\
0 & 0 & c_1 & c_2 \\
0 & 0 & c_2 & c_1}:
 c_1,c_2 \geq 0~\mathrm{and}~c_1+c_2=1
\right\}.
$$
Then $ \Lambda(K)=\left\{c_1 + c_2, c_1 -c_2, c_1 + c_2, c_1-c_2: c_1,c_2 \geq 0~\mathrm{and}~c_1+c_2=1\right\}$. Since, $ c_1 -c_2$ is an eigenvalue, which is a convex combination of $1$ and $-1$, the desired result follows. $\hfill{\square}$

In Table \ref{table:eig} we list classes $\C_i$ of $\pds_4$ such that $\Lambda(\C_i) \subseteq \R$ and $\Lambda(\C_i)$ can be determined easily from the associated characteristic polynomial. We also describe $\Lambda(\C_i)$ in terms of entries of the generic matrices which define $\C_i.$
\begin{table}[H]\small
  \begin{center}
            \begin{tabular}{|c|c|c|c|}
    \hline
     {\textbf{Class}} &{$(P,Q,R)$} & $c$ &  {$\Lambda(\M(c;P,Q,R))$}\\
         \hline
      \hline
     {$\C_7$} &{$(I_4,(34),(34))$} & $(\frac{1}{4},\frac{1}{4},c_3,\frac{1}{2}-c_3)$ & \{1,0,0,0\} \\
      \hline \hline
       {$\C_8$} & {$((34),(34),I_4)$} &  $(\frac{1}{4},\frac{1}{4},c_3,\frac{1}{2}-c_3)$ & $\left\{1,\frac{1}{2}-2c_3,0,0:0 \leq c_3 \leq \frac{1}{2}\right\}$ \\
      \hline \hline
{$\C_{10}$} & {$((34),(24),(243))$} &  $(\frac{1}{4},\frac{1}{4},c_3,\frac{1}{2}-c_3)$ &  $\{1,0,2c_3-\frac{1}{2},\frac{1}{4}-c_3:0\leq c_3 \leq \frac{1}{2}\}$\\
      \hline \hline
$\C_{12}$ & $((243),(34),(24))$  &  $(\frac{1}{4},\frac{1}{4},c_3,\frac{1}{2}-c_3)$& $\left\{1, 0,\pm \frac{1}{2\sqrt2}(4 c_3 -1):0\leq c_3 \leq \frac{1}{2}\right\}$\\
      \hline \hline
$\C_{13}$ &$((24),(243),(34))$ &  $(\frac{1}{4},\frac{1}{4},c_3,\frac{1}{2}-c_3)$ &  $\{1,0,\frac{1}{2}-2 c_3,c_3-\frac{1}{4}:0\leq c_3 \leq \frac{1}{2}\}$\\
      \hline \hline
 $\C_{14}$ &  $((243),(24),(34))$ &  $(\frac{1}{4},\frac{1}{4},c_3,\frac{1}{2}-c_3)$ & $\{1,0,c_3-\frac{1}{4},2 c_3-\frac{1}{2}:0\leq c_3 \leq \frac{1}{2}\}$ \\
\hline \hline
$\C_{18}$ &  $((1234),(14),(124))$ &  $(1-3c_4,3c_4-\frac{1}{2},\frac{1}{2}-c_4,c_4)$ & $\left\{1, 1-4 c_4,\pm \frac{1}{2} \sqrt 6 (1-4c_4):\frac{1}{6}\leq c_4 \leq \frac{1}{3}\right\}$ \\
\hline \hline
$\C_{22}$ & $((14),(1234),(124))$ &  $ (1-3c_4,3c_4-\frac{1}{2},\frac{1}{2}-c_4,c_4)$ & $ \left\{ 1, 6 c_4-\frac{3}{2}, 1-4 c_4, 4 c_4-1:\frac{1}{6}\leq c_4 \leq \frac{1}{3} \right\}$\\
\hline \hline
$\C_{24}$ & $((24),(234),(14))$  &  $(c_1,3c_1-\frac{1}{2},\frac{1}{2}-c_1,1-3c_1)$ & $\left\{1, \frac{3}{2}-6c_1,4 c_1-1, 4 c_1-1:\frac{1}{6}\leq c_1 \leq \frac{1}{3}\right\}$ \\
\hline \hline
$\C_{27}$ & $(I_4,(12)(34),(12)(34))$&  $(c_1,\frac{1}{2}-c_1,c_3,\frac{1}{2}-c_3)$ &$\{1,0,0,0\}$  \\
\hline \hline
$\C_{28}$ & $((12)(34),(12)(34),I_4)$ &  $(c_1,\frac{1}{2}-c_1,c_3,\frac{1}{2}-c_3)$ & $\{1,2c_1-2c_3,0,0 : 0\leq c_1,c_3 \leq \frac{1}{2}\}$ \\
\hline \hline
$\C_{29}$ &((34),(12),(12)(34)) &  $(c_1,\frac{1}{2}-c_1,c_3,\frac{1}{2}-c_3)$ & $\left\{1, 0, 0,2c_3-\frac{1}{2}:0 \leq c_3 \leq \frac{1}{2}\right\}.$ \\
\hline \hline
$\C_{35}$ & $((23),(1243),(13)(24))$ &  $(c_1,c_2,\frac{1}{2}-c_1,\frac{1}{2}-c_2)$ & $\left\{1,0,\pm \sqrt2 \left(c_1+c_2-\frac{1}{2}\right): 0 \leq c_1,c_2 \leq \frac{1}{2}\right\}$\\
\hline \hline
$\C_{36}$ & $(1342),(12)(34),(23))$ &  $(c_1,\frac{1}{2}-c_1,c_3,\frac{1}{2}-c_3)$ & $\left\{1,0,2(c_1-c_3),\frac{1}{2}-c_1-c_3: 0 \leq c_1,c_3 \leq \frac{1}{2}\right\}$\\
\hline
          \end{tabular}
\caption {The triplet $(P,Q,R)$ in the second column represent the class $\C_i=\bf{[\M(c;P,Q,R)]}.$ The entries of $c$ are specific due to doubly stochastic constraint on $\M(c;P,Q,R).$ The fourth column exhibits the eigenvalue region of $\C_i.$}\label{table:eig}
  \end{center}
\end{table}


Now we consider the class $\C_{33}=\bf{[\M(c;(132),(12)(34),(234))]}.$ Then the spectrum of the generic matrix which represents $\C_{33}$ is given by \begin{align*}
\Lambda(\M(c;(132),(12)(34),(234)))=&\left\{1,0,\frac{1}{2}(c_1+c_3)-\frac{1}{4}+\frac{1}{2}\sqrt{9 c_1^2+9c_3^2-14 c_1 c_3-c_1-c_3+\frac{1}{4}},\right.\\
  &\left.\frac{1}{2}(c_1+c_3)-\frac{1}{4}-\frac{1}{2}\sqrt{9 c_1^2+9c_3^2-14 c_1 c_3-c_1-c_3+\frac{1}{4}}: 0 \leq c_1,c_3 \leq \frac{1}{2}\right\},
\end{align*}
where $c=(c_1,\frac{1}{2}-c_1,c_3,\frac{1}{2}-c_3).$ The eigenvalues can easily be found by deriving the factors of the associated characteristic polynomial.We show below that $\Lambda(\M(c;(132),(12)(34),(234))) \subseteq [-1, \, 1].$ Suppose $f(c_1,c_3)=9 c_1^2+9c_3^2-14 c_1 c_3-c_1-c_3+\frac{1}{4}$ defined over the rectangle $0 \leq c_1 \leq \frac{1}{2},0 \leq c_3 \leq \frac{1}{2}.$ Now,
$$\nabla f=[18 c_1-14 c_3-1,18 c_3 -14 c_1-1]$$ and the Hessian matrix
 $$H=\bmatrix{18 & -14\\ -14 & 18}.$$
 Solving for $\nabla f=0$ we have $c_1=\frac{1}{4}, c_3=\frac{1}{4}.$ As $det(H)>0$, and $\frac{\partial^2 f}{\partial c_1^2}>0,$ $f(c_1,c_3)$ has minimum value $0$ at $(\frac{1}{4},\frac{1}{4}).$ On both the lines $c_3=0$ and $c_1=0,$ $f(c_1,c_3)$ has a minimum value $\frac{2}{9}$ at $(\frac{1}{18},0)$ and at $(0,\frac{1}{18})$ respectively. Computing the values at the corner points of the rectangle $0 \leq c_1, c_2 \leq \frac{1}{2},$ we have $f(0,0)=\frac{1}{4}=f(\frac{1}{2},\frac{1}{2}),f(\frac{1}{2},0)=2=f(0,\frac{1}{2}).$ Hence $2$ and $0$ are the global maximum and global minimum values of $f$ respectively. Hence all the eigenvalues of the class $\C_{33}$ are real.

\subsubsection{Complex spectrum of matrices described in Theorem \ref{thm:norepeat}}
  First we give the eigenvalue region of $\C_1.$ We also describe $\Lambda(\C_2)$ and $\Lambda(\C_4).$ Explicit expression of eigenvalues of $\C_3$ is cumbersome.
\begin{itemize}
\item \noindent\textbf{Eigenvalue region for the matrix class $\C_1$:}

  Recall that $\C_1=\bf[\M(c;(12)(34),(1324),(1423)]$ is comprised of the generic matrices $A,B,C$ given in (\ref{Thm4.1:eq1}). Clearly $C$ is a symbolic DS circulant matrix. Hence $\Lambda(\C_1)$ is the convex hull of $\{1, -1, i, -i\}$, which is same as $\Pi_4.$

\item \noindent\textbf{Eigenvalue region for the matrix class $\C_2$:}

Recall that $\C_2=\bf{[\M(c;(12)(34),(14)(23),(13)(24))]}.$
 If $A \in \M(c;(12)(34),(14)(23),(13)(24))$ with first row $c=(c_1, c_2, c_3,c_4)$, then
 \begin{eqnarray*}
 \Lambda(A)=
 & \left\{c_1 + c_2 + c_3 + c_4, c_1 + c_2 - c_3 - c_4, {((c_1 - c_2 + c_3  - c_4 )(c_1 - c_2 - c_3 + c_4))}^{\frac{1}{2}}, \right.  \\
 & \left.-{((c_1 - c_2 + c_3  - c_4 )(c_1 - c_2 - c_3 + c_4))}^{\frac{1}{2}}\right\}.
 \end{eqnarray*}
 As $c_1, c_2, c_3, c_4$ are nonnegative and sum up to $1$, so all the eigenvalues of the matrices in $\M(c;(12)(34),(14)(23),(13)(24))$  belong to  $[-1,1] \cup \{i t : t \in [-1, 1] \},$ and hence $\Lambda(\C_2)$ is contained in   $[-1,1] \cup \{i t : t \in [-1, 1] \}.$

\begin{figure}[H]
\centering
\subfigure[Eigenvalues of the generic matrix in $\C_1$]{\includegraphics[height=6 cm,width=6.5 cm]{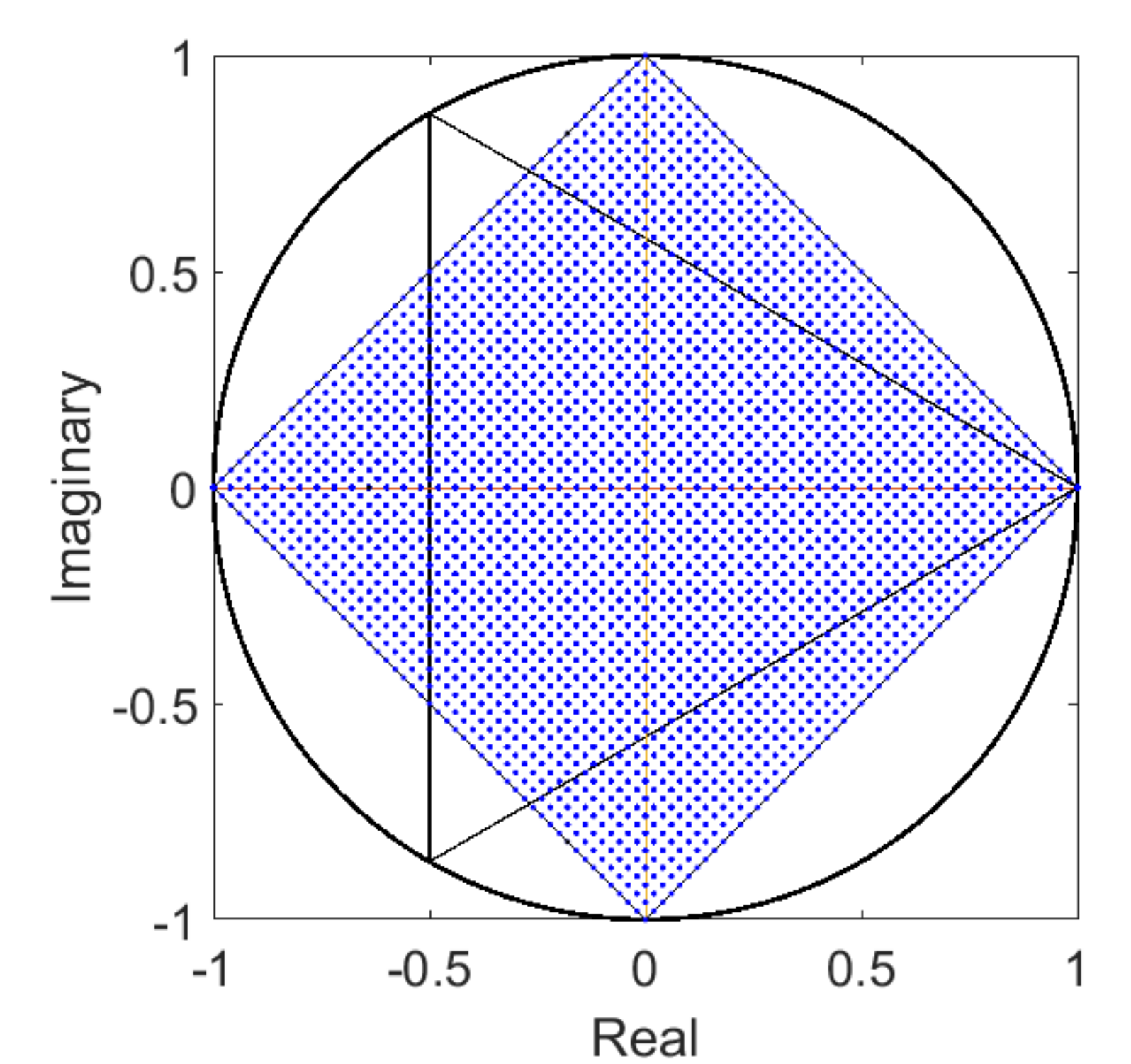}}
\hspace{0.4cm}
\subfigure[Eigenvalues of the generic matrix in $\C_2$]{\includegraphics[height=6 cm,width=6.5 cm]{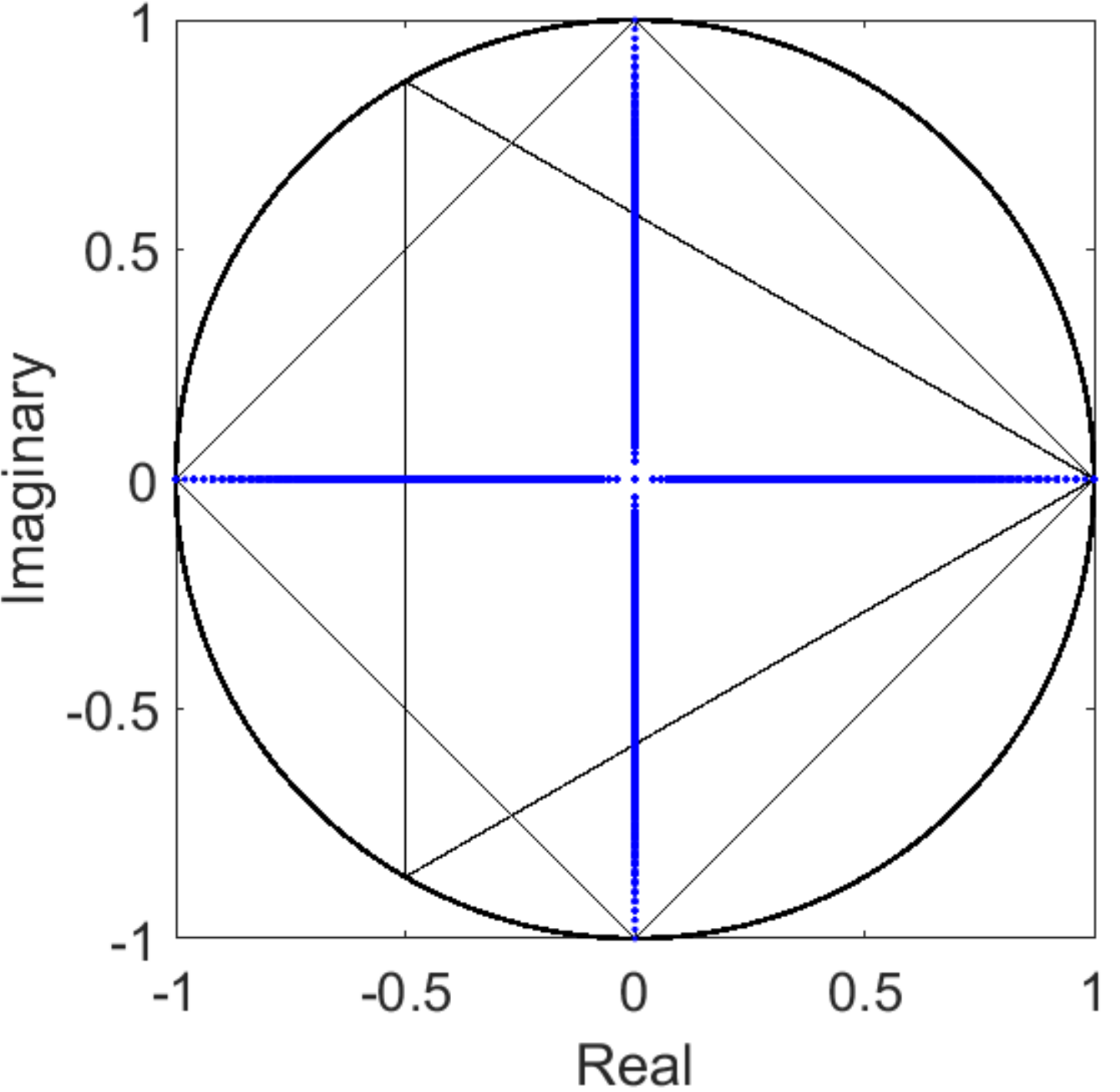}}
\caption{Setting several values of $0\leq c_i\leq 1$ such that $\sum_{i=1}^4 c_i=1,$ eigenvalues of generic matrices corresponding to $\C_1, \C_2$ are plotted in blue. }
\end{figure}

\item\noindent\textbf{Eigenvalue region for the matrix class $\C_3$:}

Recall that $\C_3=\bf{[\M(c;(1423),(1234),(13)(24))]}.$
If $A$ is the generic matrix which  represets $\M(c;(1423),(1234),(13)(24))$ with first row $c=(c_1, c_2, c_3,c_4)$, then  the characteristic polynomial of $A$ is given by
\begin{eqnarray*}
\mathcal{P}_3(x) &  = & x^4 + (- c_1 - 2 c_2 - c_3)x^3 + ( 2 c_1 c_2 - 2 c_1 c_4 + 2 c_2 c_3 - 2 c_3 c_4)x^2 +(- c_1^3 + c_1^2 c_3 + 2 c_1^2 c_4 \\
 & & - c_1 c_2^2- 4 c_1 c_2 c_3 + 2 c_1 c_2 c_4 + c_1 c_3^2 - c_1 c_4^2 + 2 c_2^3 - c_2^2 c_3 + 2 c_2 c_3 c_4 - 2 c_2 c_4^2 - c_3^3 + 2 c_3^2 c_4  \\
 & & - c_3 c_4^2)x+ ( c_1^4 - 4 c_1^2 c_2 c_4 - 2 c_1^2 c_3^2 + 4 c_1 c_2^2 c_3 + 4 c_1 c_3 c_4^2 - c_2^4 + 2 c_2^2 c_4^2 - 4 c_2 c_3^2 c_4 + c_3^4 - c_4^4).
 \end{eqnarray*}
  Let $\{1, \lambda, \sigma + \tau i, \sigma - \tau i\}$ be the spectrum of any matrix from $\M(c;(1423),(1234),(13)(24)).$
 Then we have $c_1 + 2 c_2 + c_3=1+\lambda+2 \sigma$ and $2 c_1 c_2 - 2 c_1 c_4 + 2 c_2 c_3 - 2 c_3 c_4=\lambda+2 \sigma+2 \sigma \lambda+\sigma^2+\tau^2 $ which implies that,
 \begin{align*}
   & c_2-c_4=\lambda+2 \sigma ~ \mathrm{and} \\
   & 2(c_1+c_3)(c_2-c_4)=\lambda+2 \sigma+2 \sigma \lambda+\sigma^2+\tau^2
  \end{align*}
respectively.
  By solving these two equations we obtain
  $\tau=\pm \sqrt{3 \sigma^2+(c_2-c_4)(2c_1+2c_3-2\sigma-1)}$ for $-1\leq \sigma\leq 1$ for which $\tau$ is real, and $\lambda=c_2-c_4-2\sigma.$

\item\noindent\textbf{Eigenvalue region for the matrix class $\C_4$:}

Recall that $\C_4=\bf{[\M(c;(13)(24),(14)(23),(12)(34))]}.$
 If $A$ represents $\M(c;(13)(24),(14)(23),(12)(34))$ with first row $(c_1, c_2, c_3,c_4)$, then  the characteristic polynomial of $A$ is given by:
\begin{eqnarray*}
\mathcal{P}_4(x) & = & x^4 - (c_1 + c_2 + c_3 + c_4) x^3 + 0 x^2 + ( - c_1^3 + c_1^2 c_2 + c_1^2 c_3 + c_1^2 c_4 + c_1 c_2^2 - 2 c_1 c_2 c_3 - 2 c_1 c_2 c_4 \\
& & +c_1 c_3^2 - 2 c_1 c_3 c_4 + c_1 c_4^2 - c_2^3 + c_2^2 c_3 + c_2^2 c_4 + c_2 c_3^2 - 2 c_2 c_3 c_4 + c_2 c_4^2 - c_3^3 + c_3^2 c_4 + c_3 c_4^2 - c_4^3)x \\
& & +( c_1^4 - 2 c_1^2 c_2^2 - 2 c_1^2 c_3^2 - 2 c_1^2 c_4^2 + 8 c_1 c_2 c_3 c_4 + c_2^4 - 2 c_2^2 c_3^2 - 2 c_2^2 c_4^2 + c_3^4 - 2 c_3^2 c_4^2 + c_4^4).
\end{eqnarray*}

In the next theorem, we establish the region $\Lambda(\C_4)$.
\begin{theorem}
 The  set of all eigenvalues of matrices in  $\C_4$ are of the form $$\left\{1, -2\sigma, \sigma + \sqrt 3 \sigma i, \sigma - \sqrt 3\sigma i\right\},$$ where $|\sigma| \leq \frac{1}{2}.$
 \end{theorem}
\noindent \pf
  Let $\{1, \lambda, \sigma + \tau i, \sigma - \tau i\}$ be the spectrum of any matrix from $\M(c;(13)(24),(14)(23),(12)(34)).$ Since,  all matrices in $\C_4$ have trace  $1$, we have
\begin{eqnarray} \label{sub1}
1+ \lambda + 2\sigma =1.
\end{eqnarray}
 Also, the coefficient of $x^2$ equals to zero in $\mathcal{P}_4(x)$, so
\begin{eqnarray} \label{sub2}
\lambda + 2\sigma + 2\sigma\lambda + \sigma^2 + \tau^2=0.
\end{eqnarray}

Solving $(\ref{sub1})$ and $(\ref{sub2})$, we get  $\lambda = -2 \sigma$ and $|\tau|=\sqrt3 \sigma$. By Theorem \ref{PF}, we have $|\sigma|\leq \frac{1}{2}.$ $\hfill{\square}$

\end{itemize}

\begin{figure}[H]
\centering
\subfigure[Eigenvalues of the generic matrix in $\C_3$]{\includegraphics[height=6 cm,width=6.5 cm]{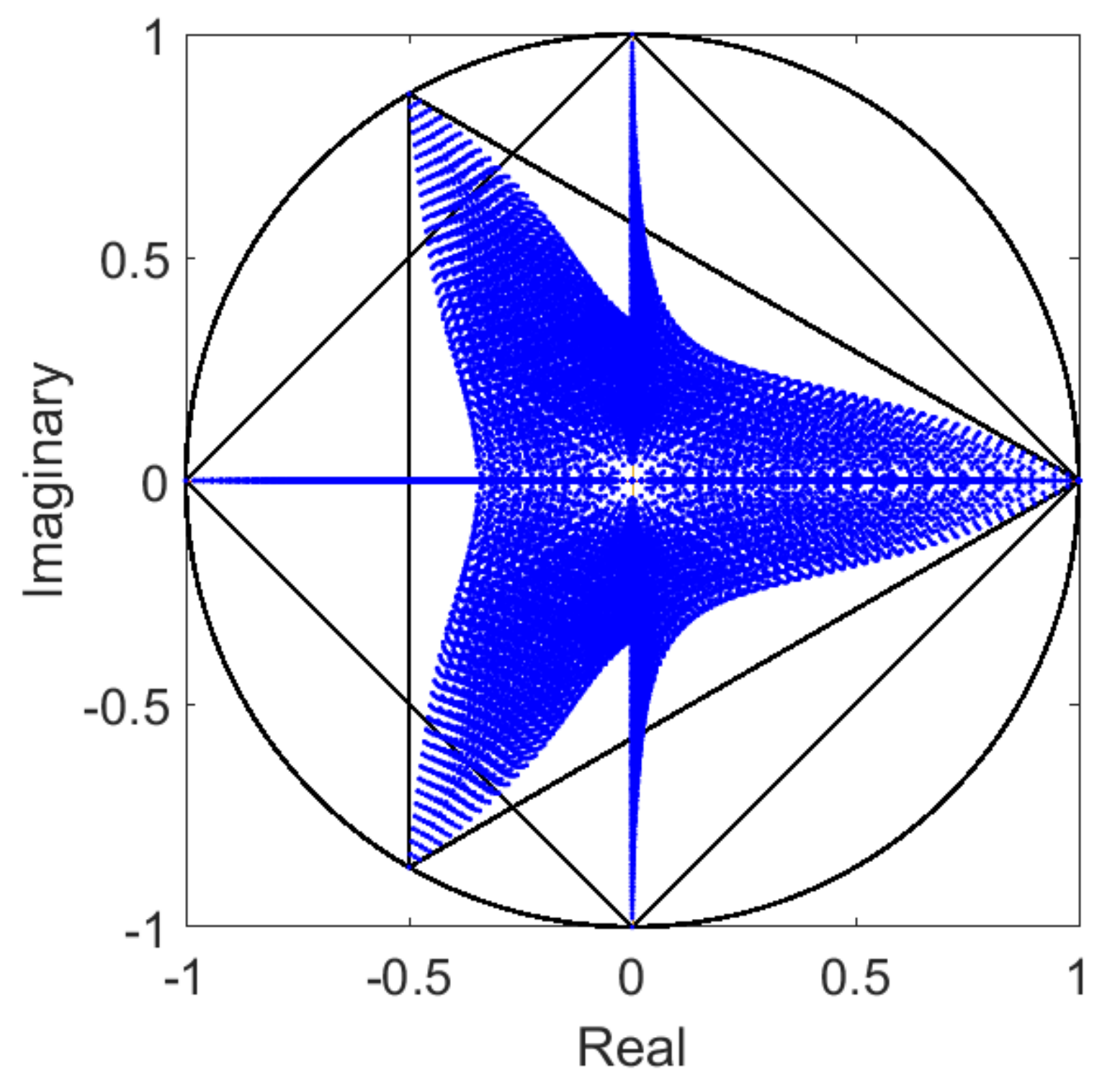}}
\hspace{0.4cm}
\subfigure[Eigenvalues of the generic matrix in $\C_4$]{\includegraphics[height=6 cm,width=6.5 cm]{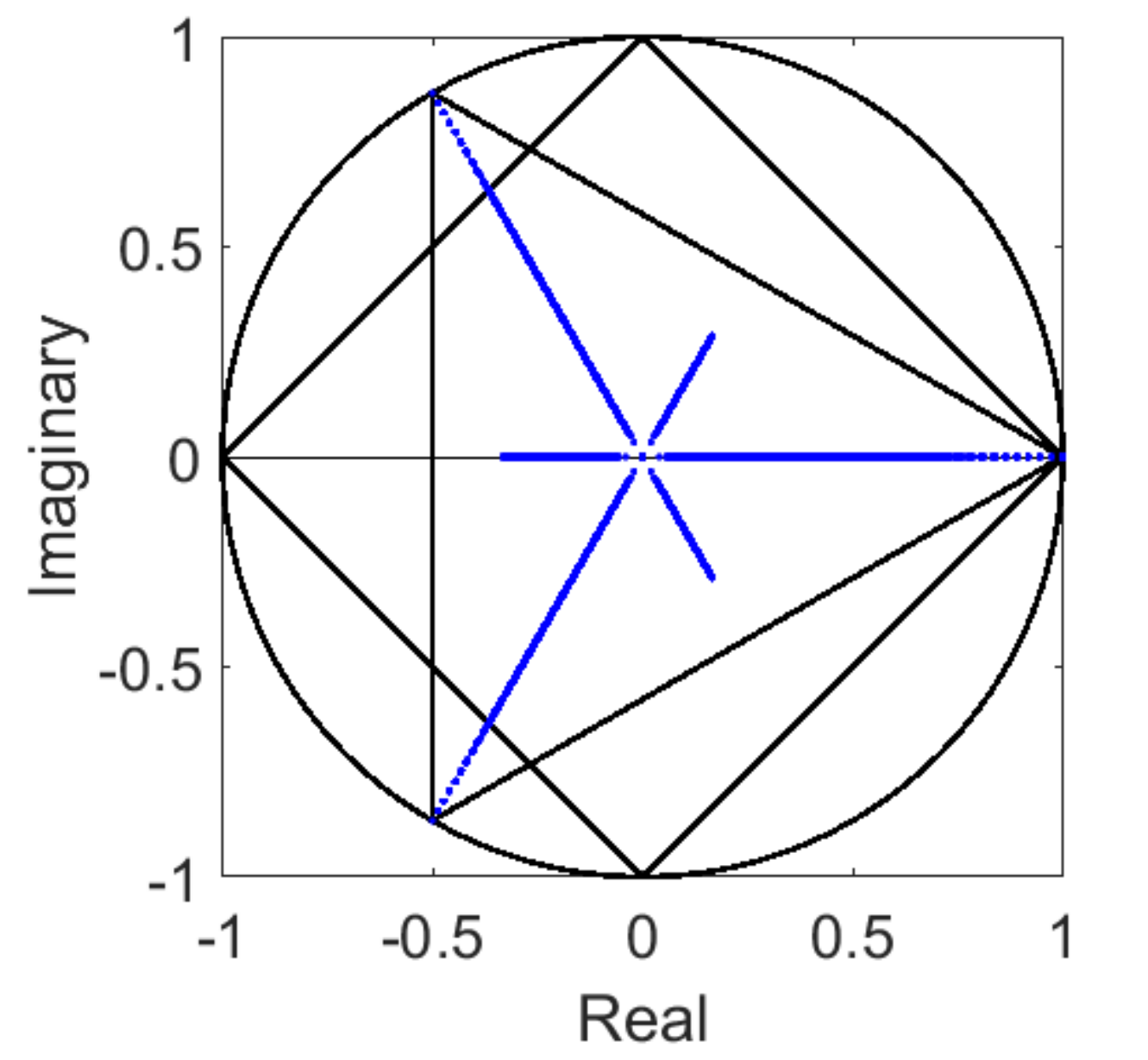}}
\caption{Setting several values of $0\leq c_i\leq 1$ such that $\sum_{i=1}^4 c_i=1,$ eigenvalues of generic matrices corresponding to $\C_3, \C_4$ are plotted in blue.}
\end{figure}

Next we consider the eigenvalue region of the matrices described in Theorem \ref{thm:repeat} having complex spectrum. In each of the following matrices the entries of $c$ have been specified due to the doubly stochastic nature of the concerned matrices.

\subsubsection{Complex spectrum of matrices of $\T_4$ type}
 Here $\C_9=\bf{[ \M(c;,(34),(243),(24))]}$ and $\C_{11}=\bf{[ \M(c;(24),(34),(243))]}$ are the two different cogredient matrix classes of type $\T_4.$
 If $A$ represents $\M(c;(34),(243),(24))$ where $c=(\frac{1}{4}, \frac{1}{4}, c_3,\frac{1}{2}-c_3)$, then it follows that
$$\Lambda(A)=\left\{1, 0, -\frac{1}{8}(4 c_3 -1)+i \frac{\sqrt 7}{8}(4 c_3-1),-\frac{1}{8}(4 c_3 -1)- \frac{\sqrt 7}{8}(4 c_3-1):0 \leq c_3 \leq \frac{1}{2}\right\}.$$
  Hence the maximum absolute value of real and imaginary part of the complex eigenvalues, among all matrices in $\C_9$  are  $\frac{1}{8}$ and $\frac{\sqrt 7}{8}$ respectively.

Now consider a generic matrix $B \in \M(c;(24),(34),(243))$ having $c=(\frac{1}{4}, \frac{1}{2}-c_3, c_3, \frac{1}{4})$ as first row. Hence
$$\Lambda(B)=\left\{1, 0, -\frac{1}{4}(4 c_3 -1)+i \frac{1}{4}(4 c_3-1),-\frac{1}{4}(4 c_3 -1)-i \frac{1}{4}(4 c_3-1):0 \leq c_3 \leq \frac{1}{2}\right\}.$$  Then it is clear that the maximum absolute value is $\frac{1}{4}$ for both the real and imaginary parts of the complex eigenvalues of $\C_{11}$.

\begin{figure}[H]
\centering
\includegraphics[height=6 cm,width=6.5 cm]{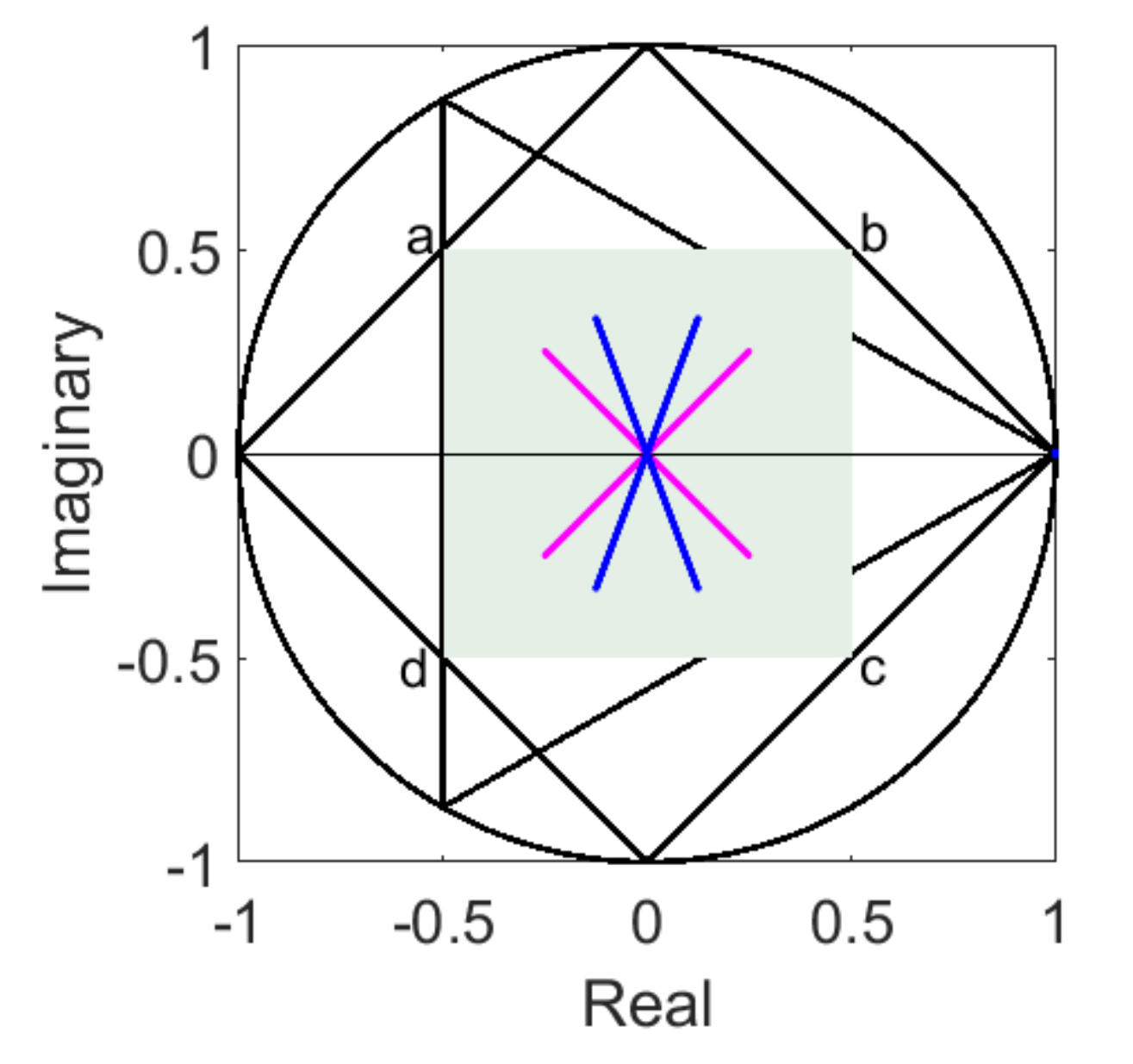}
\caption{Setting several values of $0\leq c_3\leq \frac{1}{2}$ eigenvalues of generic matrices corresponding to $\C_9, \C_{11}$ are plotted in blue and magenta respectively.}\label{fig:c9_11}
\end{figure}

\subsubsection{Complex spectrum of matrices of $\T_3$ type}
Eigenvalue regions of the classes $\C_{17}=\bf{[\M(c;(234),(24),(14)]},\C_{19}=\bf{[\M(c;(24),(34),(142))]},\C_{21}=\bf{[\M(c;(34),(1432),(243))]}$ and $\C_{26}=\bf{[\M(c;(1234),(124),(14))]}$ can be shown to be the following by directly computing the roots of the characteristic polynomials of the corresponding generic matrices.
\begin{itemize}
\item $\Lambda\left(\M(c;(234),(24),(14))\right)=\left\{1,4c_1-1,3(-\frac{1}{4}+\frac{\sqrt{15}}{12}i) (4 c_1-1), 3(-\frac{1}{4}-\frac{\sqrt{15}}{12}i) (4 c_1-1):\frac{1}{6}\leq c_1\leq \frac{1}{3}\right\},$ where $c=(c_1,3c_1-\frac{1}{2},\frac{1}{2}-c_1,1-3c_1).$

 \item $\Lambda(\M(c;(24),(34),(142)))=\left\{1,4c_1-1,\frac{1}{2}(1+\sqrt5i)(4c_1-1), \frac{1}{2}(1-\sqrt5i)(4c_1-1):\frac{1}{6}\leq c_1\leq \frac{1}{3}\right\},$
 where $c=(c_1,1-3c_1,\frac{1}{2}-c_1,3c_1-\frac{1}{2}).$

 \item $\Lambda(\M(c;(34),(1432),(243)))=\left\{1, \frac{3}{2}-6c_1, (4c_1-1)i, -(4c_1-1)i:\frac{1}{6}\leq c_1\leq \frac{1}{3}:\frac{1}{6}\leq c_1\leq \frac{1}{3}\right\},$ where $c=(c_1,1-3c_1,3c_1-\frac{1}{2},\frac{1}{2}-c_1).$

 \item $\Lambda(\M(c;(1234),(124),(14)))=\left\{1,1-4c_4,3(-\frac{1}{4}+\frac{\sqrt{15}}{12}i) (4 c_4-1), 3(-\frac{1}{4}-\frac{\sqrt{15}}{12}i) (4 c_4-1):\frac{1}{6}\leq c_4\leq \frac{1}{3}\right\},$ where $c=(1-3c_4,3c_4-\frac{1}{2},\frac{1}{2}-c_4,c_4).$
 \end{itemize}

The maximum modulus of real and imaginary part ([real,imaginary]) of  complex eigenvalues of  $\C_{17},\C_{19},\C_{21}$ and $\C_{26}$ are $[\frac{1}{4}, \frac{\sqrt{15}}{12}],[\frac{1}{6},\frac{\sqrt5}{6}],[\frac{1}{3},\frac{1}{3}]$ and $[\frac{1}{4}, \frac{\sqrt{15}}{12}]$ respectively.

  Recall that $\C_{15}=\bf{[\M(c;(34),(24),(142)]},\C_{16}=\bf{[\M(c;(34),(243),(1432))]},$
  $\C_{20}=\bf{[\M(c;(34),(142),(24))]},\C_{23}=\bf{[\M(c;(243),(34),(1432))]},$ and $\C_{25}=\bf{[\M(c;(234),(14),(24))]}.$


\begin{itemize}
\item Characteristic polynomial of the generic matrix $A_1= \M(c;(34),(24),(142))$ with $c=(c_1,1-3 c_1,\frac{1}{2}-c_1,3 c_1-\frac{1}{2})$ is:
\begin{align*}
\mathcal{P}_{15}(x)=&x^4+\left(2 c_1-\frac{3}{2}\right) x^3 + \left(\frac{1}{2}-2 c_1\right) x^2+\left(18 c_1-72 c_1^2-\frac{3}{2}+96 c_1^3\right) x+72 c_1^2-96 c_1^3-18 c_1+\frac{3}{2}, \\
&\mathrm{where}~ \frac{1}{6}\leq c_1 \leq \frac{1}{3}.
\end{align*}


   \item Characteristic polynomial of the generic matrix $A_2 = \M(c;(34),(243),(1432))$ with $c=(c_1,1-3 c_1,3 c_1-\frac{1}{2},\frac{1}{2}-c_1)$ is:
   \begin{align*}
  \mathcal{P}_{16}(x)=& x^4+\left(2c_1-\frac{3}{2}\right)x^3+\left(-\frac{1}{2}+6c_1-16c_1^2\right)x^2+\left(-26c_1+88c_1^2+\frac{5}{2}-96 c_1^3\right) x\\
  &-72 c_1^2+96 c_1^3+18 c_1-\frac{3}{2},\\
   &\mathrm{where}~ \frac{1}{6}\leq c_1 \leq \frac{1}{3}.
   \end{align*}


 \item   Characteristic polynomial of the generic matrix $A_3 = \M(c;(34),(142),(24))$ with $c=(c_1,1-3 c_1,\frac{1}{2}-c_1,3 c_1-\frac{1}{2})$ is:
    \begin{align*}
    \mathcal{P}_{20}(x)=&x^4+\left(-\frac{5}{2}+6c_1\right)x^3+\left(-6c_1+\frac{3}{2}\right)x^2+\left(72c_1^2-96c_1^3-18c_1+\frac{3}{2}\right)x+96c_1^3-72c_1^2+18c_1-\frac{3}{2},\\
     &\mathrm{where}~ \frac{1}{6}\leq c_1 \leq \frac{1}{3}.
    \end{align*}

\item     Characteristic polynomial of the generic matrix $A_4 = \M(c;(243),(34),(1432))$ with $c=(c_1,1-3 c_1,3 c_1-\frac{1}{2},\frac{1}{2}-c_1)$ is:
     \begin{align*}
    \mathcal{P}_{23}(x)=& x^4-4 c_1x^3+\left(-\frac{1}{2}+8c_1^2\right)x^2+\left(96c_1^3-80c_1^2+22c_1-2\right)x+72c_1^2-96c_1^3-18c_1+\frac{3}{2},\\
     & \mathrm{where}~\frac{1}{6}\leq c_1 \leq \frac{1}{3}.
     \end{align*}

  \item    Characteristic polynomial of the generic matrix $A_5 = \M(c;(234),(14),(24))$ with $c=(c_1,3 c_1-\frac{1}{2},\frac{1}{2}-c_1,1-3 c_1)$ is:
      \begin{align*}
\mathcal{P}_{25}(x)=&x^4-x^3+\left(-24c_1^2+12c_1-\frac{3}{2}\right)x^2+\left(96c_1^3-48c_1^2+6c_1\right)x-96c_1^3+72c_1^2-18c_1+\frac{3}{2},\\
 &\mathrm{where}~ \frac{1}{6}\leq c_1 \leq \frac{1}{3}.
     \end{align*}

 \end{itemize}

 The computable expression of the eigenvalues of the matrices $A_1,\ldots,A_5$ can be obtained from these polynomials using Cardano's method. However we do not write the explicit expression of these eigenvalues as it is cumbersome. 

  Below we plot some of the eigenvalues of the generic matrices corresponding to the above cogredient classes of $\T_3.$
  \begin{figure}[H]
\centering
\includegraphics[height=6 cm,width=6.5 cm]{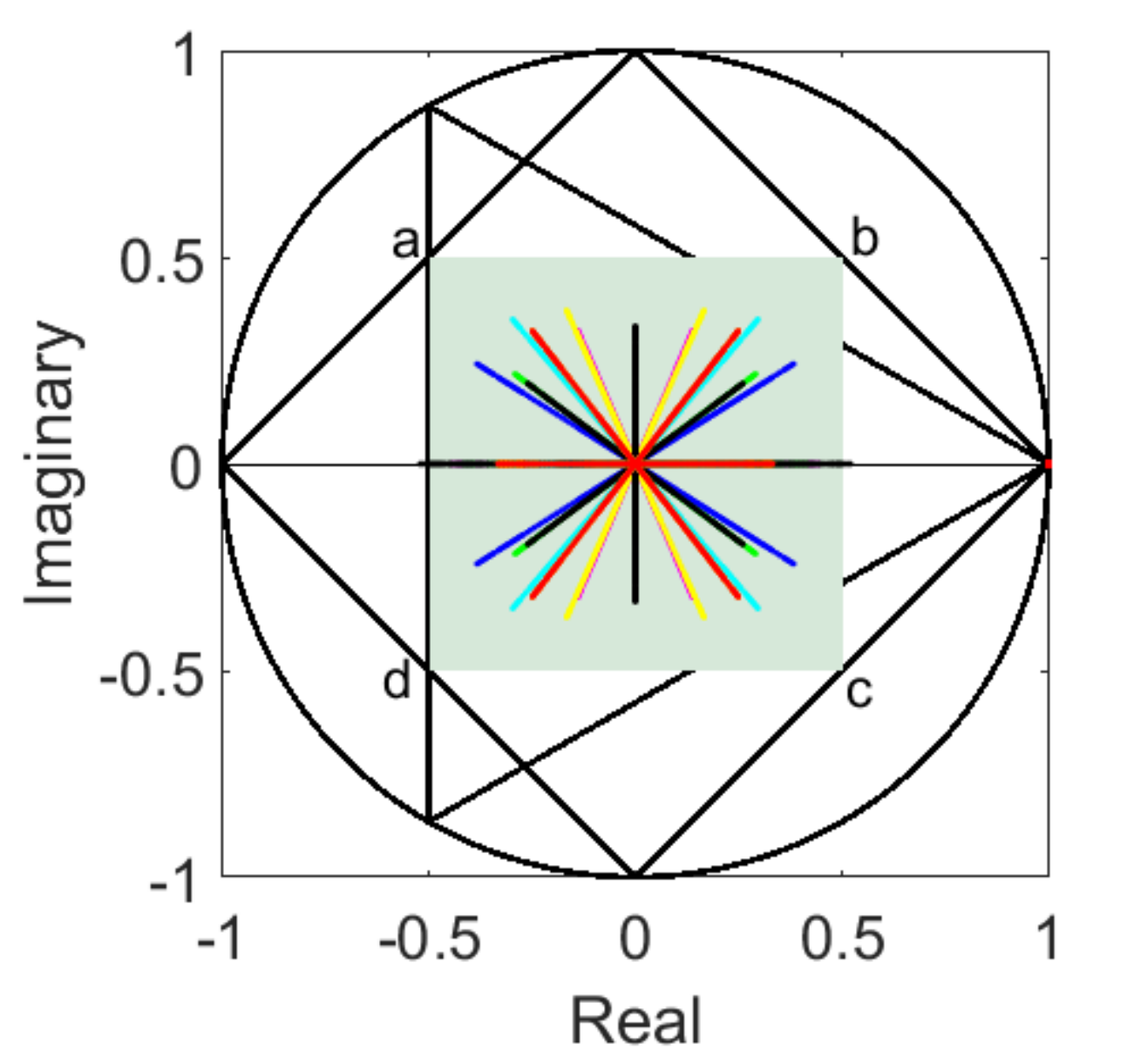}
\caption{Setting several values of the parameters which describe $\Lambda(\C_i,)$ $i=15, 16, 17, 19, 20, 21, 23, 25, 26$ above, the eigenvalues of generic matrices corresponding to these classes are plotted in different colors. }
\end{figure}

 \subsubsection{Complex spectrum of matrices of $\T_2$ type}
Now consider the cogredient classes $\C_{30}=\bf{[\M(c;(12)(34),(12),(34))]},\C_{31}=\bf{[\M(c;(234),(12)(34),(132))]},\C_{32}=\bf{[\M(c;(243),(123),(13)(24))]},\C_{34}=\bf{[\M(c;(23),(12)(34),(1342))]},$ and $\C_{37}=\bf{[\M(c;(23),(1342),(12)(34))]}$ whose corresponding generic matrices are of type $\T_2.$ We show below that they have non-real eigenvalues.

It can be easily shown that
\begin{align*}
\begin{split}
\Lambda(\M(c;(12)(34),(12),(34)))=&\left\{1,0,c_1+c_3-\frac{1}{2}+\frac{1}{2}\sqrt{4 c_1^2 + 4 c_3^2 -24 c_1 c_3 + 4 c_1 + 4 c_3 -1}, \right. \\
 &\left. c_1+c_3-\frac{1}{2}+\frac{1}{2}\sqrt{4 c_1^2 + 4 c_3^2 -24 c_1 c_3 + 4 c_1 + 4 c_3 -1}:0 \leq c_1,c_3 \leq \frac{1}{2}\right\}, ~\mathrm{where} \\
 & c=(c_1,\frac{1}{2}-c_1,c_3,\frac{1}{2}-c_3).\\
 \end{split}
\end{align*}

We shall now compute the global maximum and minimum of the function
 $f_1(c_1,c_3)=4 c_1^2 + 4 c_3^2 -24 c_1 c_3 + 4 c_1 + 4 c_3 -1$ over the rectangle $0 \leq c_1 \leq \frac{1}{2},0 \leq c_3 \leq \frac{1}{2}.$ We compute
$$ \nabla f_1
=[8 c_1 -24 c_3 +4,8 c_3 -24 c_1 +4]$$  and the Hessian matrix
 $$H_1
       =\bmatrix{ 8 & -24 \\ -24 & 8}.$$

Solving for $\nabla f_1=0$ we have $c_1=\frac{1}{5}, c_3=\frac{1}{10}.$
 Since determinant of $H_1,$ $det(H_1)<0$ at $(\frac{1}{5},\frac{1}{10}),$ it is a saddle point  . On both the lines $c_1=0$ and $c_3=0,$ $f_1(c_1,c_3)$ has negative critical points, which are out of our range. Computing the functional values at the corner points we have $f_1(0,0)=-1,f_1(\frac{1}{2},0)=2,f_1(0,\frac{1}{2})=2,f_1(\frac{1}{2},\frac{1}{2})=-1.$ Hence $2$ and $-1$ are the global maximum and global minimum value of $f_1$ respectively.
So that the maximum modulus of both real and imaginary part of the complex eigenvalues of $\M(c;(12)(34),(12),(34))$ is $\frac{1}{2}.$}

\begin{figure}[H]
\centering
\includegraphics[height=6 cm,width=7 cm]{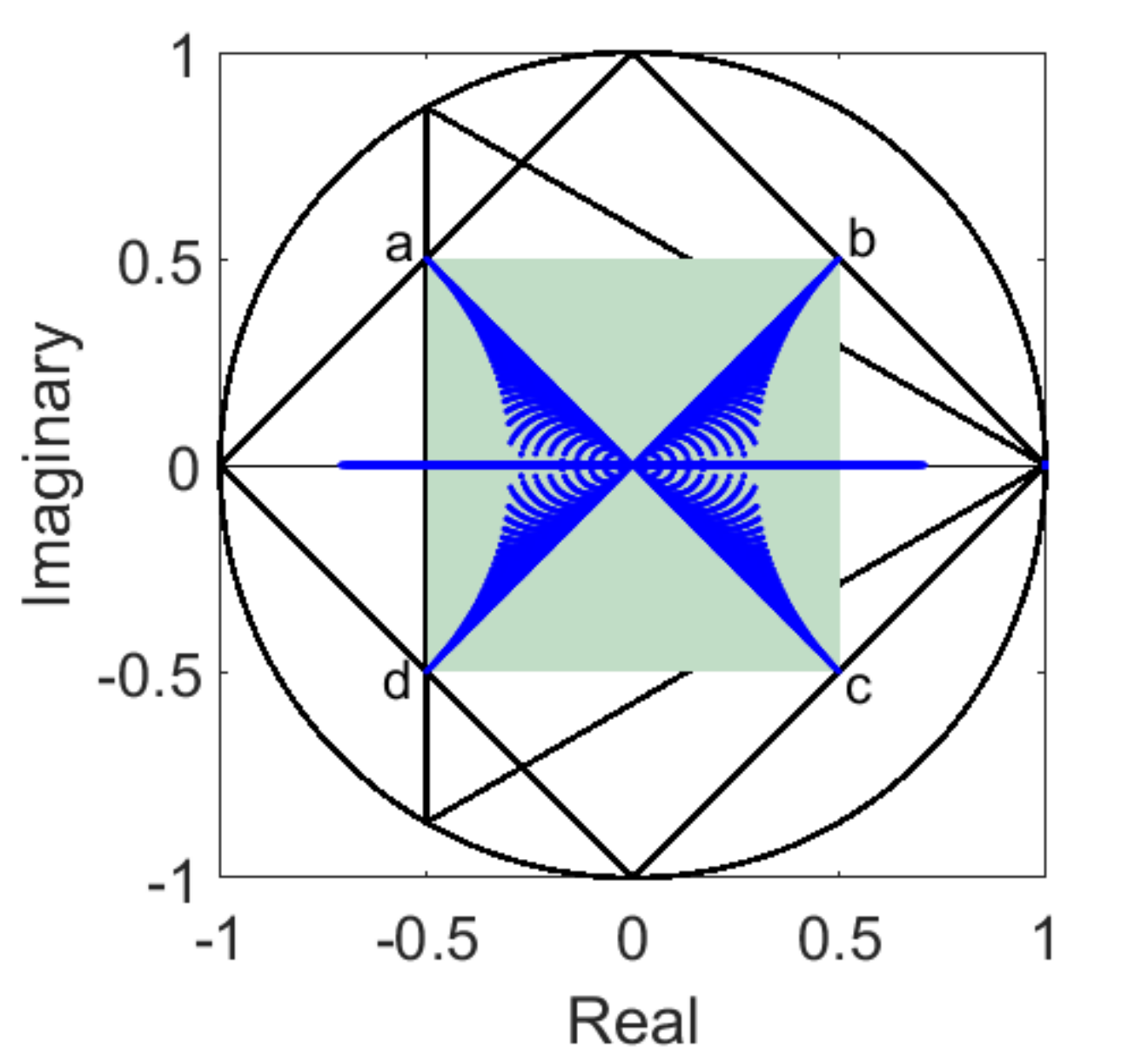}
\caption{Setting several values of $0\leq c_1, c_3\leq \frac{1}{2}$ eigenvalues of generic matrices corresponding to $\C_{30}$ are plotted in blue.}
\end{figure}

Now,
\begin{align*}
 \Lambda(\M(c;(234),(12)(34),(132)))=&\left\{1,0,\frac{1}{2}(c_1-3 c_3)+\frac{1}{4}+\frac{1}{2}\sqrt{-7 c_1^2+c_3^2+10 c_1 c_3+c_1-3 c_3+\frac{1}{4}},\right. \\
 &\left. \frac{1}{2}(c_1-3 c_3)+\frac{1}{4}-\frac{1}{2}\sqrt{-7 c_1^2+c_3^2+10 c_1 c_3+c_1-3 c_3+\frac{1}{4}}: 0 \leq c_1,c_3 \leq \frac{1}{2} \right \},\\
 &\mathrm{where}~ c=(c_1,\frac{1}{2}-c_1,c_3,\frac{1}{2}-c_3).
\end{align*}

We shall now compute the global maximum and minimum of the function
$f_2(c_1,c_3)=-7 c_1^2+c_3^2+10 c_1 c_3+c_1-3 c_3+\frac{1}{4}$ over the rectangle $0 \leq c_1 \leq \frac{1}{2},0 \leq c_3 \leq \frac{1}{2}.$
Now $$\nabla f_2=[-14 c_1+10 c_3+1,2 c_3+10 c_1-3]$$ and the Hessian matrix
$$H_2=\bmatrix{ -14 & 10 \\ 10 & 2}.$$
 Solving for $\nabla f_2=0$ we have $c_1=\frac{1}{4}, c_3=\frac{1}{4}.$ Since $det(H_2)<0, (\frac{1}{4},\frac{1}{4})$ is a saddle point. On the line $c_3=0,$ $f_2(c_1,0)$ has a maximum value $\frac{2}{7}$ at $c_1=\frac{1}{14}.$ The critical point of $f_2(0,c_3)$ is $c_3=\frac{3}{2}>\frac{1}{2}.$ Computing the values at the corner points we have $f_2(0,0)=\frac{1}{4},f_2(\frac{1}{2},0)=-1,f_2(0,\frac{1}{2})=-1,f_2(\frac{1}{2},\frac{1}{2})=\frac{1}{4}.$ Hence $\frac{2}{7}$ and $-1$ are the global maximum and global minimum value of $f_2$ respectively. Hence maximum absolute value for imaginary part as well as real part of a complex eigenvalue in $\C_{31}$ is $\frac{1}{2}.$


Similarly,
\begin{align*}
\Lambda(\M(c;(243),(123),(13)(24)))=&\left\{1,0,c_2-\frac{1}{4}+\frac{1}{2}\sqrt{8 c_1^2-4 c_2^2-4 c_1+2 c_2+\frac{1}{4}}, \right. \\
&\left.c_2-\frac{1}{4}-\frac{1}{2}\sqrt{8 c_1^2-4 c_2^2-4 c_1+2 c_2+\frac{1}{4}}: 0 \leq c_1,c_2 \leq \frac{1}{2} \right\},
\end{align*}
where $c=(c_1,c_2,\frac{1}{2}-c_1,\frac{1}{2}-c_2).$

Finally,
\begin{align*}
\Lambda(\M(c;(23),(12)(34),(1342)))=& \left\{1,0, \frac{1}{2}(c_1+c_3)-\frac{1}{4}+\frac{1}{2}\sqrt{-7 c_1^2 + 9 c_3^2+2c_1 c_3+3 c_1-5 c_3+\frac{1}{4}},\right. \\
&\left. \frac{1}{2}(c_1+c_3)-\frac{1}{4}-\frac{1}{2}\sqrt{-7 c_1^2 + 9 c_3^2+2c_1 c_3+3 c_1-5 c_3+\frac{1}{4}}: 0 \leq c_1,c_3 \leq \frac{1}{2} \right\},
\end{align*}
where $c=(c_1,\frac{1}{2}-c_1,c_3,\frac{1}{2}-c_3),$
and
\begin{align*}
\Lambda(\M(c;(23),(1342),(12)(34)))=& \left\{1,0,\frac{1}{2}(1+i)(2c_1+2c_3-1),\frac{1}{2}(1-i)(2c_1+2c_3-1): 0 \leq c_1,c_3 \leq \frac{1}{2}\right\},
\end{align*}
where $c=(c_1,\frac{1}{2}-c_1,c_3,\frac{1}{2}-c_3).$

 Let $f_3(c_1,c_2)=8 c_1^2-4 c_2^2-4 c_1+2 c_2+\frac{1}{4}$ defined over the rectangle $0 \leq c_1 \leq \frac{1}{2},0 \leq c_2 \leq \frac{1}{2}.$
We compute
$$\nabla f_3=[16 c_1-4,-8 c_2+2]$$
 and the Hessian matrix $$H_3=\bmatrix{ 16 & 0 \\ 0 & -8}. $$
 Solving for $\nabla f_3=0$ we have $c_1=\frac{1}{4}, c_3=\frac{1}{4}.$ Since $det(H_3)<0, (\frac{1}{4},\frac{1}{4})$ is a saddle point. On the line $c_2=0,$ $f_3(c_1,0)$ has a minimum value $-\frac{1}{4}$ at $c_1=\frac{1}{4}$ and on the line $c_1=0,$ $f_3(0,c_2)$ has a maximum value $\frac{1}{2}.$ At the corner points of the rectangle $0 \leq c_1, c_2 \leq \frac{1}{2},$ we have $f_3(0,0)=f_3(\frac{1}{2},0)=f_3(0,\frac{1}{2})=f_3(\frac{1}{2},\frac{1}{2})=\frac{1}{4}.$ Hence $\frac{1}{2}$ and $-\frac{1}{4}$ are the global maximum and global minimum values of $f_3$ respectively.

 Let $f_4(c_1,c_3)=-7 c_1^2 + 9 c_3^2+2c_1 c_3+3 c_1-5 c_3+\frac{1}{4}$ over the rectangle $0 \leq c_1 \leq \frac{1}{2},0 \leq c_3 \leq \frac{1}{2}.$ We compute
$$\nabla f_4=[-14 c_1+2 c_3+3,18 c_3 +2 c_1-5]$$ and the Hessian matrix
$$H_4=\bmatrix{-14 & 2\\ 2 & 18}$$ for the function $f_4.$
Solution of $\nabla f_4=0$ is given by $c_1=\frac{1}{4}, c_3=\frac{1}{4}.$ Since $det(H_4)<0, (\frac{1}{4},\frac{1}{4})$ is a saddle point. On the line $c_3=0,$ $f_4(c_1,0)$ has a maximum value $\frac{4}{7}$ at $c_1=\frac{3}{14}$ and on the line $c_1=0,$ $f_4(0,c_2)$ has a minimum value $-\frac{4}{9}$ at $c_3=\frac{5}{18}.$ Computing the values of $f_4(c_1,c_3)$ at the corner points of the rectangle $0 \leq c_1, c_3 \leq \frac{1}{2},$ we have $f_4(0,0)=\frac{1}{4}=f_4(\frac{1}{2},0),f_4(0,\frac{1}{2})=0=f_4(\frac{1}{2},\frac{1}{2}).$ Hence $\frac{4}{7}$ and $-\frac{4}{9}$ are the global maximum and global minimum values of $f_4.$

From the above it can be seen easily that the maximum modulus of the imaginary part of complex eigenvalues of  $\C_{32},\C_{34}$ and $C_{37}$ are $\frac{1}{4},\frac{1}{3}$ and $\frac{1}{2}$ respectively, whereas, the absolute value of real parts are $\frac{1}{4},\frac{1}{4},\frac{1}{2}$ respectively.

\begin{figure}[H]
\centering
\subfigure[Eigenvalues of the generic matrix corresponding to  $\C_{31}$]{\includegraphics[height=5.5 cm,width=6 cm]{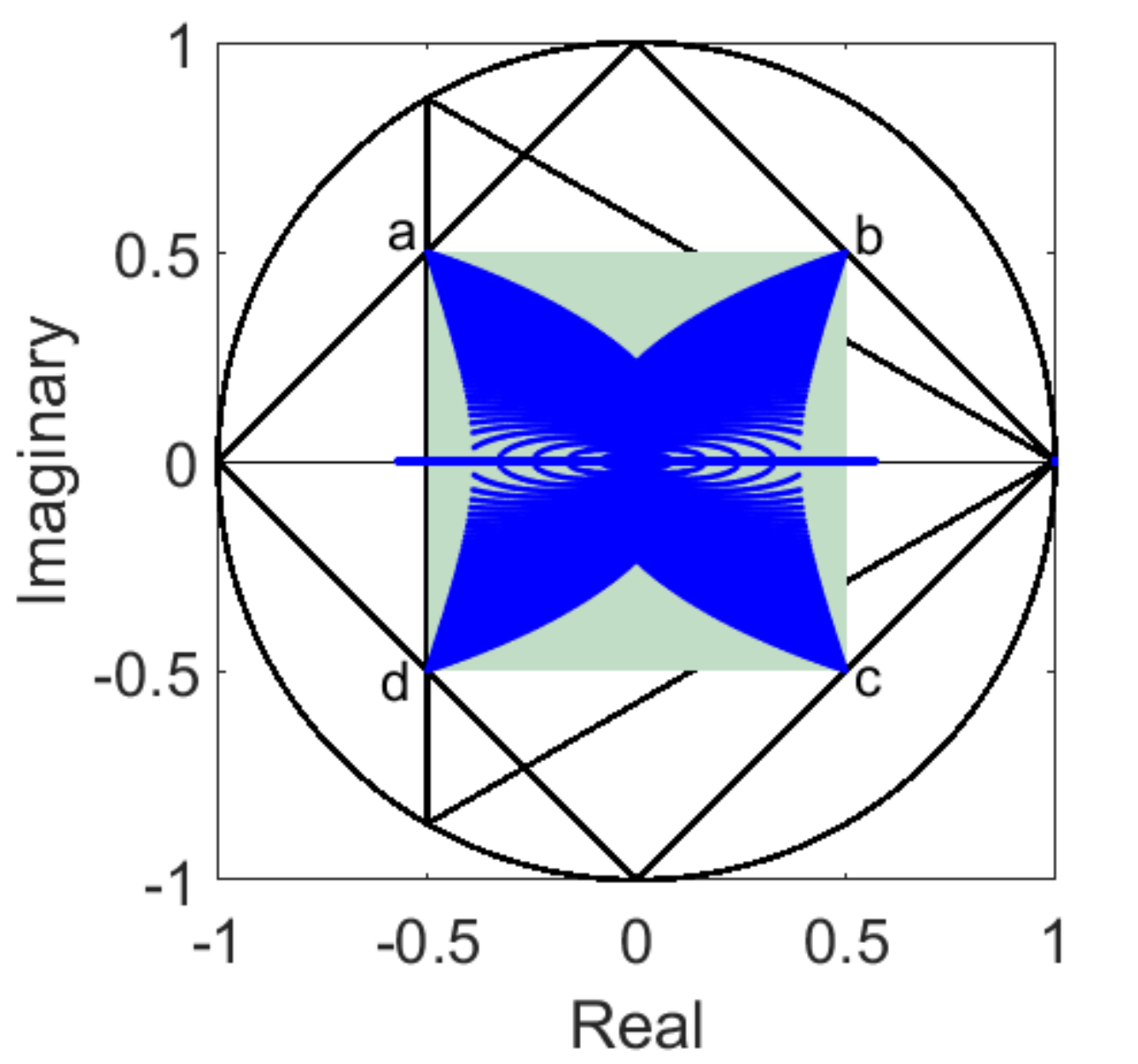}}
\hspace{0.4cm}
\subfigure[Eigenvalues of the generic matrix corresponding to  $\C_{32}$]{\includegraphics[height=5.5 cm,width=6 cm]{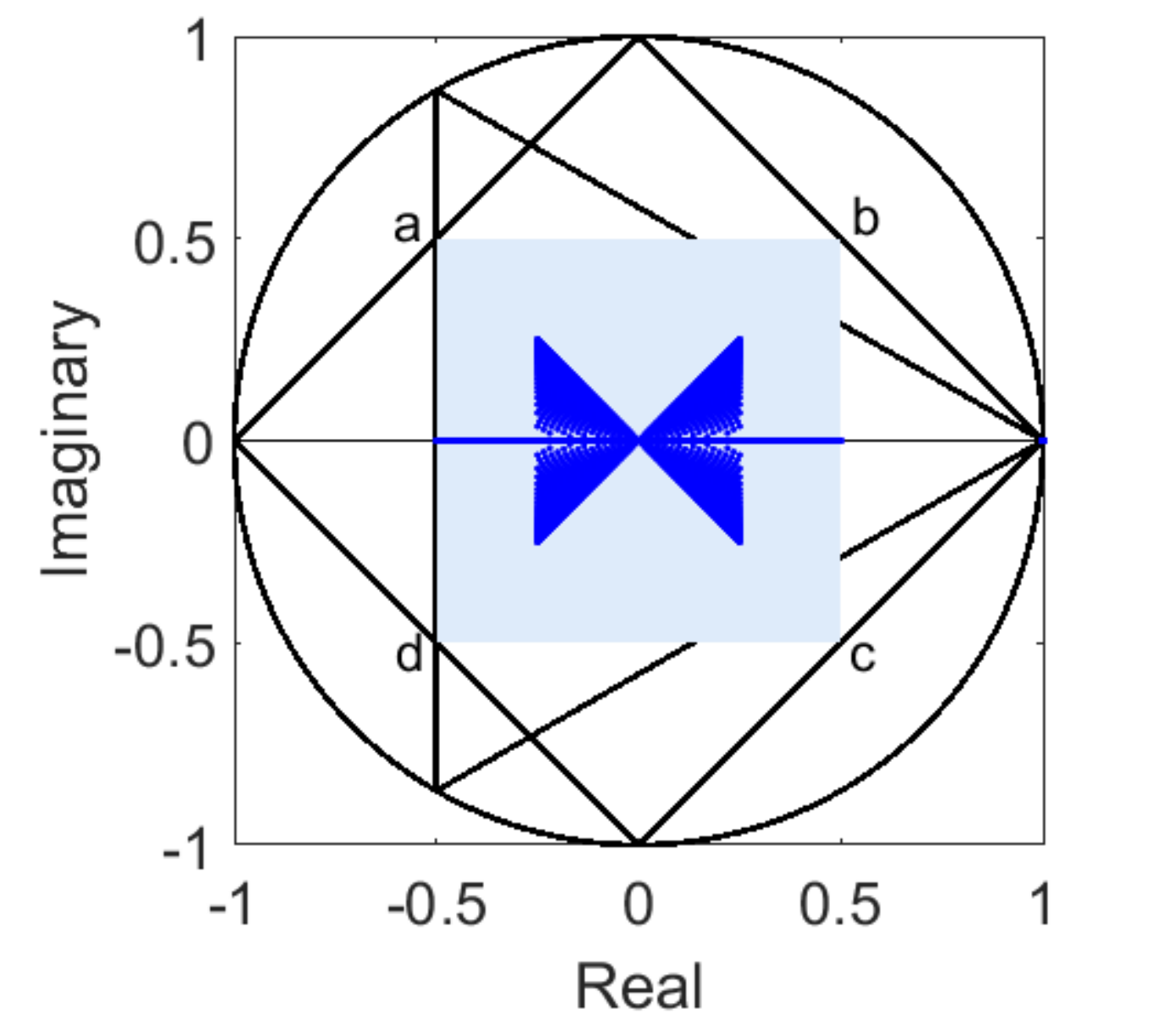}}
\caption{Setting several values of $0\leq c_1, c_3\leq \frac{1}{2}, 0\leq c_1,c_2\leq \frac{1}{2}$ eigenvalues of generic matrices corresponding to $\C_{31}, \C_{32}$ are plotted in blue, respectively.}
\end{figure}

\begin{figure}[H]
\centering
\subfigure[Eigenvalues of the generic matrix corresponding to $\C_{34}$]{\includegraphics[height=5.5 cm,width=6 cm]{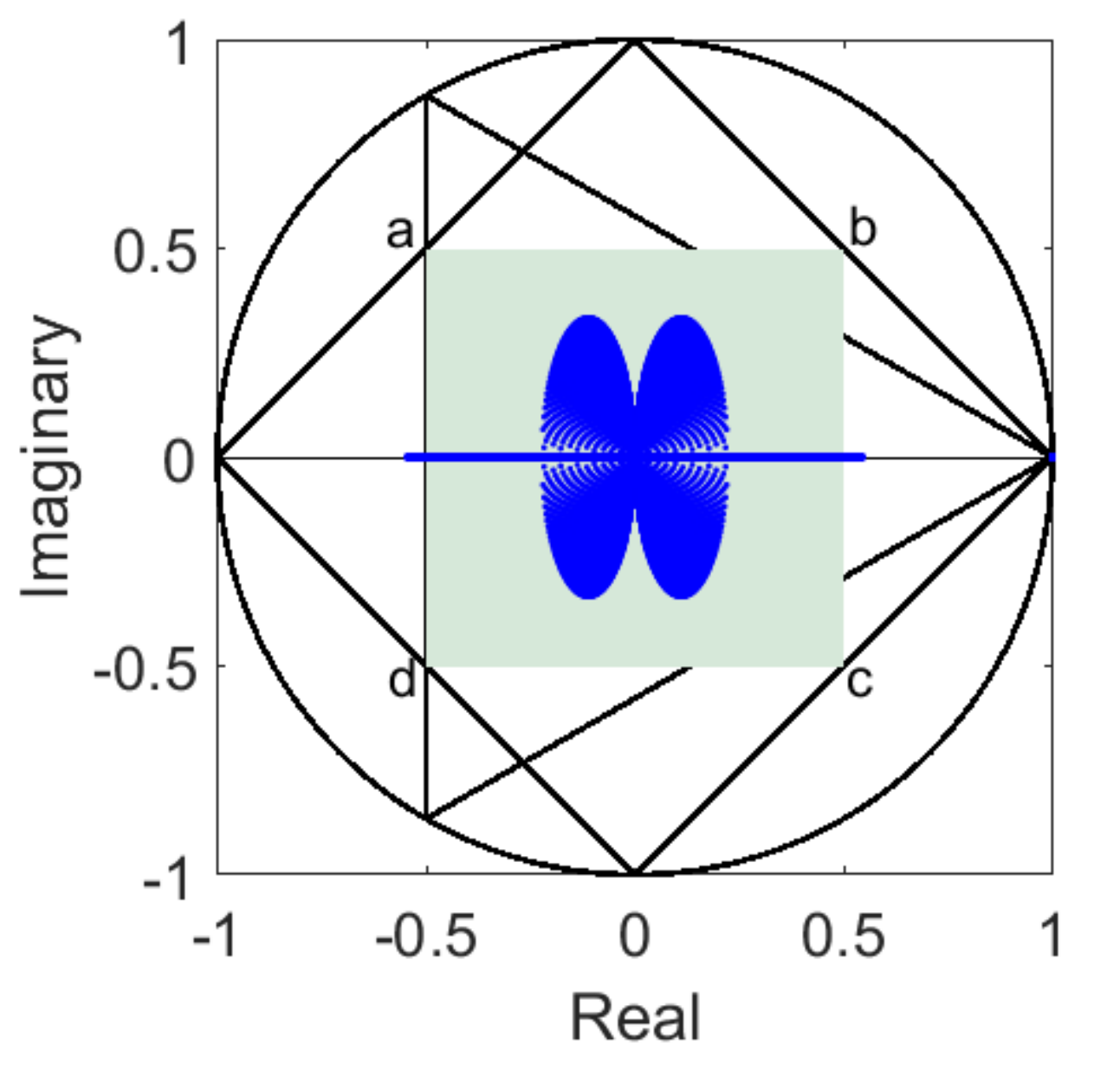}}
\hspace{0.4cm}
\subfigure[Eigenvalues of the generic matrix corresponding to $\C_{37}$]{\includegraphics[height=5.5 cm,width=6 cm]{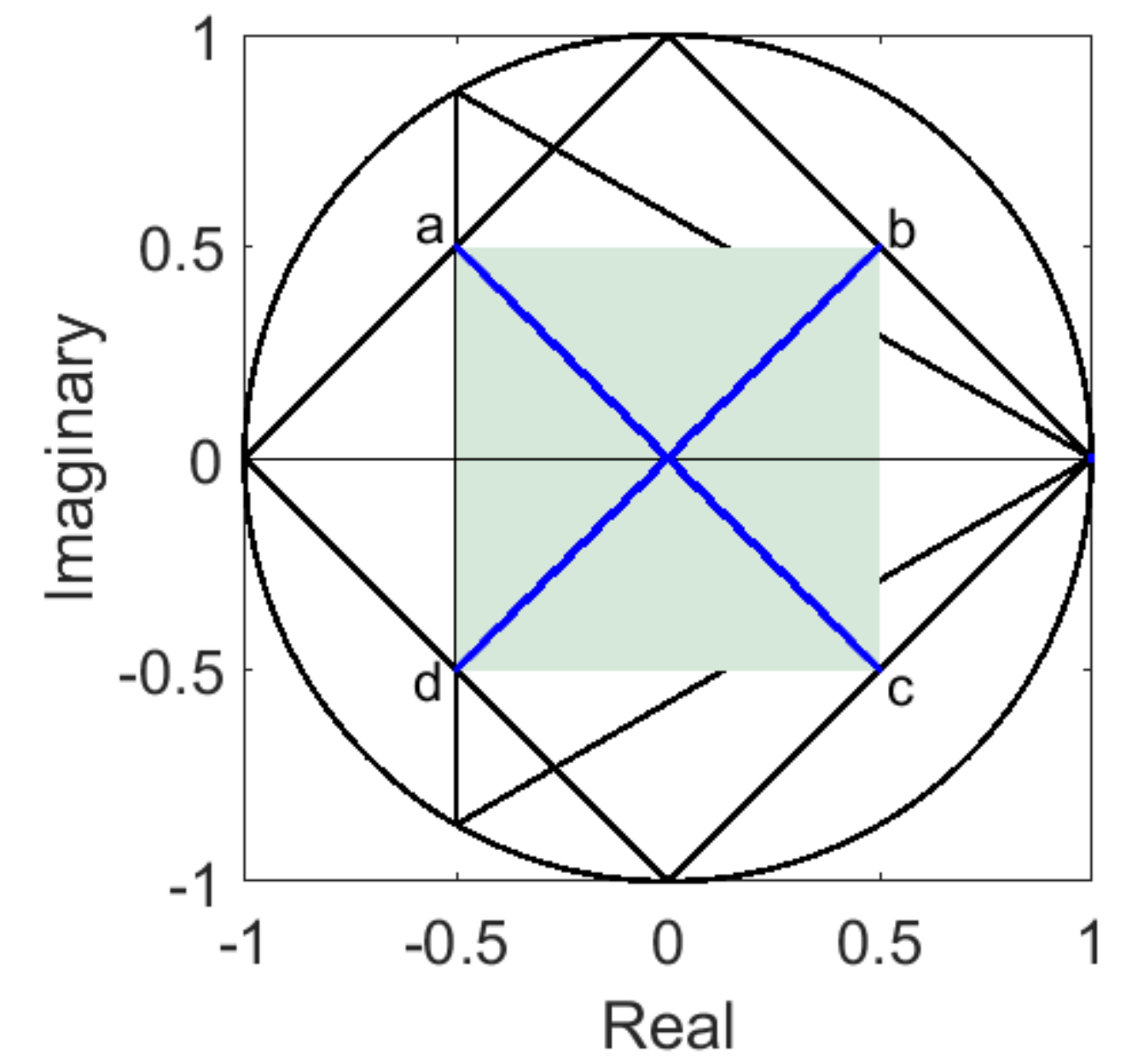}}
\caption{Setting several values of $0\leq c_1, c_3\leq \frac{1}{2}$ eigenvalues of generic matrices corresponding to $\C_{34}, \C_{37}$ are plotted in blue.}\label{fig:c34_37}
\end{figure}

We Plot the eigenvalue regions of all the cogredient classes except that  having real spectrum.  Observe that the eigenvalue regions  are contained in the union of the real line segment $[-1,1]$ and the square ${\bf{abcd}}$ drawn in Figures \ref{fig:c9_11}-\ref{fig:c34_37} having vertices $\bf{a}=-\frac{1}{2}+ i \frac{1}{2},\bf{b}=\frac{1}{2}+i \frac{1}{2},\bf{c}=\frac{1}{2}- i \frac{1}{2}$ and $\bf{d}=-\frac{1}{2}- i \frac{1}{2}, \bf{i}=\sqrt{-1}.$ Finally uniting $\Lambda(C_i), i=1,\hdots,37$ the eigenvalue region of $\pds_4$ is given by the shaded region in Figure \ref{fig1}.

 \begin{figure}[H]
   \centering
 \includegraphics[width=8 cm]{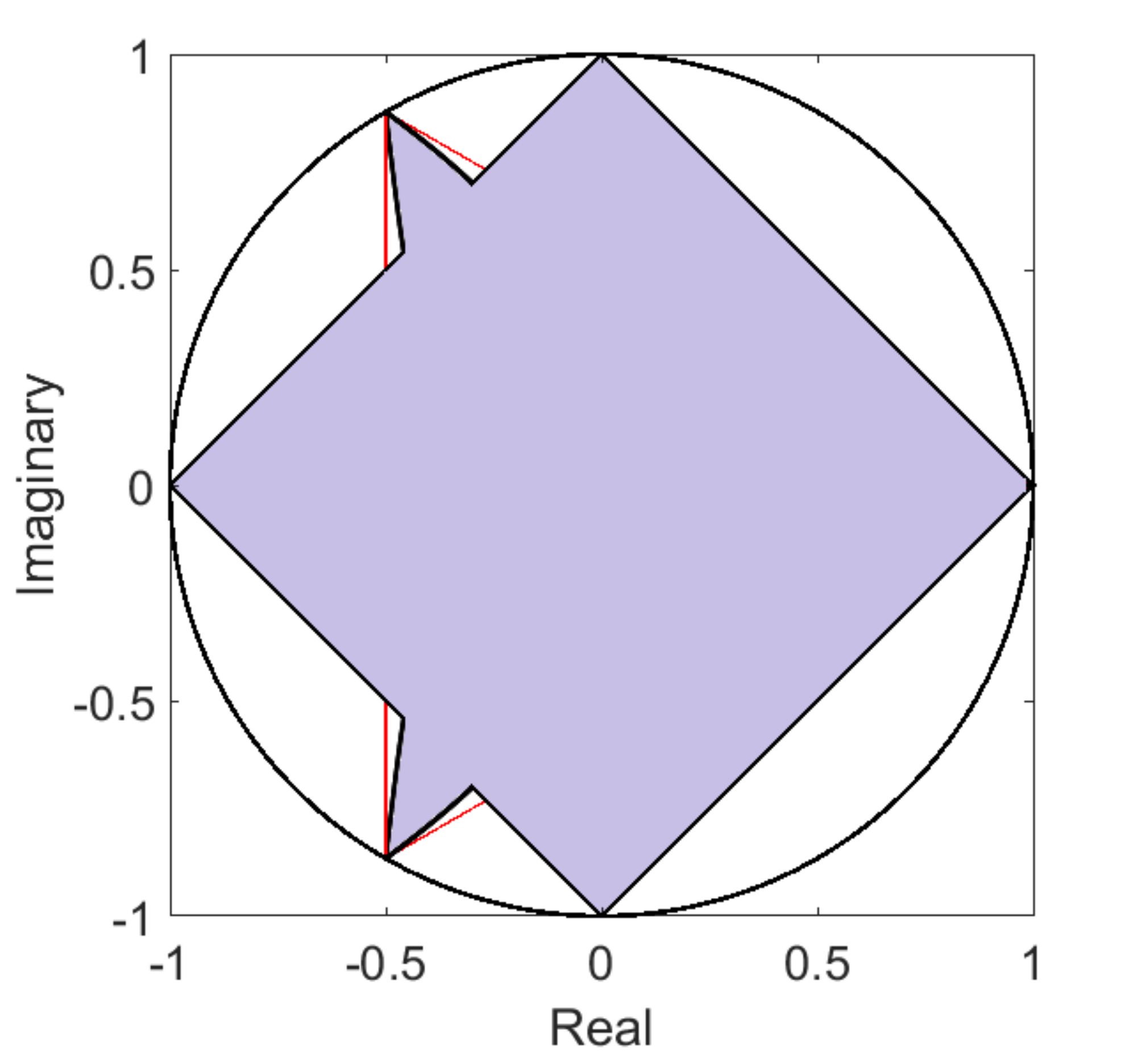}
\caption{Approximate eigenvalue region of $\pds_4,$ given by the shaded region. The red line segments are parts of the boundary of $\Pi_3.$}\label{fig1}
\end{figure}

\subsection{Eigenvalue region for $\pds_4$ is different from  eigenvalue region for $\Omega_4$}

%

In this section we pick two line segments from $\omega_4,$ and show that these lines will not appear in the eigenvalue region of $\pds_4.$ 
Here we  characterize  eigenvalues from the coefficients of characteristic polynomial  in the proof of the Theorem \ref{theorem:line1} and Theorem \ref{theorem:line2}.  First we recall the following definition.

\begin{definition}(Reducible matrix)
If a matrix $A$ can be written as
$$P^TAP=
\bmatrix{
A_1 & 0 \\
A_3 & A_2
},
$$
for some permutation matrix $P,$ where $A_1$ and $A_2$ are square matrices, then $A$ is reducible, otherwise $A$ is called irreducible.
\end{definition}

  Note that, for $-\frac{1}{2} \leq \sigma \leq 1$ and $\sigma + \sqrt{3}|\tau| \leq 1$, the set $\{1, 1, \sigma + \tau i, \sigma - \tau i \}$  is realized as spectrum of the the doubly stochastic matrix
  \begin{eqnarray}\label{DS1}
  \bmatrix{
  \frac{1 + 2 \sigma}{3} & \frac{1 - \sigma - \sqrt{3} \tau}{3} & \frac{1-\sigma + \sqrt{3}\tau}{3} & 0\\
\frac{1-\sigma + \sqrt{3} \tau}{3} & \frac{1 + 2 \sigma}{3} & \frac{1 - \sigma - \sqrt{3} \tau}{3} & 0 \\
\frac{1 - \sigma - \sqrt{3} \tau}{3} & \frac{1- \sigma + \sqrt{3}\tau}{3} & \frac{1 + 2 \sigma}{3} & 0 \\
0 & 0 & 0 & 1
}.
\end{eqnarray}
However in Theorem \ref{t1} we show  that the above set $\{1, 1, \sigma + \tau i, \sigma - \tau i \}$ with $\sigma^2 + \tau^2 <1$ and $\tau \neq 0$, can not be spectrum of any PDS matrix. The bounds for $\sigma$ and $\tau$ are obtained from the region $\Pi_3$.

 \begin{theorem} \label{t1}
 There is no $\pds$ matrix with  spectrum $\{1, 1, \sigma + \tau i, \sigma - \tau i \}$, where $\sigma^2 + \tau^2 < 1$ and $\tau \neq 0$.
 \end{theorem}
 \noindent \pf
  Suppose that $A$ is a permutative doubly stochastic  matrix with the spectrum $\{1, 1, \sigma + \tau i, \sigma - \tau i \}$, with $\sigma^2 + \tau^2 < 1$. If $A$ is reducible, then it is permutationally similar to the matrix
 $\bmatrix{
 A_1 & 0\\
 0 & A_2
 }$,  where both $A_1$ and $A_2$ are square. If $A_1$ and $A_2$ are $2 \times 2$ matrices, then all the eigenvalues of $A$ are real. If $A_1$ is a $3 \times 3$ matrix, then $A$ is a permutation matrix. Thus, all the eigenvalues of the matrix $A$ lie on the unit circle $|z|=1$, and hence we have $\sigma^2 + \tau^2 = 1$.  So $A$ must be irreducible. Now,  by Theorem \ref{PF}, the spectral radius of $A$  is $1$, which is a simple eigenvalue of $A$. Thus, there does not exist any $\pds$  matrix with the spectrum $\{1, 1, \sigma + \tau i, \sigma - \tau i \}$.
   $\hfill{\square}$

 We recall the following theorems regarding the eigenvalues of DS matrices.
\begin{theorem}\cite{Minc1988}\label{t8}
Let $A$ be an $n \times n$ irreducible doubly stochastic matrix. If $A$ has exactly $k$ eigenvalues of unit modulus, then these are the $k^{th}$ roots of unity. In addition, if $k>1,$ then $k$ is a divisor of $n.$
\end{theorem}
 \begin{theorem} \cite{Perfect1965} \label{t3}
 If all the characteristic roots of a doubly stochastic matrix $D$ lie on the unit circle, then $D$ is a permutation matrix.
 \end{theorem}
  \begin{theorem} \cite{Benvenuti2018}\label{t7}
 The set $\omega_4^0$ of points in the complex plane that are eigenvalues of four dimensional doubly stochastic matrices with zero trace is $\omega_4^0=[-1,1]\cup \Pi_4^0.$
 \end{theorem}

The complex number $-\frac{1}{2} +  \frac{1}{2} i$ is in $\Lambda(\pds_4)$.
Since ,  the matrix
$\M\left((0,\frac{1}{2},0,\frac{1}{2});(12)(34),(12),(34)\right)=$
$\bmatrix{
0 & \frac{1}{2} & 0 & \frac{1}{2}\\
\frac{1}{2} & 0 & \frac{1}{2} & 0 \\
\frac{1}{2} & 0 & 0 & \frac{1}{2}\\
0 & \frac{1}{2} & \frac{1}{2} & 0
}$ is  doubly stochastic and  permutative  with spectrum $\left\{1, 0, -\frac{1}{2} +  \frac{1}{2} i,-\frac{1}{2} -  \frac{1}{2} i \right\}$.

The following theorem shows that $-\frac{1}{2} +  \frac{\sqrt3}{2} i$ is in $\Lambda(\pds_4).$
\begin{theorem}
The complex number $-\frac{1}{2} +  \frac{\sqrt3}{2} i$ is an eigenvalue of a  matrix $A \in \Omega_4$ if and only if $A$ is a permutation matrix corresponds to the permutation of cycle type $(3)+(1)$.
\end{theorem}
\noindent \pf Let $-\frac{1}{2} +  \frac{\sqrt3}{2} i$ be an eigenvalue of a DS matrix $D$ of order $4$. Then $D$ has three eigenvalues $1, -\frac{1}{2} +  \frac{\sqrt3}{2} i$ and $-\frac{1}{2} -  \frac{\sqrt3}{2} i$ of maximum modulus. If possible let $D$ be irreducible. Then, by Theorem \ref{t8}, other  eigenvalue should be equals to  $-1$.  Since trace of $D$ is non negative, $-1$ is not an eigenvalue of $D$. So  $D$ is reducible and the spectrum is $\left\{1, 1, -\frac{1}{2} +  \frac{\sqrt3}{2}i, -\frac{1}{2} -  \frac{\sqrt3}{2} i\right\}.$ Hence, by Theorem \ref{t3}, $D$ is a permutation matrix similar to
$\left(
  \begin{array}{c|c}
    P_{\sigma} & 0 \\
    \hline
    0 & 1 \\
  \end{array}
\right)$, where $\sigma \in S_3$ is of cycle type $(3)$. The proof of the converse part is easy.$\hfill{\square}$

From the above theorem it is clear that $-\frac{1}{2} +  \frac{\sqrt3}{2} i$ is an eigenvalue of a matrix from $PDS_4.$
In the next theorem, we shall show that the   $\omega_4 \setminus \Lambda(\pds_4)$ contains the  line segment $-\frac{1}{2} +  y i$  for $\frac{1}{2} < y < \frac{\sqrt3}{2}$(red colored vertical line in Figure \ref{fig1}).
\begin{theorem} \label{theorem:line1}
 Any point on the line segment $-\frac{1}{2} +  y i,$ for  $\frac{1}{2} < y < \frac{\sqrt3}{2}$, can not be an eigenvalue of any  PDS matrix of order $4$.
  \end{theorem}
  \noindent \pf
  Let $D_y\in \pds_4$ with $-\frac{1}{2} +  y i, \frac{1}{2} < y < \frac{\sqrt3}{2},$ as an eigenvalue. Let the spectrum of $D_y$ be
 $\{1, a, -\frac{1}{2} +  y i, -\frac{1}{2} -  y i \}$. By Theorem \ref{t1}, $a = 1$ is not possible.  If $a = 0$, then $D_y$ is a trace zero doubly stochastic matrix, and hence by Theorem \ref{t7} $-\frac{1}{2} +  y i$ should lie in $\Pi_4^{0}$, which is not possible. Thus, we can assume that the spectrum of $D_y$ is $\{1, a, -\frac{1}{2} +  y i, -\frac{1}{2} -  y i \}$, where $0 < a < 1$.
Then the characteristic polynomial of $D_y$ is
 \begin{equation}
 \begin{aligned}
 &q_y(x)=x^4 - a x^3 + \left(y^2 - \frac{3}{4}\right) x^2 + \left(a-(a+1)\left(\frac{1}{4} + y^2\right)\right) x + a \left(\frac{1}{4} + y^2\right),
  \end{aligned}
  \end{equation}
where $0 <a < 1,$ and $ \frac{1}{2} < y < \frac{\sqrt3}{2}.$ Observe that  $-\frac{1}{2} + y i$ for $\frac{1}{2} < y < \frac{\sqrt3}{2},$ is not an eigenvalue of a generic matrix which represents $\C_1$ and $\C_2$. If $D_y$ is in $\C_4$, then $a=1$, which is not possible. Also note that matrix classes $\C_i$ from $\T_4,\T_3$ and $\T_2$ types except the classes $\C_{15},\C_{16},\C_{20},\C_{23},$ and $\C_{25}$ have eigenvalues explicitly in the square ${\bf{abcd}}$. Then the following cases arise. We emphasize that in all the expressions involving $y$ below, the range of $y$ is given by $\frac{1}{2} < y < \frac{\sqrt3}{2}.$

\textbf{Case I}
Suppose $D_y \in \C_3.$ Then by comparing the coefficients of $x^3, x^2, x, 1$ in the polynomials  $q_y(x)$ and $\mathcal{P}_3(x)$, we obtain
 \begin{align}
 c_1 + 2 c_2 + c_3 &=a,\label{e26}\\
 2 c_1 c_2-2 c_1 c_4+2 c_2 c_3-2 c_3 c_4 &= y^2 - \frac{3}{4},\label{e27}\\
 -c_1^3+c_1^2 c_3+2 c_1^2 c_4-c_1 c_2^2-4 c_1 c_2 c_3+2 c_1 c_2 c_4+c_1 c_3^2 \nonumber \\
 -c_1 c_4^2+2 c_2^3-c_2^2 c_3+2 c_2 c_3 c_4-2 c_2 c_4^2-c_3^3+2 c_3^2 c_4-c_3 c_4^2 & = a-(a+1)\left(\frac{1}{4} + y^2\right),
 \label{e28}\\
 c_1^4-4 c_1^2 c_2 c_4-2 c_1^2 c_3^2+4 c_1 c_2^2 c_3+4 c_1 c_3 c_4^2-c_2^4+2 c_2^2 c_4^2-4 c_2 c_3^2 c_4+c_3^4-c_4^4 &= a \left(\frac{1}{4} + y^2\right),\label{e29}
 \end{align}
   where $0 <a < 1$,
 and
 \begin{equation} \label{i1}
 \frac{1}{2} < y < \frac{\sqrt3}{2}.
 \end{equation}
 Also, we have
 \begin{equation} \label{e30}
 c_1 + c_2 + c_3 + c_4 = 1,~c_i \geq 0~ \forall ~i.
 \end{equation}
 Solving (\ref{e26}) and (\ref{e30}) it follows that
 \begin{equation} \label{e31}
\bmatrix{
c_1 \\
c_2 \\
c_3 \\
c_4
}
=
\bmatrix{
-c_3 - 2 c_4 + 2 - a \\
c_4 + a - 1 \\
c_3 \\
c_4
}, ~\mathrm{and}
 \end{equation}
 \begin{equation} \label{e32}
 1 > c_4 \geq (1-a) > 0.
 \end{equation}
Using (\ref{e27}) and (\ref{e31}),  we have
 \begin{align}
 c_4 &= 1-\frac{1}{2} a-\frac{\left(\frac{3}{4} - y^2\right)}{4 (1-a)},\label{e34}\\
 c_2 &= \frac{1}{2} a- \frac{\left(\frac{3}{4} - y^2\right)}{4 (1-a)},~\mathrm{and}\label{e35}\\
 c_1 &= \frac{\left(\frac{3}{4} - y^2\right)}{2 (1-a)}-c_3.\label{e36}
 \end{align}
Thus (\ref{e32}), (\ref{e34}) and (\ref{e36}) yield
\begin{equation} \label{e33}
 1 > a \geq \frac{\frac{3}{4}-y^2}{2 (1-a)}\geq c_3 \geq 0,
 \end{equation}
 and
 \begin{equation} \label{e37}
 \frac{1}{2} - \frac{\sqrt{8 y^2 - 2}}{4} \leq a \leq \frac{1}{2} + \frac{\sqrt{8 y^2-2}}{4}.
 \end{equation}
 By substituting the values of $c_1, c_2, c_4$ in (\ref{e28}) and in (\ref{e29}) we obtain,
 $h(a,y, c_3) =0$ where
 \begin{align}\label{e38}
 \begin{split}
 h(a,y, c_3) =& \frac{1}{256(a-1)^3}(-1024 a^2 c_3^2 y^2+512 a c_3y^4 - 64 y^6 + 256 a^5 + 1024 a^3 c_3^2+ 256 a^3 y^2  \\
                &  - 512 a^2 c_3 y^2+ 2048 a c_3^2 y^2 + 64 a y^4 - 512 c_3 y^4 - 1024 a^4 - 2304 a^2 c_3^2 - 768 a^2 y^2  \\
                &  + 256 a c_3 y^2- 1024 c_3^2 y^2+ 80 y^4 + 1600 a^3 + 384 a^2 c_3 + 1536 a c_3^2 + 672 a y^2    \\
                &  + 256 c_3 y^2- 1216 a^2- 480 a c_3-256 c_3^2-268 y^2 + 484 a + 96 c_3 - 73).
                \end{split}
                \end{align}
                That is,  $h(a,y,c_3)=0$ for some $(a,y, c_3)$ which satisfies $(\ref{i1})$, $(\ref{e33})$ and $(\ref{e37}).$ Then the roots $c_3^{(1)}(a,y)$ and $c_3^{(2)}(a,y)$ of the equation  $ h(a,y, c_3)=0$ (considered as a quadratic equation in $c_3$) are given by:
\begin{eqnarray} \label{a1}
c_3^{(1)}(a,y)=\frac{1}{16} \frac{(3-4 y^2) (-4 y^2+4 a-1) - 8 \sqrt{(1-a)^3 (4 a^2+4 y^2-4 a+1) (-4y^2+4a-1)}}{(-4y^2+4a-1)(1-a)}
\end{eqnarray}
and
\begin{eqnarray} \label{a2}
c_3^{(2)}(a,y)=\frac{1}{16} \frac{(3-4 y^2) (-4 y^2+4 a-1) + 8 \sqrt{(1-a)^3 (4 a^2+4 y^2-4 a+1) (-4y^2+4a-1)}}{(-4y^2+4a-1)(1-a)}.
\end{eqnarray}
Since $(1-a)$ and $(4 a^2+4 y^2-4 a+1)=(2a-1)^2 + 4 y^2$ are positive for all $a<1$, so  $c_3^{(1)}(a,y)$ and $c_3^{(2)}(a,y)$  are real if and only if $a \geq \frac{4y^2+1}{4}$. If $a=\frac{4y^2+1}{4},$  then $h(\frac{4y^2+1}{4},y, c_3)=\frac{(1+4y^2)^2}{16}$, which is positive for any value of $y$. Hence, for the real roots, we have $a>\frac{4y^2+1}{4}.$ Now,
\begin{equation*}
\frac{4y^2+1}{4}-\left(\frac{1}{2}+\frac{\sqrt{8y^2-2}}{4}\right)=\frac{\sqrt{8y^2-2}}{8} \left(\sqrt{8y^2-2}-2 \right) <0
\end{equation*}
and
\begin{equation*}
\frac{4y^2+1}{4}-\left(\frac{1}{2}-\frac{\sqrt{8y^2-2}}{4}\right)=\frac{\sqrt{8y^2-2}}{8} \left(\sqrt{8y^2-2}+2 \right) >0,
\end{equation*}
for $\frac{1}{2} < y < \frac{\sqrt3}{2}.$
Hence by $(\ref{e37})$,  we have,
\begin{eqnarray} \label{i2}
\frac{4y^2+1}{4}<a \leq \frac{1}{2}+\frac{\sqrt{8y^2-2}}{4}.
\end{eqnarray}

Now,
\begin{align*}
&\frac{d}{da}\{c_3^{(1)}(a,y)\}=\frac{4 (1-a)^3 p(a)+(3-4y^2)(-4y^2+4a-1)\sqrt{(1-a)^3(4a^2+4y^2-4a+1)(-4y^2+4a-1)}}{16(a-1)^2(-4y^2+4a-1)\sqrt{(1-a)^3(4a^2+4y^2-4a+1)(-4y^2+4a-1)}},
\end{align*}
for each fixed $y$, where $p(a)=(-48a^2y^2-16y^4+32a^3+64ay^2-44a^2-8y^2+16a-1).$
Hence $\frac{2}{3}$ and $\frac{1+4y^2}{4}$ are the zeros of $p'(a),$ where
$$
\frac{d}{da}p(a)=8 (3a-2)(-4y^2+4a-1).
$$
 Now
 \begin{align*}
 p''\left(\frac{1+4y^2}{4}\right)=8(12y^2-5)~~~&>0~\mathrm{if}~y>\sqrt{\frac{5}{12}}\\
                                     &<0~\mathrm{if}~y<\sqrt{\frac{5}{12}},
 \end{align*}
 and
 \begin{align*}
  p''\left(\frac{2}{3}\right)=-8(12y^2-5)~~~&>0~\mathrm{if}~y<\sqrt{\frac{5}{12}}\\
                                     &<0~\mathrm{if}~y>\sqrt{\frac{5}{12}}.
 \end{align*}
 So,  if $y > \sqrt{\frac{5}{12}}$, then $\frac{1+4y^2}{4}$ is the local minimum and $\frac{2}{3}$ is the local maximum, and if $y<\sqrt{\frac{5}{12}}$, then $\frac{1+4y^2}{4}$ is the local maximum, and $\frac{2}{3}$ is the local minimum.
 It can be shown that at $y=\sqrt{\frac{5}{12}},p(a)$ is a positive function  for all $a \geq \frac{2}{3}.$

 Note that at $y=\sqrt{\frac{5}{12}},$ $p(a)$ becomes $p_1(a)=-64a^2-\frac{64}{9}+32 a^3+\frac{128}{3}a.$ Now we consider the function $p_1(a)$ over $\frac{2}{3}\leq a.$ It is easy to see that $p_1'''(a)=192,$ and hence $p_1''(a)$ is an increasing function so that $p_1''(a)\geq p_1''(\frac{2}{3})=0$ for all $a\geq \frac{2}{3}.$ Now $p_1'(a)\geq p_1'(\frac{2}{3})=0,$ since $p_1'(a)$ is again an increasing function of $a\geq \frac{2}{3}.$ At last we have $p_1(a)$ is an increasing function of $a$ for all $a \geq \frac{2}{3},$ and hence $p_1(a)\geq p_1(\frac{2}{3})=\frac{64}{27}.$ Which yields $p_1(a)$ is a positive function for all $a\geq \frac{2}{3}.$

 Since $p\left(\frac{1+4y^2}{4}\right)=\left(\frac{3}{4}-y^2\right)(4y^2+1)^2 >0$ and $ p\left(\frac{2}{3}\right)=\left( 4y^2-\frac{5}{3}+\frac{8}{3\sqrt3}\right)\left( \frac{5}{3}+\frac{8}{3\sqrt3}-4y^2\right)>0,$ the polynomial $p(a)$ has positive local minimum value for each $y\in\left(\frac{1}{2},\frac{\sqrt3}{2}\right).$ At $a=\frac{1}{2}+\frac{\sqrt{8y^2-2}}{4},$ we have  $$p\left(\frac{1}{2}+\frac{\sqrt{8y^2-2}}{4}\right)=20 y^2+8y^2\sqrt{8y^2-2}-40y^4-2\sqrt{8y^2-2}-\frac{1}{2}.$$

  Let $p_2(y)=20 y^2+8y^2\sqrt{8y^2-2}-40y^4-2\sqrt{8y^2-2}-\frac{1}{2}.$ Now $p_2'(y)=-\frac{8y}{\sqrt{8y^2-1}}p_3(y),$ where $p_3(y)=(4y^2-1)(5\sqrt{8y^2-2}-6).$ Hence
  \begin{align*}
   p_3(y) & \leq 0 ~\mathrm{for}~\frac{1}{2} \leq y < \frac{\sqrt{43}}{10}\\
          & >0 ~\mathrm{for}~ \frac{\sqrt{43}}{10} < y \leq \frac{\sqrt3}{2}.
  \end{align*}
  So that
  \begin{align*}
   p_2'(y) & \geq 0 ~\mathrm{for}~\frac{1}{2} \leq y < \frac{\sqrt{43}}{10}\\
          & <0 ~\mathrm{for}~ \frac{\sqrt{43}}{10} < y \leq \frac{\sqrt3}{2}.
  \end{align*}
  Now $p_2(y)\geq p_2(\frac{1}{2})=2~\mathrm{for}~\frac{1}{2}\leq y  < \frac{\sqrt{43}}{10}$ and $p_2(y) > p_2(\frac{\sqrt3}{2})=0~\mathrm{for}~\frac{\sqrt{43}}{10} < y < \frac{\sqrt3}{2}.$ Also $p_2(\frac{\sqrt{43}}{10})=\frac{304}{125},$ and hence $p_2(y)>0~\forall~y\in (\frac{1}{2},\frac{\sqrt3 }{2}).$

So we obtain $p\left(\frac{1+4y^2}{4}\right)>0$ and $p\left(\frac{1}{2}+\frac{\sqrt{8y^2-2}}{4}\right)>0.$
 Hence the global minimum of $p(a)$ is positive in the compact set  $\frac{4y^2+1}{4}\leq a \leq \frac{1}{2}+\frac{\sqrt{8y^2-2}}{4}$ for each fixed $y$ which satisfies $(\ref{i1}).$
 Thus $ c_3^{(1)}(a,y)$ is a strictly increasing function for each fixed $y$. Now,
 \begin{align}\label{e77}
 c_3^{(1)}\left(a,y\right)=&\frac{-\sqrt{(-4 y^2+4 a-1)}}{16(-4y^2+4a-1)(1-a)}{\left((4 y^2-3) \sqrt{(-4 y^2+4 a-1)} + 8 \sqrt{(1-a)^3 (4 a^2+4 y^2-4 a+1)}\right)}.
 \end{align}
 Let $g(a)={\left((4 y^2-3) \sqrt{(-4 y^2+4 a-1)} + 8 \sqrt{(1-a)^3 (4 a^2+4 y^2-4 a+1)}\right)}.$ Then
 \begin{align*}
 g\left(\frac{1}{2}+\frac{\sqrt{8y^2-2}}{4}\right)=\left(4y^2-3\right)\sqrt{-4y^2+1+\sqrt{8y^2-2}}+\frac{1}{\sqrt 2}\sqrt{\left(-2+\sqrt{8y^2-2}\right)^3\left(1-12y^2\right)}
 \end{align*}
 is a continuous function of $y$ satisfying (\ref{i1}), and has only zeros at $y=\pm\frac{\sqrt 3}{2}.$ Hence for $\frac{1}{2} < y < \frac{\sqrt3}{2},$ $g\left(\frac{1}{2}+\frac{\sqrt{8y^2-2}}{4}\right)$ is either positive or negative. Since $g(\frac{2}{3})=0.0285$ (approx), hence $g(y)>0$ for all $y$ satisfying (\ref{i1}).
 From (\ref{e77}), it is clear that $c_3^{(1)}\left(\frac{1}{2}+\frac{\sqrt{8y^2-2}}{4},y\right)$ is negative, which implies that $c_3^{(1)}(a,y) < 0$, for all $a$ in  (\ref{i2}). Which is not possible as $c_3 \geq 0.$

Clearly, $c_3^{(2)}(a,y) > 0$ for $a > \frac{4y^2+1}{4}.$ Now, $c_3^{(2)}(a,y)-\frac{\frac{3}{4}-y^2}{2(1-a)}=-c_3^{(1)}(a,y),$
 which implies $c_3^{(2)}(a,y) > \frac{\frac{3}{4}-y^2}{2(1-a)}$, for all $a$ in  (\ref{i2}). Which is again not possible since $c_3 \leq \frac{\frac{3}{4}-y^2}{2(1-a)}.$
Hence, there does not exist any triplet $(a,y,c_3)$ satisfying (\ref{i1}), (\ref{e33}) and (\ref{e37}) such that $h(a,y,c_3)=0$.

Now we show that $-\frac{1}{2}+iy,\frac{1}{2}<y<\frac{\sqrt3}{2}$ is not an eigenvalue of a matrix which belongs to the classes $\C_{15},\C_{16},\C_{20},\C_{23}$ and $\C_{25}$.

\textbf{Case II:}
If possible let $D_y$ be a generic matrix which represents $\C_{15}.$ Then we compare the coefficients of $x^3, x^2, x, 1$ in the polynomials  $q_y(x)$ and $\mathcal{P}_{15}(x)$ and obtain
 \begin{align*}
2 c_1-\frac{3}{2} &=-a, \\
\frac{1}{2}-2 c_1 &=y^2 - \frac{3}{4}, \\
18 c_1-72 c_1^2-\frac{3}{2}+96 c_1^3 &=a-(a+1)\left(\frac{1}{4} + y^2\right), \\
72 c_1^2-96 c_1^3-18 c_1+\frac{3}{2} &=a \left(\frac{1}{4} + y^2\right),~\mathrm{where}~ \frac{1}{6}\leq c_1 \leq \frac{1}{3}.
\end{align*}
The first two equations of the above system give $c_1= \frac{3}{4}-\frac{1}{2}a$ and $a=y^2+\frac{1}{4},$ where
$\frac{5}{6} \leq a < 1.$ Now the left hand side of third equation yields $12(1-a)^3(>0),$ whereas the right hand side gives $-a^2(<0),$ which is not possible.

\textbf{Case III:}
Let $D_y \in \C_{16}.$ Then comparing the coefficients of $x^3, x^2, x, 1$ in the polynomials  $q_y(x)$ and $\mathcal{P}_{16}(x)$ we obtain
\begin{align*}
2 c_1-\frac{3}{2} &=-a,\\
-\frac{1}{2}+6c_1-16 c_1^2 &=y^2 - \frac{3}{4}, \\
-26c_1+88c_1^2+\frac{5}{2}-96 c_1^3 &=a-(a+1)\left(\frac{1}{4} + y^2\right),\\
-72 c_1^2+96 c_1^3+18 c_1-\frac{3}{2} &=a \left(\frac{1}{4} + y^2\right),~\mathrm{where}~ \frac{1}{6}\leq c_1 \leq \frac{1}{3}.
\end{align*}
So, we have  $c_1= \frac{3}{4}-\frac{1}{2}a,$ where $\frac{5}{6} \leq a < 1.$ The second equation gives $-4 a^2+9 a-4=y^2 + \frac{1}{4}.$ Then, the left hand side of the third equation yields $12(a-\frac{2}{3})(a-1)^2(>0)$ and right hand side gives $4a^3-5a^2-4a+4=4(1-a)(1-a^2)-a^2\leq \frac{4}{6}(1-a^2)-a^2<0$ for all $a \in \left[ \frac{5}{6}, 1\right).$
 Hence, a contradiction.

\textbf{Case IV:}
  Assigning $D_y$ to be in $\C_{20}$ and comparing the coefficients of $x^3, x^2, x, 1$ in the polynomials  $q_y(x)$ and $\mathcal{P}_{20}(x)$, we obtain
 \begin{align*}
 -\frac{5}{2}+6c_1 & =-a,\\
 -6c_1+\frac{3}{2} & = y^2 - \frac{3}{4},\\
 72c_1^2-96c_1^3-18c_1+\frac{3}{2} & =a-(a+1)\left(\frac{1}{4} + y^2\right),\\
 96c_1^3-72c_1^2+18c_1-\frac{3}{2} & =a \left(\frac{1}{4} + y^2\right), \mathrm{where}~ \frac{1}{6}\leq c_1 \leq \frac{1}{3}.
 \end{align*}
 So, we have $c_1= \frac{5}{12}-\frac{1}{6}a,$  where $\frac{1}{2} \leq a < 1.$ Second equation yields $a=y^2+\frac{1}{4}.$ From third equation we have
 \begin{align*}
 4 a^3-3 a^2+12 a-4=0.
 \end{align*}
 However,
 $$
 4a^3-3a^2+12a-4=4a^2(a-1)+a^2+12 a-4
                 \geq 4(a-1)+a^2+12a-4
                 \geq \frac{1}{4},
 $$
 for $\frac{1}{2} \leq a < 1$.
 Hence $4a^3-3a^2+12a-4\neq0, $ for $a \in [\frac{1}{2},1).$

\textbf{Case V:}
  Let $D_y$ represent $\C_{23}.$ Then, by comparing the coefficients of $x^3, x^2, x, 1$ in the polynomials  $q_y(x)$ and $\mathcal{P}_{23}(x)$, we have the following:
  \begin{align*}
  4c_1&=a \\
  8c_1^2-\frac{1}{2} & =y^2 - \frac{3}{4}\\
  96c_1^3-80c_1^2+22c_1-2 & =a-(a+1)\left(\frac{1}{4} + y^2\right)\\
  72c_1^2-96c_1^3-18c_1+\frac{3}{2} & =a \left(\frac{1}{4} + y^2\right),~\mathrm{where}~ \frac{1}{6}\leq c_1 \leq \frac{1}{3}.
   \end{align*}
  Hence, we have $c_1 = \frac{1}{4}a,$  where $\frac{2}{3} \leq a < 1.$ From the second equation we will have $y^2+\frac{1}{4}=\frac{1}{2}(1+a^2).$
  Using these, third equation becomes
 \begin{align*}
 -4a^3+9a^2-10a+3=0.
 \end{align*}
 However,
 $$
 -4a^3+9a^2-10a+3=4a^2(1-a)+5(1-a)^2-2
                 \leq \frac{4}{3}a^2-\frac{13}{9}
                 < 0 ,
 $$
for $\frac{2}{3} \leq a < 1$. So, $-4a^3+9a^2-10a+3 \neq 0,$ for $\frac{2}{3} \leq a < 1. $

 \textbf{Case VI:}
 Let $D_y \in \C_{25},$ then by comparing the coefficients of $x^3$ in the polynomials  $q_y(x)$ and $\mathcal{P}_{25}(x)$, we get
  $a=1.$ But this is not our case.

 Hence $-\frac{1}{2} +  y i,$ where $\frac{1}{2} < y < \frac{\sqrt 3}{2}$ is not an eigenvalue of the matrices in $\bigcup_{i=1}^{37}\C_i,$ and the proof follows. $\hfill{\square}$

 We emphasize that $-\frac{1}{2} +  y i,\frac{1}{2} < y < \frac{\sqrt 3}{2}$ are eigenvalues of the DS matrices given in (\ref{DS1}) for $\sigma=-\frac{1}{2}$ and $\tau=y.$

 Let $L = \{z = x+ yi : x + \sqrt3 y =1 ~\mbox{and}~ -\frac{1}{2} < x < \frac{1-\sqrt3}{1+\sqrt3} \}$(red colored inclined line in upper half plane of Figure \ref{fig1}). In the next theorem, we show any point on $L$ can be an eigenvalue of $\Omega_4$ but not of $\pds_4.$

 \begin{theorem} \label{theorem:line2}
 	Any point on the line segment $L$  is not an eigenvalue of any PDS matrix of order $4$.
  \end{theorem}
 \noindent\pf
 Let $D_r$ be a permutative doubly stochastic matrix of order $4$ such that
  $ \left\{1, a, r + \frac{1-r}{\sqrt 3} i, r - \frac{1-r}{\sqrt 3} i \right\}$ is the spectrum of $D_r$, where $-\frac{1}{2} < r < \frac{1-\sqrt3}{1+\sqrt3}$ and $a < 1.$ Then observe that $D_r$ can only belong to $\C_i,i=3,4,15,16,20,23,25.$ We consider the possibility of $D_r$ to be in this classes and show contradictions. The characteristic polynomial of $D_r$ is
 \begin{align*}
&q(x)= x^4 - (1 + a + 2r) x^3 + \left( a (1+ 2r) + \frac{(1+2r)^2}{3}\right) x^2 - \left( a\frac{(1+2r)^2}{3} + r^2 + \frac{(1-r)^2}{3}\right) x \\
&\hspace{1 cm}+ a \left(r^2 + \frac{(1-r)^2}{3} \right),
~~~~~~~\mbox{where}~-\frac{1}{2} < r < \frac{1-\sqrt3}{1+\sqrt3}.
 \end{align*}

  Next we have the following instances.

 \textbf{Case I:}
 Let $D_r \in \C_4.$ Then, by comparing the coefficients of $\mathcal{P}_4(x)$ and $q(x)$, we have
 \begin{eqnarray}\label{trace1}
 & a+2r=0 ~\mathrm{and}~  a(1+2r)+\frac{(1+2r)^2}{3}=0,
 \end{eqnarray}
which implies that  $r=-\frac{1}{2}, \frac{1}{4}.$
This is not possible.

\textbf{Case II:}
If $D_r \in \C_3$, then, by comparing the coefficients of $\mathcal{P}_3(x)$ and $q(x)$, we have
 \begin{align}
  c_1 + 2 c_2 + c_3 =1+a+2r,\label{e60} \\
 2 c_1 c_2-2 c_1 c_4+2 c_2 c_3-2 c_3 c_4 = a(1+2r) +\frac{(1+2r)^2}{3} ,\label{e61} \\
 -c_1^3+c_1^2 c_3+2 c_1^2 c_4-c_1 c_2^2-4 c_1 c_2 c_3+2 c_1 c_2 c_4+c_1 c_3^2-c_1 c_4^2+2 c_2^3-c_2^2 c_3 \nonumber \\
+ 2 c_2 c_3 c_4-2 c_2 c_4^2-c_3^3+2 c_3^2 c_4-c_3 c_4^2 = -a \frac{(1+2r)^2}{3} - (r^2+\frac{(1-r)^2}{3}), \label{e62} \\
 c_1^4-4 c_1^2 c_2 c_4-2 c_1^2 c_3^2+4 c_1 c_2^2 c_3+4 c_1 c_3 c_4^2-c_2^4
 +2 c_2^2 c_4^2-4 c_2 c_3^2 c_4+c_3^4-c_4^4 = a (r^2+\frac{(1-r)^2}{3}),\label{e63}
 \end{align}
 and
 \begin{equation} \label{e65}
 -\frac{1}{2} < r < \frac{1-\sqrt3}{1+\sqrt3}.
 \end{equation}
 Also, we have \begin{equation} \label{e64}
 c_1 + c_2 + c_3 + c_4 = 1,~c_i \geq 0~ \forall ~i.
 \end{equation}
From  (\ref{e60}) and  (\ref{e64}), we obtain
 \begin{eqnarray} \label{e66}
\bmatrix{
c_1 \\
c_2 \\
c_3 \\
c_4
}
=
\bmatrix{
1-c_3 - 2 c_4 - a-2r \\
c_4 + a +2r \\
c_3 \\
c_4
}.
 \end{eqnarray}

From (\ref{e61}) and (\ref{e66}), it follows that
 \begin{equation} \label{e69}
 c_4 = \frac{1}{2}-\frac{1}{2}a-r -  \frac{a(1+2 r)}{4 (a+2 r)}-\frac{(1+2r)^2}{12 (a+2 r)},
 \end{equation}

 \begin{equation} \label{e70}
 c_2 = \frac{1}{2} + \frac{1}{2} a + r - \frac{a (1+2 r)}{4 (a+2 r)} - \frac{(1+2 r)^2}{12 (a+2 r)},
 \end{equation}
 \begin{equation} \label{e71}
 c_1 = \frac{ a (1 + 2 r)}{2 (a+2 r)} + \frac{(1+2 r)^2}{6 (a+2 r)}-c_3.
 \end{equation}

 By substituting the values of $c_1, c_2, c_4$ in (\ref{e62}) and in (\ref{e63}). we get $h(a,r,c_3)=0$,
   where
 \begin{eqnarray*}
 h(a, r, c_3) & = & -\frac{1}{108 (a+2r)^3} \left(216 a^5 r + 1728 a^4 r^2 + 864 a^3 c_3^2 r - 864 a^3 c_3 r^2 + 5688 a^3 r^3 + 4032 a^2 c_3^2- 2880 a^2 c_3 r^3 \right.\\
& &\left. + 9072 a^2 r^4 + 5760 a c_3^2 r^3 - 2688 a c_3 r^4 + 7200 a r^5 + 2304 c_3^2 r^4 - 768 c_3 r^5 + 2368 r^6- 432 a^3 c_3 r\right.\\
& &\left. + 72 a^3 r^2 - 288 a^2 c_3^2 r - 1440 a^2 c_3 r^2 - 360 a^2 r^3 - 1152 a c_3^2 r^2 - 1344 a c_3 r^3 - 1344 a r^4 - 1152 c_3^2 r^3\right.\\
                               & &\left. - 384 c_3 r^4 - 1056 r^5 + 126 a^3 r + 144 a^2 c_3^2 - 144 a^2 c_3 r+ 648 a^2 r^2 + 576 a c_3^2 r - 288 a c_3 r^2+1056 a r^3 \right. \\
                               & &\left. + 576 c_3^2 r^2 + 624 r^4 -72 a^2 c_3 + 54 a^2 r - 192 a c_3 r + 72 a r^2 - 96 c_3 r^2 + 16 r^3 + 9 a^2  \right.\\
                               & &\left.- 24 a c_3+ 30 a r - 48 c_3 r + 12 r^2 + 6 a + 6 r + 1\right).
 \end{eqnarray*}

 Now,  if $h(a,r, c_3)=0$, then either $c_3(a,r) = c_3^{(1)}(a,r)$ or $c_3(a,r) = c_3^{(2)}(a,r)$,  where
\begin{eqnarray*} \label{e73}
&c_3^{(1)}(a, r)=\frac{1}{12} \frac{(2r+1) (2r+3a+1)(6ar+4r^2-2r+1)+6\sqrt{-2r(6ar+4r^2-2r+1)(3a^2+6ar+4r^2-2r+1)(a+2r)^3}}{(a+2r)(6ar+4r^2-2r+1)}
\end{eqnarray*}
and
\begin{eqnarray*} \label{e74}
&c_3^{(2)}(a, r)=\frac{1}{12} \frac{(2r+1) (2r+3a+1)(6ar+4r^2-2r+1)-6\sqrt{-2r(6ar+4r^2-2r+1)(3a^2+6ar+4r^2-2r+1)(a+2r)^3}}{(a+2r)(6ar+4r^2-2r+1)}.
\end{eqnarray*}
Let  $\mathcal{F}(a)= 6ar+4r^2-2r+1$. Then, for each  $r <0$, it is clear that  $\mathcal{F}(a)$ is a strictly decreasing function of $a$.
If  $a=-\frac{1}{6}\frac{4 r^2-2 r +1}{r}$, i.e. $\mathcal{F}(a)=0,$ then
\begin{align*}
h\left(-\frac{1}{6}\frac{4 r^2-2 r +1}{r},r,c_3\right)=&-\frac{1}{18} \frac{(16 r^4 - 16 r^3 + 12 r^2 - 4 r + 1)}{r} \\
                                           =&-\frac{1}{18r}(4 r^2-2 r+1)^2 >0.
\end{align*}
Hence, if  $h(a,r,c_3)=0,$ then $ a\neq-\frac{1}{6}\frac{4 r^2-2 r +1}{r}.$

 Further, from (\ref{e71}) it follows that
   \begin{equation} \label{e68}
 1 >\frac{2 a (1 + 2 r)}{4 a+8 r} + \frac{2(1+2 r)^2}{12 a+24 r}\geq c_3 \geq 0.
 \end{equation}
Since $c_4-c_2 = -a-2r$, we have  $-1< a+2r < 1$ with $a+2r \neq 0.$ Otherwise for $a+2r=0,$ consider (\ref{trace1}). It is clear that
\begin{eqnarray*} \label{e67}
& 1 > 1-(a+2r)\geq c_4 \geq 0 > -(a+2r), ~\mathrm{if}~ (a+2r)>0, \\
& 1-(a+2r)>1\geq c_4\geq -(a+2r) > 0, ~\mathrm{if}~ (a+2r)<0,
\end{eqnarray*}
 Thus we have the following two subcases.

 \noindent\textbf{Subcase 1}: If $(a+2r)<0$, then
 \begin{align} \label{e72(a)}
 \begin{split}
&-\frac{3}{2}r - \frac{1}{4} - \frac{1}{12} \sqrt{-156 r^2 - 84 r + 33} \leq a  \leq-\frac{5}{2} r - \frac{3}{4} + \frac{1}{12} \sqrt{228 r^2 + 156 r+57} \\ &~\mathrm{where}~-\frac{1}{2} < r < \frac{1-\sqrt3}{1+ \sqrt3}.
\end{split}
\end{align}
So that
$$
\mathcal{F}\left(-\frac{5}{2} r - \frac{3}{4} + \frac{1}{12} \sqrt{228 r^2 + 156 r+57}\right) \leq \mathcal{F}(a),
$$
for all $a$ satisfying (\ref{e72(a)}).
Besides
$$
\mathcal{F}\left(-\frac{5}{2} r - \frac{3}{4} + \frac{1}{12} \sqrt{228 r^2 + 156 r+57}\right)= -11r^2-\frac{13}{2}r+\frac{r}{2}\sqrt{228r^2+156r+57}+1,
$$
 has the  zeros  $-\frac{1}{2}, \frac{1}{4},$ and it is a  continuous function of $r$ in the range given by (\ref{e65}). At $r=-\frac{4}{10},$ the value of this function is $\frac{46}{25}-\frac{1}{25} \sqrt {777} > 0$ and hence $\mathcal{F}(a)>0$  all $a$ and $r$ satisfying (\ref{e72(a)}). Since
$$
3a^2+6ar+4r^2-2r+1=3(a+r)^2 + (r-1)^2 > 0,
$$  neither  $c_3^{(1)}(a,r)$  nor $c_3^{(2)}(a,r)$ is real.

\noindent\textbf{Subcase 2}: If $(a+2r)>0$, then
\begin{align} \label{e72(b)}
\begin{split}
&-\frac{5}{2} r + \frac{1}{4} - \frac{1}{12} \sqrt{228 r^2 + 12 r-15} \leq a \leq -\frac{5}{2} r + \frac{1}{4} + \frac{1}{12} \sqrt{228 r^2 + 12 r-15} \\ &~\mathrm{where}~-\frac{1}{2} < r \leq -\frac{1}{38}-\frac{2}{19}\sqrt6.
\end{split}
 \end{align}
 Then the function
$$-\frac{5}{2} r + \frac{1}{4} - \frac{1}{12} \sqrt{ 228 r^2 + 12 r - 15} +\frac{1}{6r}(4r^2-2 r +1),$$ has zeros only at $r=-\frac{1}{2}$ and at $r=\frac{1}{4}$, and has negative value $-\frac{\sqrt{417}}{60} + \frac{7}{30}$ at $r=-\frac{4}{10}$. Hence,
$$-\frac{5}{2} r + \frac{1}{4} - \frac{1}{12} \sqrt{ 228 r^2 + 12 r - 15} <-\frac{1}{6r}(4r^2-2 r +1),
$$
for $-\frac{1}{2} < r \leq -\frac{1}{38}-\frac{2}{19}\sqrt6.$

Observe that the function
$$
-\frac{5}{2} r + \frac{1}{4} + \frac{1}{12} \sqrt{ 228 r^2 + 12 r - 15} +\frac{1}{6r}\left(4r^2-2 r +1\right)
$$
 has zeros at $\frac{1}{16}-\frac{\sqrt{33}}{16} $ and at $\frac{1}{16}+\frac{\sqrt{33}}{16},$
and hence the values of it is positive or negative in the range $-\frac{1}{2} < r \leq \frac{1}{16}-\frac{\sqrt{33}}{16}$. If $r = -\frac{2}{5},$ then  $-\frac{1}{2} < -\frac{2}{5} < \frac{1}{16}-\frac{\sqrt{33}}{16},$  and it becomes positive, and hence it is positive for all $r$ in $\left(-\frac{1}{2}, \frac{1}{16}-\frac{\sqrt{33}}{16}\right]$. For $r > \frac{1}{16}-\frac{\sqrt{33}}{16},$ the function is negative, and hence
$$
a \leq -\frac{5}{2} r + \frac{1}{4} + \frac{1}{12} \sqrt{ 228 r^2 + 12 r - 15} < -\frac{1}{6r}\left(4r^2-2 r +1\right).
$$
However this is not possible since $a > -\frac{1}{6r}\left(4r^2-2 r +1\right).$

 Consequently,
\begin{equation}\label{e75}
-\frac{1}{6r}\left(4r^2-2 r +1\right) < a \leq -\frac{5}{2} r + \frac{1}{4} + \frac{1}{12} \sqrt{ 228 r^2 + 12 r - 15},
\end{equation}
for $-\frac{1}{2} < r \leq \frac{1}{16}-\frac{\sqrt{33}}{16}$ (since $-\frac{1}{6r}\left(4r^2-2 r +1\right) \neq a$).
Since $\mathcal{F}(a)$ is a strictly decreasing function of $a,$ where   $-\frac{1}{2} < r \leq \frac{1}{16}-\frac{\sqrt{33}}{16}$. Hence, we have
$$
\mathcal{F}\left( -\frac{1}{6r}\left(4r^2-2 r +1\right)\right) > \mathcal{F}(a),
$$
due to (\ref{e75}). So $\mathcal{F}(a)$ is negative, and hence there is no real $c_3^{(1)}(a,r)$  and  $c_3^{(2)}(a,r)$ are not real when $a+2r>0.$

\textbf{Case III:}
If possible let $D_r \in \C_{15}.$  By comparing the coefficients of $x^3, x^2, x, 1$ in the polynomials  $q(x)$ and $\mathcal{P}_{15}(x)$, we obtain
 \begin{align*}
2 c_1-\frac{3}{2} &=-(1+a+2r), \\
 \frac{1}{2}-2 c_1 &=a(1+2r)+\frac{1}{3}(1+2r)^2, \\
18 c_1-72 c_1^2-\frac{3}{2}+96 c_1^3 &=-\left(\frac{1}{3}a(1+2r)^2+r^2+\frac{1}{3}(1-r)^2\right), \\
72 c_1^2-96 c_1^3-18 c_1+\frac{3}{2} &=a \left(r^2+\frac{(1-r)^2}{3}\right),~ \mathrm{where}~ \frac{1}{6}\leq c_1 \leq \frac{1}{3}.\\
\end{align*}
The first equation of the above system gives $c_1= \frac{1}{4}-\frac{1}{2}a-r,$ so that  $ -\frac{1}{6}-2r \leq  a \leq \frac{1}{6}-2r.$  Second equation implies $2ra+\frac{4}{3}r^2-\frac{2}{3}r+\frac{1}{3}=0$ and hence $a=-\frac{1}{6} \frac{4r^2-2r+1}{r},$ where $-\frac{1}{2}<r\leq-\frac{1}{16}-\frac{1}{16}\sqrt{33}.$
  Setting this value of $a$ and $c_1$ in the third equation we have
  $$-\frac{1}{18r^3}(528r^6+368r^5-84r^4-92r^3+13r^2+6r-1)=0.$$

  Let $f_1(r)=\frac{528r^6+368r^5-84r^4-92r^3+13r^2+6r-1}{18},$ for $-\frac{1}{2}<r\leq-\frac{1}{16}-\frac{1}{16}\sqrt{33}.$
Then $f_1''''(r)=10560 r^2+\frac{7360}{3}r-112>0$ implies $f_1'''(r)$ is a strictly increasing function, so that $f_1'''(r)<f_1'''(-\frac{1}{16}-\frac{1}{16}\sqrt{33})<0.$ Hence $f_1''(r)>f_1''(-\frac{1}{16}-\frac{1}{16}\sqrt{33})>0$ over the said range of $r.$ Now $f_1'(r)<f_1'(-\frac{1}{16}-\frac{1}{16}\sqrt{33})<0$ implies $f_1(r)$ is a strictly decreasing function of $r.$ So that, $f_1(r)>f_1\left(-\frac{1}{16}-\frac{1}{16}\sqrt{33}\right)>0$ for all $r\in\left(-\frac{1}{2},-\frac{1}{16}-\frac{1}{16}\sqrt{33}\right],$ which is a contradiction.

 \textbf{Case IV:}
If possible $D_r$ be a generic matrix which represents $\C_{16}.$ Then, We compare the coefficients of $x^3, x^2, x, 1$ in the polynomials  $q(x)$ and $\mathcal{P}_{16}(x)$ and obtain
\begin{align*}
2 c_1-\frac{3}{2} &=-(1+a+2r),\\
-\frac{1}{2}+6c_1-16 c_1^2 &=a(1+2r)+\frac{1}{3}(1+2r)^2, \\
-26c_1+88c_1^2+\frac{5}{2}-96 c_1^3 &=-\left(\frac{1}{3}a(1+2r)^2+r^2+\frac{1}{3}(1-r)^2\right),\\
-72 c_1^2+96 c_1^3+18 c_1-\frac{3}{2} &=a \left(r^2+\frac{(1-r)^2}{3}\right),~\mathrm{where}~ \frac{1}{6}\leq c_1 \leq \frac{1}{3}.
\end{align*}
  These yield $c_1= \frac{1}{4}-\frac{1}{2}a-r$ where $-\frac{1}{6}-2r \leq a \leq \frac{1}{6}-2r.$ The second equation becomes $-4a^2-18ar-\frac{52}{3}r^2+\frac{2}{3}r-\frac{1}{3}=0$ which further implies either $a=-\frac{9}{4}r+\frac{1}{12}\sqrt{105r^2+24r-12}$ or $a=-\frac{9}{4}r-\frac{1}{12}\sqrt{105r^2+24r-12}.$

  When $a=-\frac{9}{4}r+\frac{1}{12}\sqrt{105r^2+24r-12},$ we have $-\frac{1}{16}-\frac{1}{48}\sqrt{393}\leq r \leq -\frac{4}{35}-\frac{2}{35}\sqrt{39}.$ Then the third equation becomes $(\frac{37}{36}r^2+\frac{1}{9}r-\frac{1}{18})\sqrt{105r^2+24r-12}-\frac{39}{4}r^3=0$, which is not possible since $\frac{37}{36}r^2+\frac{1}{9}r-\frac{1}{18}>0$ .

  If $a=-\frac{9}{4}r-\frac{1}{12}\sqrt{105r^2+24r-12},$ we have $-\frac{1}{2} < r \leq -\frac{4}{35}-\frac{2}{35}\sqrt{39}.$ Hence the third equation becomes $$f_2(r)=\left(-\frac{37}{36}r^2-\frac{1}{9}r+\frac{1}{18}\right)\sqrt{105r^2+24r-12}-\frac{39}{4}r^3=0.$$
  The function $\sqrt{105r^2+24r-12}$ is a strictly decreasing function of $r$ for  $r \in [-\frac{1}{2},-\frac{4}{35}-\frac{2}{35}\sqrt{39}],$ and hence $\frac{3}{2}>\sqrt{105r^2+24r-12}>0.$ So,
  $f_2(r)\geq g_2(r),$ where $$g_2(r)=\frac{1}{12}-\frac{1}{6}r-\frac{37}{24}r^2-\frac{39}{4}r^3.$$ It can be easily verified that $g_2'(r)<0,$ so that $g_2(r)$ is a strictly decreasing function. Hence $g_2(r)>g_2(-\frac{4}{35}-\frac{2}{35}\sqrt{39})>0$ implies that $f_2(r)>0$  for all $r\in [-\frac{1}{2},-\frac{4}{35}-\frac{2}{35}\sqrt{39}].$

  \textbf{Case V:}
 Let $D_r$  represents $\C_{20}.$ Then, by comparing the coefficients of $x^3, x^2, x, 1$ in the polynomials  $q(x)$ and $\mathcal{P}_{20}(x)$, we obtain
 \begin{align*}
 -\frac{5}{2}+6c_1 &=-(1+a+2r),\\
 -6c_1+\frac{3}{2} &= a(1+2r)+\frac{1}{3}(1+2r)^2,\\
 72c_1^2-96c_1^3-18c_1+\frac{3}{2} &=-\left(\frac{1}{3}a(1+2r)^2+r^2+\frac{1}{3}(1-r)^2\right),\\
 96c_1^3-72c_1^2+18c_1-\frac{3}{2} &=a \left(r^2+\frac{(1-r)^2}{3}\right),~ \mathrm{where},~ \frac{1}{6}\leq c_1 \leq \frac{1}{3}.\\
 \end{align*}
  Therefore $c_1= \frac{1}{4}-\frac{1}{6}a-\frac{1}{3}r,$  where $-\frac{1}{6}-2r  \leq a \leq \frac{1}{6}-2r.$ The second equation of the above system gives $a=-\frac{1}{6}\frac{(4r^2-2r+1)}{r},$ where $r~\in(-\frac{1}{2}, \frac{1}{16}-\frac{1}{16}\sqrt{33}].$ From the third equation we have
 \begin{align*}
\frac{1}{486}\frac{\left(80r^6+816r^5-420r^4+20r^3-15r^2+6r-1\right)}{r^3}=0.
 \end{align*}
  Let $f_3(r)=\frac{1}{486}\left(80r^6+816r^5-420r^4+20r^3-15r^2+6r-1\right).$ Then
 \begin{align*}
 f''''_3(r)&=\frac{80}{27}(20r^2+68r-7)\\
                     &=\frac{1600}{27}\left(r-\frac{1}{10}\right)\left(r+\frac{7}{2}\right) <0.
 \end{align*}
 Hence $f_3'''(r)=\frac{1600}{81}r^3+\frac{2720}{27}r^2-\frac{560}{27}r+\frac{20}{81}$ is a strictly decreasing function of $r,$ where $f_3'''\left(\frac{1}{16}-\frac{1}{16}\sqrt{33}\right)>0.$ So that, $f_3''(r)=\frac{400}{81}r^4+\frac{2720}{81}r^3-\frac{280}{27}r^2+\frac{20}{81}r-\frac{5}{81}$ is a strictly increasing function and $f_3''\left(\frac{1}{16}-\frac{1}{16}\sqrt{33}\right)<0.$ Hence, $f_3'(r)$ is a strictly decreasing function of $r,$
 where $f_3'\left(\frac{1}{16}-\frac{1}{16}\sqrt{33}\right)>0.$ This implies $f_3(r)$ is strictly increasing. Finally, $f_3\left(\frac{1}{16}-\frac{1}{16}\sqrt{33}\right)<0$ shows that $f_3(r)$ is negative for $r~\in(-\frac{1}{2}, \frac{1}{16}-\frac{1}{16}\sqrt{33}].$

 \textbf{Case VI:}
 Assign $D_r$ to be in $\C_{23}$ and comparing the coefficients of $x^3, x^2, x, 1$ in the polynomials  $q(x)$ and $\mathcal{P}_{23}(x)$, we obtain
 \begin{align*}
 -4 c_1 &=-(1+a+2r),\\
 -\frac{1}{2}+8c_1^2 &= a(1+2r)+\frac{1}{3}(1+2r)^2,\\
 96c_1^3-80c_1^2+22c_1-2 &=-\left(\frac{1}{3}a(1+2r)^2+r^2+\frac{1}{3}(1-r)^2\right),\\
 72c_1^2-96c_1^3-18c_1+\frac{3}{2} &=a \left(r^2+\frac{(1-r)^2}{3}\right),~\mathrm{where},~ \frac{1}{6}\leq c_1 \leq \frac{1}{3}.
 \end{align*}
 Thus $c_1=\frac{1}{4}+\frac{1}{4}a+\frac{1}{2}r,$ where $-\frac{1}{3}-2r  \leq a \leq \frac{1}{3}-2r.$ Further the second equation implies
 $a^2=\frac{2}{3}-\frac{4}{3}r-\frac{4}{3}r^2$ i.e. $a=\pm\frac{1}{3}\sqrt{-12r^2-12r+6}.$ Using this value of $a^2,$ the third equation gives
 \begin{align*}
 &\frac{4}{3}a-\frac{8}{3}ar+\frac{52}{3}ar^2+6r-12r^2=0. \\
 \mathrm{Hence,}~&a=\frac{9}{2}\frac{r(2r-1)}{(13r^2-2r+1)}>0~\mathrm{for}~r\in\left(\frac{1}{2},\frac{1-\sqrt3}{1+\sqrt3}\right).
 \end{align*}
  It is easy to see that, $\frac{1}{3}\sqrt{-12r^2-12r+6}-\frac{9}{2}\frac{r(2r-1)}{(13r^2-2r+1)}\neq0~\forall ~r\in\left(\frac{1}{2},\frac{1-\sqrt3}{1+\sqrt3}\right).$ Hence a contradiction.

 \textbf{Case VII:}
Let $D_r \in \C_{25}.$ Then, we compare the coefficients of $x^3, x^2$ in the polynomials  $q(x)$ and $\mathcal{P}_{25}(x)$ and obtain
  \begin{align*}
 -1 &=-(1+a+2r),\\
 -24c_1^2+12c_1-\frac{3}{2} &= a(1+2r)+\frac{1}{3}(1+2r)^2,\mathrm{where},~ \frac{1}{6}\leq c_1 \leq \frac{1}{3}.
 \end{align*}
 This gives $a=-2r,$ from the first equation. The second equation gives
 \begin{align*}
 -24c_1^2+12 c_1+\left(-\frac{11}{6}+\frac{2}{3}r+\frac{8}{3}r^2\right)=0,
 \end{align*}
 which has discriminant
 \begin{align*}
  &16(16r^2+4r-2)=256\left(r+\frac{1}{2}\right)\left(r-\frac{1}{4}\right)<0.
 \end{align*}
 So it is not satisfied by any real $c_1$ for $r\in(-\frac{1}{2},\frac{1-\sqrt3}{1+\sqrt3}).$

 From the above discussions it is clear that $D_r$ cannot belongs to $\bigcup_{i=1}^{37}\C_i.$ Hence the proof.$\hfill{\square}$

 We emphasize that $r \pm \frac{1-r}{\sqrt 3} i,-\frac{1}{2} < r < \frac{1-\sqrt3}{1+\sqrt3}$ are eigenvalues of the DS matrices given in (\ref{DS1}).

 \subsection{Eigenvalue region of PDS matrices of higher order }
In this section, we prove some inclusion theorems for the eigenvalue region of  PDS  matrices of order $n\geq 2.$ First we recall the following lemma. 
\begin{lemma} \label{l2}[Theorem 3, \cite{Perfect1965}]
If $\lambda$ is an eigenvalue,  different from $1$, of a column stochastic matrix, and $v=(v_1, v_2,..., v_n)^t$ is an eigenvector corresponding to $\lambda$, then $\displaystyle\sum_{i=1}^{n}{v_i}=0$.
\end{lemma}
The following theorem states  a geometric property of eigenvalue region of permutative doubly stochastic matrices. The proof follows from  the  Theorem 3 \cite{Perfect1965}.

\begin{theorem}
$\Lambda(\pds_n), n\geq 2$ is star shaped.
\end{theorem}
\begin{proof}
Let $D$ be a permutative doubly stochastic matrix, and $\lambda \neq 1$ is an eigenvalue of $D$ corresponding to the eigenvector $v$. Then for any choice of $P_1, P_2,..., P_{n-1} \in \S_n$, $\frac{1}{n}J=\mathcal{M}(\mathbf{1};P_1,\ldots,P_{n-1})$ is a PDS matrix, where $\mathbf{1}=(1,1,\ldots,1)$.  Now, using Lemma \ref{l2},  we have
\begin{equation*}
\left(t D + (1-t) \frac{1}{n} J\right)v
= tDv + (1-t) \frac{1}{n} Jv
= (t\lambda)v,~\mathrm{for}~ 0 \leq t \leq 1.
\end{equation*}
Hence $t\lambda$ is an eigenvalue of the matrix $\left(t D + (1-t) \frac{1}{n} J\right)$, which is a PDS matrix. This completes the proof.
\end{proof}

\begin{theorem}
 $[-1, 1] \subseteq \Lambda(\pds_n),$  $n \geq 2$.
 \end{theorem}
 \noindent
\begin{proof}
For $0 \leq a \leq 1$, define
$$
M_a =
\bmatrix{
a & 1-a \\
1-a & a
}.
$$
Then, it is easy to see that, the eigenvalues of $M_a$ are $1$ and $2a-1$ .

 If $n$ is an even positive integer then the matrix
 $$
\bmatrix{
M_a & 0 & \hdots & 0\\
0 & M_a & \hdots & 0\\
\vdots & \vdots & \ddots & \vdots \\
0 & 0 & \hdots & M_a
}
$$
 with first row $c = (a, 1-a, 0,...,0)$ is permutative doubly stochastic.
Hence, if $n$ is even, then $[-1, 1] \subset \Lambda(\pds_n).$ For $n=3,$ we have proved that $[-1,1] \subset\Lambda(\pds_3)$ corresponds to the matrix $A_x$ in Theorem \ref{eig_order3}.

Next, if $n$ is odd and $n \geq 5$ then we proceed as follows. Let
$$
\bmatrix{
M_a & 0 & \hdots & 0\\
0 & M_a & \hdots & 0\\
\vdots & \vdots & \ddots & \vdots \\
0 & 0 & \hdots & A_a
},
$$
where $A_a=\bmatrix{
   0 & a & 1-a \\
  a & 1-a & 0\\
  1-a & 0 & a
    }, $ for $0 \leq a \leq 1$ and $M_a$ repeats along the diagonal for $\frac{n-3}{2}$ times. Then the desired results follows as above. 
\end{proof}

Next theorem shows  a subset in the complex unit circular disc, which is contained in the eigenvalue region of  PDS matrices of  order $n.$
\begin{theorem} \label{subset_ineq}
$[-1 ,1] \cup \Pi_3 \cup \Pi_4 \cup\hdots \cup \Pi_{\left[\frac{n}{2}\right]} \cup \Pi_{\left[\frac{n+1}{2}\right]}^0 \cup \Pi_n \subseteq  \Lambda(\pds_n),$ $n\geq 3$.
\end{theorem}
\noindent
\begin{proof}
Let $D_n$ be an $n \times n$ permutative doubly stochastic matrix such that
\begin{eqnarray*} 
\mathcal{D}_n=
\bmatrix{
D_m & 0 \\
0 & D_{n-m}
}
\end{eqnarray*}
for $m = 2, 3,...,[\frac{n}{2}]$. Then, the eigenvalues of $\mathcal{D}_n$ is the union of the eigenvalues of $D_m$ and $D_{n-m}$. Hence, $\Lambda(\pds_m) \subseteq \Lambda(\pds_n)$. By varying $D_m$ over the  circulant doubly stochastic matrix of order $m$, we obtain $\Pi_m \subseteq \Lambda(\pds_m) \subseteq \Lambda(\pds_n)$ for $m = 2, 3,...,[\frac{n}{2}]$. Also, for $m = [\frac{n+1}{2}]$, by varying $D_m$ over the  trace zero circulant doubly stochastic matrix of order $m$, we get $\Pi_m^0 \subseteq \Lambda(\pds_n).$
\end{proof}
 Clearly the set  inequality in Theorem \ref{subset_ineq} is proper when $n \geq 4.$

\section{Conclusion}

We study the eigenvalue region of permutative doubly stochastic (PDS) matrices. First we characterize the set $\pds_n,$ which denotes the set of all symbolic PDS matrices of order $n,$ by partitioning it into different classes, called cogredient classes such that any two matrices in a class are permutationally similar. We determine all such classes when $n\leq 4.$ For $n=2,3,4,$ we show that there are $1,2,37$ cogredient classes respectively. We also derive the symbolic PDS matrices which represent these classes. We show that eigenvalue region of $\pds_n$ is same as the eigenvalue region of $\Omega_n$ for $n=2,3,$ where $\Omega_n$ denotes the set of all doubly stochastic matrices of order $n.$ The eigenvalue region of $\pds_n, n=4$ is obtained as the union of eigenvalue region of $\C_i,$ where $\C_i,$ $1\leq i\leq 37$ denote the cogredient classes. We provide analytical expression for the eigenvalues of all the classes except six classes for $n=4$ whose expressions are cumbersome; though it can be obtained by performing computations in a computer algebra system. We perform numerical examples for matrices in such classes to estimate the eigenvalue region.  We are able to show analytically that eigenvalue region of $\pds_4$ is a proper subset of eigenvalue region of $\Omega_4$ by identifying two line segments from the eigenvalue region of $\Omega_4$ that are not part of the eigenvalue region of $\pds_4.$

It is evident from the results derived in the article that eigenvalue region of the cogredient class $\C_1$ is same as $\Pi_4,$ the convex hull of the $4^{th}$ roots of unity. It has also been observed using numerical computations that the eigenvalue regions of $\C_{15},\C_{16},\C_{20},\C_{23},\C_{25}$ lie inside $\Pi_4,$ and the eigenvalue regions of $\C_3, \C_4$ contain parts of both $\Pi_3$ and $\Pi_4$. Thus we attempt to estimate the eigenvalue region of $\pds_4$ inside $\Pi_3$ but lies outside $\Pi_4.$

We conjecture that the boundary curves of the exact eigenvalue region of $\pds_4$ that lie inside $\Pi_3\setminus\Pi_4$ can be approximated by $\lambda(t), \overline{\lambda(t)}, \mu(s), \overline{\mu(s)}$ where $\lambda(t)$ and $\mu(s)$ are non-real eigenvalues of $$
A(t)=\bmatrix{t & 1-t & 0 & 0\\1-t & 0 & t & 0\\0 & 0 & 1-t & t\\0 & t & 0 & 1-t}~~\mathrm{and}~~
B(s)=\bmatrix{s & 0 & 1-s & 0\\0 & 0 & s & 1-s\\0 & 1-s & 0 & s\\1-s & s & 0 & 0}
$$ respectively, where  $0.82032\leq t \leq 1$  and for $0.76786 \leq s \leq 1.$ Finally we prove some inclusion theorem for eigenvalue region of $\pds_n.$


\section*{Acknowledgment} Amrita Mandal thanks Council for Scientific and Industrial Research(CSIR) for financial support in the form of a junior/senior research fellowship. M. Rajesh
Kannan thanks the Department of Science and Technology, India, for financial assistance through
the  MATRICS project  (MTR/2018/0009863) .

\section{Appendix}
\noindent Below we describe generic matrices which represent $\C_i,i=7,\ldots,37.$ First we write the generic matrices from which other generic matrices in type $\T_j,$ for $j\in \{2,3,4\}$ can be obtained by multiplying it with a permutation matrix from left. Here $\C_i \equiv \M(c;P,Q,R),$ means that $\C_i$ is the class corresponding to the generic matrix $\M(c;P,Q,R).$

First consider $\T_4.$ Then
$$\C_7\equiv
\bmatrix{
\frac{1}{4} & \frac{1}{4}&c_3 & c_4 \\
\frac{1}{4} & \frac{1}{4}&c_3 & c_4 \\
\frac{1}{4} & \frac{1}{4}&c_4 & c_3\\
\frac{1}{4} & \frac{1}{4}&c_4 & c_3 \\
},
\C_9\equiv
 \bmatrix
{\frac{1}{4} & \frac{1}{4}&c_3 & c_4 \\
\frac{1}{4} & \frac{1}{4}&c_4 & c_3 \\
\frac{1}{4} &c_3 & c_4 & \frac{1}{4}\\
\frac{1}{4} &c_4 & c_3& \frac{1}{4} \\
}
$$
and
 $$\C_8= (234)\C_7\equiv
 \bmatrix
{\frac{1}{4} & \frac{1}{4}&c_3 & c_4 \\
\frac{1}{4} & \frac{1}{4}&c_4 & c_3 \\
\frac{1}{4} & \frac{1}{4}&c_4 & c_3\\
\frac{1}{4} & \frac{1}{4}&c_3 & c_4 \\
} ,
\C_{10} = (12)\C_9\equiv
\bmatrix{
\frac{1}{4} & \frac{1}{4}&c_3 & c_4 \\
\frac{1}{4} & \frac{1}{4}&c_4 & c_3 \\
\frac{1}{4}&c_4 & c_3 & \frac{1}{4}\\
\frac{1}{4}&c_3 & c_4  & \frac{1}{4}\\
},
\C_{11}= (14)\C_9\equiv
\bmatrix{
\frac{1}{4}& \frac{1}{4} & c_3 & c_4 \\
\frac{1}{4} & c_4 & c_3 &\frac{1}{4} \\
\frac{1}{4} &\frac{1}{4}&c_4&c_3\\
\frac{1}{4} & c_3 & c_4 &\frac{1}{4} \\
} ,$$

 $$\C_{12}= (23)\C_9\equiv
\bmatrix{
\frac{1}{4} & \frac{1}{4}&c_3 & c_4 \\
\frac{1}{4} &c_3 & c_4 & \frac{1}{4} \\
\frac{1}{4} & \frac{1}{4}&c_4 & c_3\\
\frac{1}{4} &c_4 & c_3 & \frac{1}{4} \\
},
\C_{13}= (24)\C_9\equiv
\bmatrix{
\frac{1}{4} & \frac{1}{4} & c_3 & c_4 \\
\frac{1}{4} &c_4 & c_3 & \frac{1}{4} \\
\frac{1}{4} &  c_3 & c_4&\frac{1}{4}\\
\frac{1}{4} & \frac{1}{4} & c_4 & c_3\\
} ,
\C_{14}= (124)\C_9\equiv
\bmatrix{
\frac{1}{4} & \frac{1}{4}&c_3 & c_4 \\
\frac{1}{4} & c_3 & c_4 & \frac{1}{4} \\
\frac{1}{4} & c_4 & c_3& \frac{1}{4}\\
\frac{1}{4} & \frac{1}{4}& c_4 & c_3 \\
}.$$

Consider generic matrices from $\T_3$ type. Then
$$\C_{15}\equiv
\bmatrix{
c_1&c_2 & c_3 & c_4 \\
c_1 & c_2 & c_4&c_3\\
c_1 & c_4 &c_3& c_2 \\
c_2 &c_4 & c_3 & c_1 \\
}$$
and
 $$\C_{16}= (12)\C_{15}\equiv
\bmatrix{
c_1 &c_2 & c_3 & c_4 \\
c_1 &c_2 & c_4 & c_3\\
c_1 &c_3 & c_4 & c_2 \\
c_2 &c_3 & c_4 & c_1 \\
} ,
\C_{17}= (13)\C_{15}\equiv
\bmatrix{
c_1 &c_2 & c_3 & c_4 \\
c_1 & c_4 &c_2 & c_3\\
c_1 & c_4 &c_3 & c_2 \\
 c_4 &c_2 &c_3 & c_1 \\
},
\C_{18}= (14)\C_{15}\equiv
\bmatrix{
c_1  & c_2 & c_3 & c_4 \\
 c_4 & c_1 & c_2 & c_3\\
 c_4 & c_2 & c_3 & c_1 \\
 c_4 & c_1 & c_3 &c_2  \\
},$$

$$\C_{19}= (23)\C_{15}\equiv
\bmatrix{
c_1  & c_2 & c_3 & c_4 \\
 c_1 & c_4 & c_3 & c_2\\
 c_1 & c_2 & c_4 & c_3 \\
 c_2 & c_4 & c_3 &c_1 \\
},
\C_{20}= (34)\C_{15}\equiv
\bmatrix{
c_1  & c_2 & c_3 & c_4 \\
c_1  & c_2 & c_4 & c_3\\
c_2  & c_4 & c_3 & c_1 \\
c_1  & c_4 & c_3 & c_2 \\
},
\C_{21}= (12)(34)\C_{15}\equiv
\bmatrix{
c_1  & c_2 & c_3 & c_4 \\
 c_1 & c_2 & c_4 &c_3\\
  c_2 & c_3 &c_4 & c_1 \\
  c_1 & c_3 &c_4 & c_2 \\
},$$

 $$\C_{22}= (14)(23)\C_{15}\equiv
\bmatrix{
c_1  & c_2 & c_3 & c_4 \\
 c_4 & c_2 & c_3 & c_1 \\
  c_4 & c_1 & c_2 & c_3\\
 c_4 & c_1 & c_3 &c_2 \\
} ,
\C_{23}= (123)\C_{15}\equiv
\bmatrix{
c_1  & c_2 & c_3 & c_4 \\
c_1  & c_3 & c_4 & c_2\\
c_1  & c_2 & c_4 & c_3 \\
c_2  & c_3 & c_4 & c_1\\
} ,
\C_{24}= (132)\C_{15}\equiv
\bmatrix{
c_1  & c_2 & c_3 & c_4 \\
 c_1 & c_4 & c_3 & c_2\\
 c_1&c_4 & c_2 & c_3 \\
 c_4 & c_2 & c_3 &c_1 \\
} ,$$

 $$\C_{25}= (134)\C_{15}\equiv
\bmatrix{
c_1  & c_2 & c_3 & c_4 \\
 c_1 & c_4 & c_2 & c_3\\
 c_4 & c_2 & c_3 & c_1 \\
 c_1 & c_4 & c_3 &c_2 \\
} ,
\C_{26} = (143)\C_{15}\equiv
\bmatrix{
c_1  & c_2 & c_3 & c_4 \\
 c_4 & c_1 & c_2 & c_3\\
 c_4 & c_1 & c_3 &c_2 \\
 c_4 & c_2 & c_3 & c_1 \\
}.$$

Consider the matrix classes from $\T_2$ type. Then,
$$\C_{27}\equiv
\bmatrix{
c_1&c_2 & c_3 & c_4 \\
c_1&c_2 & c_3 & c_4 \\
c_2 &c_1 & c_4 & c_3\\
c_2 &c_1 & c_4 & c_3 \\
},
\C_{29}\equiv
\bmatrix{
c_1&c_2 & c_3 & c_4 \\
c_1&c_2 & c_4 & c_3 \\
c_2 &c_1 & c_3 & c_4\\
c_2 &c_1 & c_4 & c_3 \\
},$$

 $$\C_{31}\equiv
\bmatrix{
c_1&c_2 & c_3 & c_4 \\
c_1 &c_4 & c_2 & c_3\\
c_2&c_1 & c_4 & c_3 \\
c_2 &c_3 & c_1 & c_4 \\
},
\C_{34}\equiv
\bmatrix{
 c_1&c_2 & c_3 & c_4 \\
c_1 &c_3 & c_2 & c_4\\
c_2&c_1 & c_4 & c_3 \\
c_2 &c_4 & c_1 & c_3 \\
}$$
and

 $$\C_{28} = (13)\C_{27}\equiv
\bmatrix{
c_1&c_2 & c_3 & c_4 \\
c_2 &c_1 & c_4 & c_3\\
c_2 &c_1 & c_4 & c_3\\
c_1&c_2 & c_3 & c_4 \\
},
\C_{30}= (13)\C_{29}\equiv
\bmatrix{
c_1&c_2 & c_3 & c_4 \\
c_2 &c_1 & c_4 & c_3\\
c_2 &c_1 & c_3 & c_4 \\
c_1&c_2 & c_4 & c_3 \\
},$$

 $$\C_{32}= (12)\C_{31}\equiv
\bmatrix{
c_1&c_2 & c_3 & c_4 \\
c_1 &c_3 & c_4 & c_2\\
c_3&c_1 & c_2 & c_4 \\
c_3 &c_4 & c_1 & c_2 \\
},
\C_{33} = (13)\C_{31}\equiv
\bmatrix{
c_1&c_2 & c_3 & c_4 \\
c_2 &c_3 & c_1 & c_4\\
c_2&c_1 & c_4 & c_3 \\
c_1 &c_4 & c_2 & c_3 \\
},$$

$$\C_{35}= (12)\C_{34}\equiv
\bmatrix{
c_1&c_2 & c_3 & c_4 \\
c_1 &c_3 & c_2 & c_4\\
c_3&c_1 & c_4 & c_2 \\
c_3 &c_4 & c_1 & c_2 \\
} ,
\C_{36}= (13)\C_{34}\equiv
\bmatrix{
c_1&c_2 & c_3 & c_4 \\
c_2 &c_4 & c_1 & c_3\\
c_2&c_1 & c_4 & c_3 \\
c_1 &c_3 & c_2 & c_4 \\
} ,$$
and
$$\C_{37}= (34)\C_{34}\equiv
 \bmatrix{
c_1 & c_2 & c_3 & c_4 \\
c_1 & c_3 & c_2 & c_4\\
c_2 & c_4 & c_1 & c_3 \\
c_2 & c_1 & c_4 & c_3 \\
} .$$

\end{document}